\documentclass{amsart}
\usepackage[utf8]{inputenc}
\usepackage{amsmath}
\usepackage{mathtools}
\usepackage{amsfonts}
\usepackage{amssymb}
\usepackage{amsmath}
\usepackage{paralist}
\usepackage{amsthm}
\usepackage{amscd}
\usepackage{tikz}
\usepackage{graphicx}
\usepackage{caption}
\usepackage{comment}

\newtheorem{theorem}{Theorem}[section]
\newtheorem{corollary}[theorem]{Corollary}

\newtheorem{lemma}[theorem]{Lemma}
\newtheorem{proposition}[theorem]{Proposition}

\theoremstyle{definition}
\newtheorem{definition}[theorem]{Definition}
\newtheorem{remark}[theorem]{Remark}
\newtheorem{assumption}{Assumption}

\newtheorem{example}{Example}

\newcommand{\R}{\mathbb{R}}
\newcommand{\C}{\mathbb{C}}
\newcommand{\N}{\mathbb{N}}

\newcommand{\Z}{\mathbb{Z}}
\newcommand{\eps}{\varepsilon}

\makeatletter
\@namedef{subjclassname@2020}{%
  \textup{2020} Mathematics Subject Classification}
\makeatother

\begin{document}
\begin{abstract}
We prove Strichartz estimates for Maxwell equations in media in the fully anisotropic case with H\"older-continuous coefficients.
To this end, we use the FBI transform to conjugate the problem to phase space. After reducing to a scalar estimate by means of
a matrix symmetrizer, we show oscillatory integral estimates for a variable-coefficient Fourier extension operator. The
characteristic surface has conical singularities for any non-vanishing time frequency.
Combined with energy estimates, we improve the local well-posedness for certain fully anisotropic quasilinear Maxwell equations.
\end{abstract}
\title[Fully anisotropic Maxwell equations]{Strichartz estimates for Maxwell equations in media: the fully anisotropic case}
\author{Robert Schippa}
\email{robert.schippa@kit.edu}
\author{Roland Schnaubelt}
\email{schnaubelt@kit.edu}
\address{Department of Mathematics, Karlsruhe Institute of Technology, Englerstrasse 2, 76131 Karlsruhe, Germany}

\maketitle
\section{Introduction and main results}
Maxwell equations are the fundamental principles of electro-magnetism.
In media they relate \emph{displacement and electric field} $(\mathcal{D},\mathcal{E}) : \R \times \R^3 \to \R^{3} \times \R^{3}$
and \emph{magnetic and magnetizing field} $(\mathcal{B},\mathcal{H}): \R \times \R^3 \to \R^3 \times \R^3$	via
\begin{equation}
\label{eq:MaxwellEquationMedia}
\left\{ \begin{array}{rlrl}
\partial_t \mathcal{D}\!\!\!\! &= \nabla \times \mathcal{H} - \mathcal{J}_e,&  \quad\nabla \cdot \mathcal{D} \!\!\!\!&= \rho_e,
    \quad (t,x') \in \R \times \R^3, \\
\partial_t \mathcal{B}\!\! \!\! &= - \nabla \times \mathcal{E} - \mathcal{J}_m,& \quad \nabla \cdot \mathcal{B} \!\!\!\! &= \rho_m,
\end{array} \right.
\end{equation}
where $(\rho_e,\rho_m): \R \times \R^3 \to \R \times \R$ denote the \emph{electric and magnetic charges}
and $(\mathcal{J}_e,\mathcal{J}_m):\R \times \R^3 \to \R^3 \times \R^3$ the \emph{electric and magnetic current}.
(Space-time coordinates are written as $x=(t,x')=(x_0,x_1,\ldots,x_d) \in \R \times \R^d$ and
the dual variables in Fourier space as $\xi = (\tau,\xi')= (\xi_0,\xi_1,\ldots,\xi_d) \in \R \times \R^d$.)
 We refer to the physics' literature for further explanations \cite{FeynmanLeightonSands1964,LandauLifschitz1990}.
 We remark that magnetic charges and currents are hypothetical, but included in the analysis to highlight symmetry between
electric and magnetic field.

To obtain a closed system in \eqref{eq:MaxwellEquationMedia}, one has to link the fields via material laws. In the linear case,
we consider the pointwise laws
\begin{equation}
\label{eq:PointwiseMaterialLaws}
\mathcal{D}(x) = \varepsilon(x) \mathcal{E}(x), \quad \mathcal{B}(x) = \mu(x) \mathcal{H}(x)
\end{equation}
for the permittivity and permeability $\varepsilon, \mu \in C^{s}(\R \times \R^3;\R^{3 \times 3})$ with $0<s\leq 1$,
whose values are supposed to be symmetric and uniformly elliptic, i.e.,
\begin{equation}
\label{eq:Ellipticity}
\exists \, \lambda, \Lambda > 0: \, \forall v \in \R^3, \, \forall x \in \R^4: \;
   \quad \lambda |v|^2 \leq A^{ij}(x) v_i v_j \leq \Lambda |v|^2, \;\; A \in \{ \varepsilon, \mu \}.
\end{equation}
In the above display sum convention is in use.

In the previous work \cite{Schippa2021Maxwell3d}, the first author considered the partially anisotropic case with
$\varepsilon(x) = \text{diag}(\varepsilon_1(x),\varepsilon_2(x),\varepsilon_2(x))$ and $\mu(x) \equiv \mu_0$.
In \cite{Schippa2021Maxwell3d}, the Maxwell system was diagonalized to a system of half-wave equations, for which the sharp Strichartz range
for the three dimensional wave equation could be recovered in the isotropic case $\varepsilon_1(x) = \varepsilon_2(x)$. In the partially
anisotropic case, $\varepsilon_1(x) \neq \varepsilon_2(x)$, it remains an open question whether solutions to Maxwell equations satisfy
Strichartz estimates as for (half-)wave equations in the $C^2$-case. However, for Lipschitz coefficients
  the first author recovered  Strichartz estimates with derivative loss as for wave equations with Lipschitz coefficients
  (cf. \cite{Tataru2001}).

  In two space dimensions \cite{SchippaSchnaubelt2022} we established sharp Strichartz estimates in the fully anisotropic case
\begin{equation*}
 \varepsilon(x)=
 \begin{pmatrix}
  \varepsilon^{11}(x) & \varepsilon^{12}(x) \\
  \varepsilon^{21}(x) & \varepsilon^{22}(x)
 \end{pmatrix},
\quad \mu(x) = \mu_0(x)
\end{equation*}
via diagonalization without imposing additional assumptions on $\varepsilon$. In three dimensions the fully anisotropic case (see \eqref{eq:FullyAnisotropicCondition}), Strichartz estimates have not been studied before. In this case the diagonalization
procedure introduces singularities in the conjugation matrices, and we opt for a different approach.
The fully aniso\-tro\-pic case appears in the study of biaxial crystals  (cf. \cite{FladtBaur1975, Knoerrer1986}). Moreover, when
treating nonlinear material laws $(\varepsilon(\mathcal{E}), \mu(\mathcal{H}))$ one needs Strichatz estimates for differentiated
fields, see Section~\ref{section:QuasilinearEquations}. However, these fields only fullfil a Maxwell system with modified coefficients,
which are matrix-valued even if the original  $(\varepsilon(\mathcal{E}), \mu(\mathcal{H}))$ are scalar and thus isotropic.

We work under the
following assumption on $\varepsilon$ and $\mu$, which can be described as uniform anisotropic, possibly off-diagonal material law.
\begin{assumption}\label{AssumptionMaterialLaws}
 Let $\varepsilon,\mu: \R \times \R^3 \to \R_{\text{sym}}^{3 \times 3}$ satisfy \eqref{eq:Ellipticity} and suppose that there is
 $\Phi \in C^1(\R \times \R^3; \R^{3 \times 3})$ with\footnote{In $\Phi$ we deviate from the usual matrix index notation
 and consider $\varphi_i$ as column vectors.}
\begin{equation*}
\Phi = \begin{pmatrix}
\varphi_1 & \varphi_2 & \varphi_3
\end{pmatrix}
=
\begin{pmatrix}
\varphi_{11} & \varphi_{21} & \varphi_{31} \\
\varphi_{12} & \varphi_{22} & \varphi_{32} \\
\varphi_{13} & \varphi_{23} & \varphi_{33}
\end{pmatrix}
, \quad \Phi^t(x) \Phi(x) = 1_{3 \times 3},
\end{equation*} 
and $\varepsilon^d,\mu^d \in C(\R \times \R^3;\R^{3 \times 3})$ with
\begin{equation*}
\varepsilon^d(x) = \text{diag}(\varepsilon^1(x),\varepsilon^2(x),\varepsilon^3(x)) \text{ \ and \ }
    \mu^d(x) = \text{diag}(\mu^1(x),\mu^2(x),\mu^3(x))
\end{equation*}
such that
\begin{equation*}
\varepsilon^d = \Phi^t \varepsilon \Phi, \quad \mu^d = \Phi^t \mu \Phi,
\end{equation*}
and $\varepsilon^d$, $\mu^d$ satisfy
\begin{equation*}
 \exists c > 0: \, \forall x \in \R^4: \, \forall i \neq j: \, \big| \frac{\varepsilon^i(x)}{\mu^i(x)} - \frac{\varepsilon^j(x)}{\mu^j(x)} \big| \geq c.
\end{equation*}
\end{assumption}
We let
\begin{equation}
\label{eq:MaxwellConcise}
P(x,D) = 
\begin{pmatrix}
- \partial_t (\varepsilon \cdot) & \nabla \times \\
\nabla \times & \partial_t (\mu \cdot)
\end{pmatrix}
\end{equation}
such that \eqref{eq:MaxwellEquationMedia} is concisely written as
\begin{equation*}
P(x,D) \begin{pmatrix}
\mathcal{E} \\ \mathcal{H}
\end{pmatrix}
=
\begin{pmatrix}
\mathcal{J}_e \\ - \mathcal{J}_m
\end{pmatrix}, \quad \nabla \cdot( \varepsilon \mathcal{E}) = \rho_e , \; \nabla \cdot (\mu \mathcal{H}) = \rho_m.
\end{equation*}
We write $u = (\mathcal{E},\mathcal{H})$ and $\rho_{em}= (\rho_e,\rho_m)$.

As noted above in the isotropic and to some extent in the partially anisotropic case, by \cite{Schippa2022}
the Maxwell system possesses Strichartz estimates as the scalar wave equation. However,
in the fully anisotropic case already for constant coefficients $\varepsilon= \text{diag}(\varepsilon_1,\varepsilon_2,\varepsilon_3)$,
$\mu = \text{diag}(\mu_1,\mu_2,\mu_3)$ satisfying
\begin{equation}
\label{eq:FullyAnisotropicCondition}
\frac{\varepsilon_1}{\mu_1} \neq \frac{\varepsilon_2}{\mu_2} \neq \frac{\varepsilon_3}{\mu_3} \neq \frac{\varepsilon_1}{\mu_1},
\end{equation}
Liess \cite{Liess1991} showed that the dispersive properties in the charge-free case
$\nabla \cdot \mathcal{D} = \nabla \cdot \mathcal{B} = 0$ are only the same as for wave equations in two dimensions, namely
\begin{equation}\label{eq:liess}
 \| S_1'( \mathcal{E},\mathcal{H})(t) \|_{L^\infty(\R^3)} \lesssim (1+|t|)^{-\frac{1}{2}} \| (\mathcal{E},\mathcal{H})(0) \|_{L^1(\R^3)}.
\end{equation}
Here we use  a homogeneous dyadic decomposition $(S'_\lambda)_{\lambda \in 2^{\mathbb{Z}}}$ of spatial frequencies, and
$(S_\lambda)_{\lambda \in 2^{\mathbb{Z}}}$ denotes a decomposition of space-time frequencies, see \eqref{eq:Slambda}.

This surprising behavior is connected to the shape of the characteristic surface depending on the eigenvalues of $\varepsilon$ and
$\mu$, which has been discussed in \cite[Section~2.3]{Schippa2022}. In \cite{MandelSchippa2022} the first author and R.~Mandel
 proved the existence of solutions to the time-harmonic Maxwell equations in the fully anisotropic case under the assumption that $\varepsilon$ and $\mu$ are commuting. To the best of the authors' knowledge, this was the first time the condition \eqref{eq:FullyAnisotropicCondition} appeared explicitly in the literature. In this case the characteristic surface of Maxwell equations ceases to be regular as the union of two ellipsoids (as in the (partially) anisotropic case). Instead it  becomes the singular Fresnel
 surface, see \cite{MandelSchippa2022} for a detailed quantitative analysis.
The properties of the Fresnel surface are recalled in Section \ref{section:DiagonalProof}. We further note that the characteristic surface can have more singularities if $\varepsilon$ and $\mu$ do not commute,  \cite{FavaroHehl2016}.

The decreased dispersive effect for the Fresnel surface forces us to consider the wave admissibility condition in two dimensions:
\begin{equation}
\label{eq:Admissibility}
\frac{2}{p} + \frac{1}{q} \leq \frac{1}{2}, \quad 2 \leq p,q \leq \infty.
\end{equation}
By \eqref{eq:liess} and the interpolation argument of Keel--Tao \cite{KeelTao1998}, for $(p,q)$ satisfying \eqref{eq:Admissibility}
the solutions  $(\mathcal{E},\mathcal{H})$ to homogeneous Maxwell equations with constant coefficients fulfill the Strichartz estimate
\begin{equation*}
 \| |D'|^{-\rho} (\mathcal{E},\mathcal{H}) \|_{L_t^p(\R,L_{x'}^q(\R^3))} \lesssim \| (\mathcal{E},\mathcal{H})(0) \|_{L^2(\R^3)}
\end{equation*}
 in the charge-free case.
We give the details and local-in-time estimates in the case of non-vanishing charges in Section \ref{section:ConstantCoefficientEstimates}.
In the following we denote space-time Lebesgue spaces by $L_t^p L_{x'}^q = L^p L^q$ for the sake of brevity, where $L^p=L^pL^p$.
In the above display the derivative loss is denoted by $\rho$ which satisfies the scaling condition
\begin{equation}
\label{eq:DerivativeLoss}
\rho = 3 \big( \frac{1}{2} - \frac{1}{q} \big) - \frac{1}{p}.
\end{equation}
Furthermore, we use fractional derivatives given as Fourier multipliers by
\begin{equation*}
(|D|^\alpha f) \widehat (\xi) = |\xi|^\alpha \hat{f}(\xi), \; (|D'|^\alpha f) \widehat (\xi) = |\xi'|^\alpha \hat{f}(\xi), \; (\langle D \rangle^\alpha f) \widehat (\xi) = \langle \xi \rangle^\alpha \hat{f}(\xi), \quad \text{etc.},
\end{equation*}
where $\langle x \rangle = (1+|x|^2)^{1/2}$.

We first show the following Strichartz estimates for rough coefficients in the diagonal case.
\begin{theorem}
\label{thm:StrichartzEstimatesFullyAnisotropic}
Let $T>0$ and $\delta > 0$, and suppose that $\varepsilon,\mu \in C^s(\R \times \R^3; \R_{\text{sym}}^{3 \times 3})$ with
$0<s\leq 1$ and $\| \partial (\varepsilon,\mu) \|_{L_T^2 L_{x'}^\infty} < \infty$ satisfy Assumption \ref{AssumptionMaterialLaws}
with $\Phi \equiv 1$. Let  $P$ be as in \eqref{eq:MaxwellConcise}. Then the estimate
\begin{align}
\label{eq:StrichartzEstimatesRoughCoefficients}
\| \langle D' \rangle^{-\rho + \frac{s-2}{2p} - \delta} u \|_{L^p(0,T; L^q(\R^3))}
    &\lesssim \| u_0 \|_{L^2_{x'}} + \| P u \|_{L^1_T L_{x'}^2} + \| \langle D' \rangle^{-\frac{1}{2}+\frac{s-2}{8}} \rho_{em}(0) \|_{L^2_{x'}} \notag\\
&\quad + \| \langle D' \rangle^{-\frac{1}{2}+\frac{s-2}{8}} \partial_t \rho_{em} \|_{ L^1_T L^2_{x'}}
\end{align}
holds provided that $p,q$ satisfy \eqref{eq:Admissibility}, $q \neq 2$, and $\rho$ is given by \eqref{eq:DerivativeLoss}. The implicit
constant depends on $\delta$, $T$,  $\| (\varepsilon,\mu) \|_{C^s}$, and $\| \partial (\varepsilon,\mu) \|_{L_T^2 L_{x'}^\infty}$.
\end{theorem}
We stress that both the different admissibility conditions and the appearance of the charges highlight the effect of the system
character of the Maxwell equations, compared to the case of scalar wave  equations treated in \cite{Tataru2000,Tataru2001, Tataru2002}.
%

To establish \eqref{eq:StrichartzEstimatesRoughCoefficients}, we first show
\begin{equation}
 \label{eq:StrichartzEstimatePhaseSpaceAnalysis}
 \| |D|^{-\rho+\frac{s-2}{4}} u \|_{L^pL^q}\lesssim_{\|(\varepsilon,\mu)\|_{C^s}} \| u \|_{L^2} + \| |D|^{\frac{s-2}{2}} Pu \|_{L^{2}}
   + \| |D|^{-\frac{1}{2} + \frac{s-2}{4}} \rho_{em} \|_{L^2}
\end{equation}
(with a modification at the endpoint $q=\infty$) if also  $\frac{2}{p} + \frac{1}{q} = \frac{1}{2}$ and $q>p$  via phase space analysis,
see Theorem~\ref{thm:LocalStrichartzEstimatesFullyAnisotropic}. The core step in the argument is an estimate of a Fourier extension
operator on the Fresnel surface $S$, where we can reduce to scalar problem using a symmetrizer. The singularities of $S$ are handled
by a dyadic scaling around them on annuli in Fourier space. Here we also use ideas from \cite{MandelSchippa2022} in the time-harmonic
case with constant coefficients. The degeneracy at $|\xi'|\gg |\xi_0|$ of the main symbol of the Maxwell system is treated  using the
charges in a separate microlocal argument.

We note that the derivative loss $\rho+\frac{2-s}{4}$ in \eqref{eq:StrichartzEstimatePhaseSpaceAnalysis} is the same as in \cite{Tataru2000}
for scalar wave equations. We actually improve the loss in \eqref{eq:StrichartzEstimatesRoughCoefficients}, but so far we cannot reach the
sharp regularity loss established in
\cite{Tataru2001, Tataru2002} for the wave equation or \cite{SchippaSchnaubelt2022,Schippa2021Maxwell3d} in the 2D or isotropic Maxwell case.
In contrast to these works up to now we cannot use deeper properties of the Hamilton flow of the problem because of the strong system
character in the fully isotropic case. In these papers also the case $s\in[1,2]$ and different norms on the right-hand side of
\eqref{eq:StrichartzEstimatesRoughCoefficients} have been treated. We plan to tackle these issues in future work.

Starting from \eqref{eq:StrichartzEstimatePhaseSpaceAnalysis} we reduce the regularity loss by passing through short-time Strichartz estimates. In Proposition~\ref{prop:ShorttimeStrichartz} we  show the endpoint estimate for $p=4$ and $q=\infty$ on finite-time intervals
\begin{align*}
 \| \langle D' \rangle^{-\rho + \frac{s-2}{8} - \delta} u \|_{L_T^4L^\infty_{x'}}
 &\lesssim \| u_0 \|_{L^2_{x'}} + \| Pu \|_{L^1_T L^2_{x'}} + \| \langle D' \rangle^{\frac{s-2}{8}-\frac{1}{2}} \rho_{em}(0) \|_{L^2_{x'}} \\
 &\quad + \| \langle D' \rangle^{\frac{s-2}{8}-\frac{1}{2}} \partial_t \rho_{em} \|_{L^1_T L^2_{x'}}
 \end{align*}
with implicit constant depending on $\|(\varepsilon,\mu) \|_{C^1}$, $T$ and $\delta$, where we set $L^p_T L^q_{x'}=L^p(0,T; L^q_{x'})$.
The different form of the charge term stems from an argument involving Duhamel's formula, see \eqref{eq:duhamel} and also
\cite{Schippa2021Maxwell3d,SchippaSchnaubelt2022}. Theorem~\ref{thm:StrichartzEstimatesFullyAnisotropic} then follows by interpolating
the above display with the standard energy estimate
\begin{equation*}
 \| u \|_{L^\infty_TL^2_{x'}} \lesssim_{\|(\varepsilon,\mu)\|_{C^1}, T} \| u_0 \|_{L^2} + \| P u \|_{L^1_T L_{x'}^2}.
\end{equation*}
We note that  \eqref{eq:StrichartzEstimatePhaseSpaceAnalysis} is shown for $q>6$ and that our method of proof breaks down for $(p,q)$
closer to the energy point $(p,q) = (\infty,2)$.
Broadly speaking, our arguments show that we can prove Strichartz estimates for characteristic surfaces $S= \{\tilde{q}(\xi_0) =  0\}$
with isolated degeneracies $\tilde{q}(\xi_0) = 0$, $\nabla \tilde{q}(\xi_0) = 0$, $\det ( \partial^2 \tilde{q}(\xi_0)) \neq 0$.

Second, we establish the following variant in the possibly non-diagonal case.
\begin{theorem}
\label{thm:StrichartzEstimatesFullyAnisotropicOffDiagonal}
Let $T>0$, $\delta > 0$, and suppose that $\varepsilon,\mu \in C^1(\R \times \R^3; \R_{\text{sym}}^{3 \times 3})$ satisfy
Assumption \ref{AssumptionMaterialLaws}. Let  $P$ be as in \eqref{eq:MaxwellConcise}. Then the estimate
\begin{align*}
\|\langle D'\rangle^{-\rho -\frac{1}{2p}-\delta} u \|_{L^p(0,T; L^q(\R^3))} &\lesssim \| u_0 \|_{L^2} + \| P u \|_{L^1_T L_{x'}^2}\notag \\
&\quad + \| \langle D' \rangle^{-5/8} \rho_{em}(0) \|_{L^2_{x'}} + \| \langle D' \rangle^{-5/8} \partial_t \rho_{em} \|_{ L^1_T L^2_{x'}}
\end{align*}
holds provided that $(p,q)$ satisfy \eqref{eq:Admissibility}, $q \neq 2$, and $\rho$ is given by \eqref{eq:DerivativeLoss}. The implicit constant depends on $\delta$, $T$, and $\| (\varepsilon,\mu) \|_{C^1}$.
\end{theorem}

To show this  theorem, one uses the orthogonality transformation $\Phi$ to pass to a Maxwell-type system with the diagonal
coefficents $\varepsilon^d$ and $\mu^d$ from  Assumption \ref{AssumptionMaterialLaws}, but with modified differential operators
in $x'$ depending on $\Phi$. This system can be treated similar to the system arising in the proof of Theorem~\ref{thm:StrichartzEstimatesFullyAnisotropic}.

We believe that the global existence and regularity of eigenvectors as stated in Assumption \ref{AssumptionMaterialLaws} is non-trivial. Hence, we devote Section \ref{section:RegularityEigenvectors} to a discussion
of existence and regularity of eigenvectors for parameter-dependent matrices. It turns out that if the eigenvalues are uniformly separated and the parameter domain is simply connected, the eigenvectors are as regular as the eigenvalues and admit global parametrisations.
The results of Section \ref{section:RegularityEigenvectors} imply the following proposition, which leads to a corollary to Theorem \ref{thm:StrichartzEstimatesFullyAnisotropicOffDiagonal} stated below.

\begin{proposition}
\label{prop:RegularDiagonalizationPermittivity}
Suppose that $\mu \equiv 1$ and $\varepsilon \in C^1(\R\times\R^3; \R^{3 \times 3}_{\text{sym}})$ satisfies \eqref{eq:Ellipticity} and
\begin{equation*}
\exists c > 0: \, \forall x \in \R^4: \, |\varepsilon_i(x) - \varepsilon_j(x)| \geq c > 0, \quad i \neq j,
\end{equation*}
where $(\varepsilon_i(x))_{i=1,2,3}$ are the eigenvalues of $\varepsilon(x)$. Then $(\varepsilon,\mu)$ fulfill Assumption \ref{AssumptionMaterialLaws}.
\end{proposition}
\begin{corollary}
Let $\varepsilon$ and $\mu$ be like in Proposition \ref{prop:RegularDiagonalizationPermittivity}. Then the Strichartz estimates \eqref{eq:StrichartzEstimatesRoughCoefficients} hold true under the assumptions of Theorem \ref{thm:StrichartzEstimatesFullyAnisotropicOffDiagonal}.
\end{corollary}

\smallskip

In Section \ref{section:QuasilinearEquations} we apply the Strichartz estimates from Theorem~\ref{thm:StrichartzEstimatesFullyAnisotropic}
to improve the local well-posedness theory of quasilinear Maxwell equations
\begin{equation}
 \label{eq:QuasilinearMaxwellIntroduction}
 \left\{ \begin{array}{rlrlrl}
  \partial_t \mathcal{D}\!\!\!\! &= \nabla \times \mathcal{H},& \quad \nabla \cdot \mathcal{D} \!\!\!\! &= 0,&
    \quad \mathcal{D}(0) \!\!\!\! &= \mathcal{D}_0 \in H^s(\R^3;\R^3), \\
  \partial_t \mathcal{B}\!\!\!\! &= - \nabla \times \mathcal{E},& \quad \nabla \cdot \mathcal{B} \!\!\!\! &= 0, &
     \quad \mathcal{B}(0)\!\!\!\! & = \mathcal{B}_0 \in H^s(\R^3;\R^3),
 \end{array} \right.
\end{equation}
with small initial fields and such that $\varepsilon(\mathcal{E})$ has uniformly separated eigenvalues. For sake of simplicity,
we suppose that $\mu = 1_{3 \times 3}$. Contrary to the isotropic Kerr case analyzed in \cite{SchippaSchnaubelt2022,Schippa2021Maxwell3d},
we cannot automatically deduce energy estimates by symmetrization. It turns out that this requires additional symmetries of the permittivity.
We shall rewrite $\mathcal{D} = \varepsilon(\mathcal{E}) \mathcal{E}$
to $\mathcal{E} = \psi(\mathcal{D}) \mathcal{D}$, which is possible for small fields under mild assumptions on $\varepsilon$ by invoking the implicit function theorem. In this form, we can phrase the condition for the existence of a symmetrizer as
\begin{equation}
\label{eq:SymmetryCoefficientsIntroduction}
 \varepsilon^{ijk} \big( \frac{\partial}{\partial \mathcal{D}_j} \psi(\mathcal{D})_{k \ell} \big) \mathcal{D}^{\ell} = 0
\end{equation}
for $i \in \{1,2,3\}$ and summation over $j,k \in \{1,2,3\}$, where $\varepsilon^{ijk}$ denotes the Levi-Civita symbol.
(Such a condition was also used in \cite{LPS}.)
We refer to the paragraph before Proposition \ref{prop:APrioriQuasilinear} for discussion of examples like
\begin{equation*}
\varepsilon = \text{diag}(\varepsilon_0^1, \varepsilon_0^2, \varepsilon_0^3) + \text{diag} (\alpha_1 |\mathcal{E}_1|^2, \alpha_2 |\mathcal{E}_2|^2, \alpha_3 |\mathcal{E}_3|^2 ).
\end{equation*} 
  It seems plausible that permittivities obtained like this can model biaxial crystals with nonlinear electric response.

We show the following theorem, which improves the local well-posedness via Strichartz estimates by $1/9$ derivatives compared to
energy arguments. It is the first result of this kind for cases of fully anisotropic Maxwell systems.
This improvement is smaller as in \cite{Tataru2002} or \cite{Schippa2021Maxwell3d,SchippaSchnaubelt2022}.
(For instance, for scalar quasilinear wave equations one gains $1/3$ derivatives compared to energy arguments by the results in \cite{Tataru2002}.)
Strichartz estimates with smaller regularity loss would imply better results, but it is unclear 
what can be achieved in the fully anisotropic Maxwell case.
\begin{theorem}
 \label{thm:QuasilinearAnisotropicMaxwellEquation}
Let $\varepsilon_i \in C^\infty(\R;\R)$ and let  $\varepsilon(\mathcal{E})= \mathrm{diag}(\varepsilon_1(\mathcal{E}_1), \varepsilon_2(\mathcal{E}_2), \varepsilon_3(\mathcal{E}_3))$ be uniformly positive definite with uniformly separated eigenvalues, i.e.,
there are $c,\delta>0$ such that for any $\mathcal{E} \in \R^3$ with $|\mathcal{E}| \leq \delta$ we have
\begin{equation*}
|\varepsilon_i (\mathcal{E}) - \varepsilon_j(\mathcal{E})| \geq c > 0.
\end{equation*}
Write $\mathcal{D} = \varepsilon(\mathcal{E}) \mathcal{E}$ as $\mathcal{E} = \psi(\mathcal{D}) \mathcal{D}$ for $\| \mathcal{D} \|_{L^\infty} \leq \delta$ and assume that $\psi$ satisfies \eqref{eq:SymmetryCoefficientsIntroduction}. Then there is $\delta' > 0$ such that \eqref{eq:QuasilinearMaxwellIntroduction} is locally well-posed for $s>2 + \frac{7}{18}$ for initial data $\| u_0 \|_{H^s} \leq \delta'$.
\end{theorem}

For the proof we use Strichartz estimates to improve on Sobolev embedding. Under the above symmetry assumptions, we find energy estimates
\begin{equation*}
\| (\mathcal{E},\mathcal{H})(t) \|_{H^s} \lesssim e^{ c(A) \int_0^t \| \partial (\mathcal{E},\mathcal{H})(s) \|_{L_{x'}^\infty} ds} \| (\mathcal{E},\mathcal{H})(0) \|_{H^s}
\end{equation*}
with $A = \| (\mathcal{E},\mathcal{H}) \|_{L_T^\infty L_{x'}^\infty}$. An estimate of $\| \partial (\mathcal{E},\mathcal{H}) \|_{L_T^1 L_{x'}^\infty}$ based on Sobolev embedding would lead to a regularity of $s>5/2$. We want to use non-trivial Strichartz estimates for $\| \partial( \mathcal{E},\mathcal{H}) \|_{L_T^4 L_{x'}^\infty}$ provided that the coefficients are in $C^\alpha_x$. The H\"older norm is again controlled by a spatial Sobolev norm and we can close the argument for $s>3/2 + \alpha$, if the derivative loss in the Strichartz estimates \emph{for $ \partial (\mathcal{E},\mathcal{H}) $} is controlled in the regularity $H^s$. This leads to the condition $s > \frac{3}{2} + \frac{8}{9}$. In the next step, we control the $L^2$-norm of differences of solutions $v = (\mathcal{E}^1,\mathcal{H}^1) - (\mathcal{E}^2,\mathcal{H}^2)$ by
\begin{equation*}
\| v(t) \|_{L^2} \lesssim e^{c(A) \int_0^t B(s) ds} \| v(0) \|_{L^2}
\end{equation*}
with $A = \| (\mathcal{E}^1, \mathcal{H}^1) \|_{L^\infty_T L^\infty_{x'}} + \| (\mathcal{E}^2,\mathcal{H}^2) \|_{L^\infty_T L_{x'}^\infty}$ and $B(s) = \| (\mathcal{E}^1, \mathcal{H}^1)(s) \|_{L^\infty_{x'}} + \| (\mathcal{E}^2,\mathcal{H}^2)(s) \|_{L_{x'}^\infty}$. By the same argument as above, we obtain Lipschitz continuous dependence in $L^2$ for initial data in regularity $H^s$ if $s>\frac{3}{2} + \frac{8}{9}$. To find continuous dependence to hold in $H^s$, we use the frequency envelope approach due to Tao \cite{Tao2001}. (See also \cite{IfrimTataru2020} for an exposition, and \cite{SchippaSchnaubelt2022} for a previous application in the context of Maxwell equations.) This does not yield uniform continuous dependence, which cannot be expected for a quasilinear hyperbolic problem.

\smallskip

\emph{Outline of the paper.} In Section \ref{section:ConstantCoefficientEstimates} we revisit Strichartz estimates in the constant-coefficient case for fully anisotropic Maxwell equations.
In Section \ref{section:DiagonalProof} we prove Strichartz estimates for diagonal variable coefficients in the uniformly fully aniso\-tro\-pic case as stated in Theorem \ref{thm:StrichartzEstimatesFullyAnisotropic}.
In Section \ref{section:Reductions} we reduce the more general case of Assumption \ref{AssumptionMaterialLaws} to the case of diagonal permittivity and permeability and prove Theorem \ref{thm:StrichartzEstimatesFullyAnisotropicOffDiagonal}.
In Section \ref{section:QuasilinearEquations} we argue how the Strichartz estimates for rough coefficients improve the local well-posedness theory for quasilinear Maxwell equations. In Section
\ref{section:RegularityEigenvectors} we show the global existence and regularity of eigenvectors in case of separated eigenvalues of parameter-dependent matrices in simply connected domains. 

\section{Strichartz estimates in the constant-coefficient case}
\label{section:ConstantCoefficientEstimates}
In this section we revisit the Strichartz estimates in the fully anisotropic case for constant coefficients. For the remainder of the
section, let
\begin{equation}
 \label{eq:DiagonalConstantCoefficients}
 \varepsilon = \text{diag}(\varepsilon_1,\varepsilon_2,\varepsilon_3), \quad \mu = \text{diag}(\mu_1,\mu_2,\mu_3)
 \quad \text{with \ } \varepsilon_i, \ \mu_j \in \R_{>0}.
\end{equation}
 By the symmetries of Maxwell equations (cf.\ \cite{MandelSchippa2022}), we can reduce the
general case of positive definite $\varepsilon,\mu \in \R^{3 \times 3}_{\text{sym}}$ to \eqref{eq:DiagonalConstantCoefficients} provided
that $\varepsilon$ and $\mu$ commute. We supplement the linear Maxwell equations
\begin{equation}
 \label{eq:ConstantCoefficientMaxwell}
 \left\{ \begin{array}{rlrl}
  \partial_t \mathcal{D} \!\!\!\!&= \nabla \times \mathcal{H},& \quad \nabla \cdot \mathcal{D} \!\!\!\!&= \rho_e, \\
  \partial_t \mathcal{B}  \!\!\!\!&= - \nabla \times \mathcal{E},& \quad \nabla \cdot \mathcal{B}  \!\!\!\!&= \rho_m
 \end{array} \right.
\end{equation}
on $\R\times \R^3$ with the constant-coefficient material laws
\begin{equation}
 \label{eq:ConstantCoefficientMaterialLaws}
 \mathcal{D} = \varepsilon \mathcal{E}, \quad \mathcal{B} = \mu \mathcal{H}, \quad \text{where \eqref{eq:DiagonalConstantCoefficients}
 is true.}
\end{equation}

In the charge-free case, dispersive time-decay  of solutions to \eqref{eq:ConstantCoefficientMaxwell} was proved by Liess \cite{Liess1991}.
Here we deduce Strichartz estimates also allowing for charges. The first observation is that non-trivial charges in
\eqref{eq:ConstantCoefficientMaxwell} inhibit global-in-time Strichartz estimates
\begin{equation*}
 \| |D'|^{-\rho} u \|_{L^p(\R;L^q_{x'})} \lesssim \| u(0) \|_{L^2_{x'}}
\end{equation*}
with $u = (\mathcal{E},\mathcal{H})$, $\rho = 3 \big( \frac{1}{2} - \frac{1}{q} \big) - \frac{1}{p}$, and
$\frac{2}{p} + \frac{1}{q} \leq \frac{1}{2}$.

Indeed, given nonzero $(\rho_e,\rho_m)\in \dot{H}^{\gamma} \times \dot{H}^{\gamma}$, let $(\Phi_1,\Phi_2)$ solve
the elliptic equations
\begin{equation}
\label{eq:EllipticEquationCharges}
\nabla \cdot (\varepsilon \nabla \Phi_1) = \rho_e, \qquad \nabla \cdot (\mu \nabla \Phi_2) = \rho_m.
\end{equation}
Then \eqref{eq:ConstantCoefficientMaxwell} possesses the stationary solution $(\mathcal{E},\mathcal{H})(0) = u$ with $\mathcal{E}(0) = \nabla \Phi_1$ and $\mathcal{H}(0) = \nabla \Phi_2$. We have $(\Phi_1,\Phi_2) \in \dot{H}^{\gamma+2}$ and $u(0) \in \dot{H}^{\gamma+1}$. However, the solution does not decay, and hence global-in-time Strichartz estimates for \eqref{eq:ConstantCoefficientMaxwell} are not possible.
To obtain such estimates in the charge-free case, we use the  following dispersive estimate due to Liess \cite[Theorem~1.3]{Liess1991}.

\begin{proposition}
 Let $u=(\mathcal{E},\mathcal{H})$ solve \eqref{eq:ConstantCoefficientMaxwell} with \eqref{eq:ConstantCoefficientMaterialLaws}
and $\rho_e = \rho_m = 0$. We then have
 \begin{equation}
  \label{eq:DispersiveEstimate}
  \| S_1' u(t) \|_{L^\infty_{x'}} \lesssim (1+|t|)^{-\frac{1}{2}} \| u(0) \|_{L^1_{x'}}.
 \end{equation}
\end{proposition}

To define the needed dyadic frequency decomposition, let $\chi \in C^\infty_c(\R;\R_{\geq 0})$ be a radially decreasing with
$\chi(x) = 1$ for $|x| \leq 1$ and $\chi(x) = 0$ for $|x| \geq 2$. We set
\begin{equation}\label{eq:Slambda}
\begin{split}
(S'_\lambda f) \widehat (\xi) &= (\chi(\| \xi' \|/\lambda) - \chi(\| \xi' \|/ 2 \lambda)) \hat{f}(\xi), \\
(S_\lambda f) \widehat (\xi) &= (\chi(\| \xi\| /\lambda) - \chi(\| \xi \|/ 2 \lambda)) \hat{f}(\xi)
\end{split}
\end{equation}
for $\lambda \in 2^{\Z}$. Moreover, we write
\begin{align*}
S_0' &= 1 - \sum_{\lambda \in 2^{\N_0}} S'_\lambda, \quad S_0 = 1 - \sum_{\lambda \in 2^{\N_0}} S_\lambda, \qquad
S_{\geq 1}' = 1 - S'_0, \quad S_{\geq 1} = 1 - S_0,\\
S'_{\sim M} &= \sum_{\lambda = M/8}^{8 M} S'_\lambda, \quad S_{\sim M} = \sum_{\lambda = M/8}^{8M} S_\lambda.
\end{align*}

By Littlewood-Paley decomposition, rescaling, and the Keel--Tao interpolation argument \cite[Theorem~1.2]{KeelTao1998}, we find
the desired global  estimates.
\begin{theorem}
\label{thm:GlobalStrichartzConstantCoefficients}
Let $u=(\mathcal{E},\mathcal{H})$ be a solution to \eqref{eq:ConstantCoefficientMaxwell} with \eqref{eq:ConstantCoefficientMaterialLaws}
and $\rho_e = \rho_m = 0$. Then the
global Strichartz estimates
\begin{equation}
 \label{eq:GlobalStrichartz}
 \| |D'|^{-\rho} u \|_{L^p(\R;L^q_{x'})} \lesssim \| u(0) \|_{L^2_{x'}}
\end{equation}
hold with $\rho$ given by \eqref{eq:DerivativeLoss} and $(p,q)$ satisfying \eqref{eq:Admissibility} and $q<\infty$.
\end{theorem}

\begin{remark}
For $q=\infty$,  estimate \eqref{eq:GlobalStrichartz} is true for the homogeneous Besov space:
\begin{equation*}
\| u \|_{L_t^p(\R;\dot{B}^{-\rho}_{q,2}(\R^3))} \lesssim \| u(0) \|_{L^2_{x'}}.
\end{equation*}
To lighten the discussion in the following, we suppose  $q < \infty$. See also Remark \ref{rem:Besov}.
\end{remark}

We next show local-in-time estimates involving charges.
\begin{theorem}
\label{thm:LocalStrichartzConstantCoefficients}
Let $\eps$, $\mu$, $\rho$, $p$, and $q$ be as in Theorem \ref{thm:GlobalStrichartzConstantCoefficients}. We then have
\begin{equation}
\label{eq:LocalStrichartzConstantCoefficients}
\|\langle D' \rangle^{-\rho} u \|_{L^p_TL^q_{x'}} \lesssim T^{\frac{1}{p}} ( \| u(0) \|_{L^2} + \| \rho_{em} \|_{H^{\frac{1}{p}-1}_{x'}} ).
 \end{equation}
\end{theorem}
\begin{proof}
 We begin with estimating the low frequencies
 \begin{equation*}
  \| \langle D' \rangle^{-\rho} S'_{0} u \|_{L^p_TL^q_{x'}} \lesssim T^{\frac{1}{p}} \| u(0) \|_{L^2_{x'}}
 \end{equation*}
  by H\"older in time and Bernstein's inequality.

For the estimate of the high frequencies, we split the initial data into stationary and dispersive components. Let $(\Phi_1,\Phi_2)$
be the solutions to the elliptic equations in \eqref{eq:EllipticEquationCharges}. Since we confine to high frequencies,
we can replace the homogeneous with inhomogeneous norms. We have the stationary solution
$u_{\text{stat}}(t)=S'_{\geq 1}(\nabla \Phi_1,\nabla \Phi_2)$ and we write
\begin{equation*}
S'_{\geq 1} (\mathcal{E},\mathcal{H})(0) = S'_{\geq 1} (\nabla \Phi_1,\nabla \Phi_2) + S'_{\geq 1} (\mathcal{E},\mathcal{H})_{\text{disp}}(0).
\end{equation*}
We denote the solution emanating from the last term by $u_{\text{disp}}$, which is clearly charge free.
Taking the Fourier transform of $\nabla \cdot (\varepsilon \mathcal{E}(0)) = \nabla \cdot (\varepsilon \nabla \Phi_1)$, we obtain
\begin{equation*}
 \frac{i \xi_k \varepsilon^{km} \hat{\mathcal{E}}_m(0,\xi)}{\xi_i \xi_j \varepsilon^{ij}} = \hat{\Phi}_1(\xi).
\end{equation*}
 Plancherel's theorem thus yields
 $\| \nabla \Phi_1 \|_{L^2} \lesssim \|\mathcal{E}(0)\|_{L^2}$ and likewise $\| \nabla \Phi_2 \|_{L^2} \lesssim \| \mathcal{H}(0) \|_{L^2}$.
Hence $\| u_{\text{disp}}(0) \|_{L^2} \lesssim \| (\mathcal{E},\mathcal{H})(0) \|_{L^2}$.

By linearity, we have  $S'_{\geq 1} u = u_{\text{disp}} + u_{\text{stat}}$ and Theorem \ref{thm:GlobalStrichartzConstantCoefficients} implies
\begin{align*}
  \| |D'|^{-\rho} S'_{\geq 1} u \|_{L^p_TL^q_{x'})}
    &\leq \| |D'|^{-\rho} S'_{\geq 1} u_{\text{disp}} \|_{L^p_TL^q_{x'}} + \| |D'|^{-\rho} S'_{\geq 1} u_{\text{stat}}\|_{L^p_TL^q_{x'}} \\
  &\lesssim \| S'_{\geq 1} u(0) \|_{L^2_{x'}} + T^{\frac{1}{p}} \| (\rho_e,\rho_m) \|_{H^{\frac{1}{p}-1}_{x'}}.
\end{align*}
Inequality \eqref{eq:LocalStrichartzConstantCoefficients} follows.
\end{proof}

We turn to the inhomogeneous problem
\begin{equation}
\label{eq:InhomogeneousMaxwell}
 \left\{ \begin{array}{rlrlrl}
  \partial_t \mathcal{D} \!\!\!\! &= \nabla \times \mathcal{H} - \mathcal{J}_e, &\quad \nabla \cdot \mathcal{D}  \!\!\!\!&= \rho_e, &
   \quad \mathcal{D}(0) \!\!\!\!&= \mathcal{D}_0, \\
  \partial_t \mathcal{B}  \!\!\!\!&= - \nabla \times \mathcal{E} - \mathcal{J}_m,& \quad \nabla \cdot \mathcal{B} \!\!\!\!&= \rho_m,&
  \quad \mathcal{B}(0) \!\!\!\!& = \mathcal{B}_0.
 \end{array} \right.
\end{equation}
We write $u = (\mathcal{E},\mathcal{H})$ for solutions to \eqref{eq:InhomogeneousMaxwell} as well as
$\mathcal{J} = (\mathcal{J}_e,\mathcal{J}_m)$ and $\nabla \cdot \mathcal{J} = (\nabla \cdot \mathcal{J}_e,\nabla \cdot \mathcal{J}_m)$.
\begin{theorem}
\label{thm:InhogogeneousStrichartzConstantCoefficients}
 Let $u= (\mathcal{E},\mathcal{H})$ be a solution to \eqref{eq:InhomogeneousMaxwell},  $(\rho,p,q)$ and $(\tilde{\rho},\tilde{p},\tilde{q})$
 satisfy \eqref{eq:Admissibility} and \eqref{eq:DerivativeLoss} with $q,\tilde{q}<\infty$. For $T<\infty$, then the estimates
 \begin{equation}
  \label{eq:LocalInhomogeneousStrichartz}
  \begin{split}
  \| \langle D' \rangle^{-\rho} u \|_{L^p([0,T],L^q(\R^3))} &\lesssim T^{\frac{1}{p}} ( \| u(0) \|_{L^2} + \| \rho_{em}(0) \|_{H^{\frac{1}{p}-1}}) \\
  &\quad + T^{\frac{1}{p}} \big( \| \mathcal{J} \|_{L^1([0,T],L^2)} + \| \nabla \cdot \mathcal{J} \|_{L^1([0,T],H^{\frac{1}{p}-1})} \big).
  \end{split}
 \end{equation}
 hold.
 If $\rho_{em} = 0$ and $\nabla \cdot \mathcal J = 0$, we have the global estimates
 \begin{equation}
  \label{eq:GlobalInhomogeneousStrichartz}
  \| |D'|^{-\rho} u \|_{L^p(\R;L^q(\R^3))} \lesssim \| u(0) \|_{L^2(\R^3)} + \| |D'|^{\tilde{\rho}} \mathcal{J} \|_{L^{\tilde{p}'} L^{\tilde{q}'}}.
 \end{equation}
\end{theorem}
\begin{proof}
 We denote the propagator for the free solution by $U(t)$, so that we can write
\begin{equation*} 
 (\mathcal{D},\mathcal{B})(t) = U(t) (\mathcal{D},\mathcal{B})(0) - \int_0^t U(t-s) \mathcal{J}(s) ds
\end{equation*}
by Duhamel's formula. Since $\| \langle D' \rangle^s u \|_{L^2} \sim \| \langle D' \rangle^s (\mathcal{D},\mathcal{B}) \|_{L^2}$,
we can apply Theorem \ref{thm:LocalStrichartzConstantCoefficients} with initial time $s$ and Minkowski's inequality to find
\eqref{eq:LocalInhomogeneousStrichartz} to hold. Inequality \eqref{eq:GlobalInhomogeneousStrichartz} is a consequence of the dispersive estimate \eqref{eq:DispersiveEstimate} and  Keel-Tao interpolation \cite[Theorem~1.2]{KeelTao1998}.
\end{proof}

\section{Proof of Strichartz estimates in the fully anisotropic case with diagonal material coefficients}
\label{section:DiagonalProof}
In the following we consider diagonal permittivity and permeability
\begin{equation}
\label{eq:DiagonalCoefficients}
\varepsilon = \text{diag}(\varepsilon_1,\varepsilon_2,\varepsilon_3), \quad \mu = \text{diag}(\mu_1,\mu_2,\mu_3), \;
\end{equation}
which are supposed to satisfy the uniform ellipticity condition
\begin{equation}
 \label{eq:EllipticitySectionProof}
 \exists \lambda, \Lambda > 0: \forall x \in \R^4: \, \lambda \leq \nu_i(x) \leq \Lambda \text{ for } i =1,2,3 \text{ and } \nu \in \{ \varepsilon, \mu \}.
\end{equation}
Furthermore, we require the following separability condition, which guarantees uniformity of the curvature bounds for the characteristic surfaces (cf. Assumption \ref{AssumptionMaterialLaws}):
\begin{equation}
 \label{eq:SeparationEigenvalues}
\exists c > 0: \forall x \in \R^4: \forall i \neq j: \, \big| \frac{\varepsilon_i(x)}{\mu_i(x)} - \frac{\varepsilon_j(x)}{\mu_j(x)} \big| \geq c.
\end{equation}

In the following we show Strichartz estimates for Maxwell equations with the rough material laws as above
\begin{align}
\label{eq:3dMaxwellEquations}
\left\{ \begin{array}{rlrl}
\partial_t \mathcal{D} \!\!\!\!&= \nabla \times \mathcal{H} - \mathcal{J}_e, &\quad \nabla \cdot \mathcal{D} \!\!\!\!&= \rho_e, \qquad x \in \R^4, \\
\partial_t \mathcal{B} \!\!\!\!&= -\nabla \times \mathcal{E} - \mathcal{J}_m, &\quad \nabla \cdot \mathcal{B} \!\!\!\!&= \rho_m.
\end{array} \right.
\end{align}
We use the pointwise material laws \eqref{eq:PointwiseMaterialLaws}. Setting
\begin{equation*}
P(x,D) = 
\begin{pmatrix}
- \partial_t (\varepsilon \cdot) & \nabla \times \\
\nabla \times & \partial_t (\mu \cdot)
\end{pmatrix}
, \quad u = 
\begin{pmatrix}
\mathcal{E} \\ \mathcal{H}
\end{pmatrix}
, \quad \rho_{em} =
\begin{pmatrix}
\rho_e \\ \rho_m
\end{pmatrix}
,
\end{equation*}
the Maxwell system \eqref{eq:3dMaxwellEquations} becomes
\begin{equation*}
P(x,D) u = 
\begin{pmatrix}
\mathcal{J}_e \\ -\mathcal{J}_m
\end{pmatrix}, \quad \nabla \cdot \mathcal{D} = \rho_e, \quad \nabla \cdot \mathcal{B} = \rho_m.
\end{equation*}

We first show the basic energy estimate which will be used later in an interpolation argument.
\begin{proposition}
\label{prop:EnergyEstimate}
 Let $T > 0$, $\varepsilon$, $\mu \in C(\R \times \R^3; \R^{3 \times 3})$ satisfy \eqref{eq:EllipticitySectionProof} with $\partial_t (\varepsilon,\mu) \in L^1_{T} L^\infty_{x'}$. We then have the estimate
 \begin{equation}
  \label{eq:EnergyEstimate}
  \| u \|_{L^\infty_T L^2_{x'}} \lesssim \| u_0 \|_{L^2} + \| P u \|_{L_T^1 L^2}.
 \end{equation}
\end{proposition}
\begin{proof}
 We introduce an equivalent norm on $L^2(\R^3)$ by
 \begin{equation*}
  \| u \|^2_{E^0} = \Big\langle u, \begin{pmatrix}
                                \varepsilon & 0 \\
                                0 & \mu
                               \end{pmatrix}
u \Big\rangle.
 \end{equation*}
For homogeneous solutions we compute
\begin{equation*}
\begin{split}
 \partial_t \big( \int_{\R^3} \mathcal{E}. \mathcal{D} + \mathcal{H}.\mathcal{B} \big)
 &= \int_{\R^3} ( \partial_t \mathcal{E}).\mathcal{D} + \mathcal{E}.(\partial_t \mathcal{D}) + (\partial_t \mathcal{H}). \mathcal{B} + \mathcal{H}.(\partial_t \mathcal{B}) \\
 &=: I + II + III + IV.
 \end{split}
\end{equation*}
We find
\begin{equation*}
II + IV = \int_{\R^3} \mathcal{E}.(\nabla \times \mathcal{H}) - (\nabla \times \mathcal{E}).\mathcal{H} = 0
\end{equation*}
by the symmetry of the curl-operator. Note that
\begin{equation*}
 \begin{split}
  (\partial_t \varepsilon) \mathcal{E} + \varepsilon \partial_t \mathcal{E} = \nabla \times \mathcal{H}
  &\iff \partial_t \mathcal{E} = \varepsilon^{-1} \nabla \times \mathcal{H} - \varepsilon^{-1} (\partial_t \varepsilon) \mathcal{E}, \\
  (\partial_t \mu) \mathcal{H} + \mu \partial_t \mathcal{H} = - \nabla \times \mathcal{E}
  &\iff \partial_t \mathcal{H} = - \mu^{-1} (\nabla \times \mathcal{E}) - \mu^{-1} (\partial_t \mu) \mathcal{H},
 \end{split}
\end{equation*}
which yields
\begin{equation*}
\begin{split}
 I + III &= \int (\varepsilon^{-1} (\nabla \times \mathcal{H})). \mathcal{D} -  \mu^{-1} (\nabla \times \mathcal{E}).\mathcal{B} -  \varepsilon^{-1} ((\partial_t \varepsilon) \mathcal{E}).\mathcal{D} -\mu^{-1} ((\partial_t \mu) \mathcal{H}). \mathcal{B}\\
 &= - \int (\partial_t \varepsilon) \mathcal{E}. \mathcal{E} + \mathcal{H}. (\partial_t \mu) \mathcal{H}.
 \end{split}
\end{equation*}
We obtain
\begin{equation*}
 |I + III| \lesssim \| \partial_t (\varepsilon,\mu) \|_{L^\infty_{x'}} \| u \|^2_{E^0},
\end{equation*}
which leads to the estimate
\begin{equation*}
 \partial_t \| u \|^2_{E^0} \lesssim \| (\partial_t \varepsilon, \partial_t \mu) \|_{L_{x'}^\infty} \| u \|^2_{E^0}.
\end{equation*}
Gr\o nwall's lemma and the equivalence of the norms imply
\begin{equation}
 \label{eq:FreeSolutionEstimate}
  \| u(t) \|^2_{L^2} \lesssim e^{ C\| (\partial_t \varepsilon, \partial_t \mu) \|_{L_{t}^1 L_{x'}^\infty}} \| u(0) \|^2_{L^2}.
\end{equation}

For the proof of the inhomogeneous estimate we consider the evolution of the variables $(\mathcal{D},\mathcal{B})$.
By \eqref{eq:EllipticitySectionProof}, for vanishing currents $\mathcal{J}_{em}=0$ we have
\begin{equation*}
 \| (\mathcal{D},\mathcal{B})(t) \|_{L^2} \lesssim e^{ C \int_0^t \| \partial_t (\varepsilon,\mu) \|_{L^\infty_{x'}} ds} \| (\mathcal{D},\mathcal{B})(0) \|_{L^2}.
\end{equation*}
Let $U(t,s)$ be the propagator of the homogeneous Maxwell equations
\begin{equation*}
 \left\{ \begin{array}{cl}
  \partial_t (\mathcal{D},\mathcal{B}) \!\!\!\! &= Q (\mathcal{D},\mathcal{B}), \\
  (\mathcal{D},\mathcal{B})(s) \!\!\!\! &= (\mathcal{D},\mathcal{B})_0 \in L^2,
 \end{array} \right.
\end{equation*}
The claim now follows from Duhamel's formula, writing
\begin{equation*}
 (\mathcal{D},\mathcal{B})(t) = U(t,0) (\mathcal{D},\mathcal{B})_0 - \int_0^t U(t,s) (\mathcal{J}_e,\mathcal{J}_m)(s) ds
\end{equation*}
and applying Minkowski's inequality and \eqref{eq:FreeSolutionEstimate}.
\end{proof}

We next show the following local-in-time Strichartz estimates.
\begin{theorem}
\label{thm:LocalStrichartzEstimatesFullyAnisotropic}
Let $0<s \leq 1$, $\varepsilon , \mu \in C^s(\R^4;\R^{3\times3})$ be as in \eqref{eq:DiagonalCoefficients} satisfying \eqref{eq:EllipticitySectionProof} and \eqref{eq:SeparationEigenvalues}. Suppose that $\rho = 3 \big( \frac{1}{2} - \frac{1}{q} \big) - \frac{1}{p}$, $2 \leq p < q \leq \infty$, and $\frac{2}{p} + \frac{1}{q} = \frac{1}{2}$.
We find the following estimate to hold:
\begin{equation}
\label{eq:StrichartzEstimateFullyAnisotropicC1}
\| |D|^{-\rho+\frac{s-2}{4}} u \|_{(L^p L^q)_{2}} \lesssim (1+\| (\varepsilon,\mu) \|_{C^s}) \| u \|_{L^2} + \| |D|^{\frac{s-2}{2}} P u \|_{L^2} + \| |D|^{-\frac{1}{2}+\frac{s-2}{4}} \rho_{em} \|_{L^2}.
\end{equation}
\end{theorem}
\begin{remark}
\label{rem:Besov}
We consider the Besov norm with mixed coefficients defined by
\begin{equation*}
\| f \|_{(L^p L^q)_{2}} = \big( \sum_{\lambda \in 2^{\mathbb{N}_0} \cup \{ 0 \}} \| S_\lambda f \|_{L^p L^q}^2 \big)^{\frac{1}{2}}
\end{equation*}
on the left hand-side because at the endpoint $q=\infty$ the Littlewood-Paley square-function estimate fails.
For $q<\infty$ we can simplify \eqref{eq:StrichartzEstimateFullyAnisotropicC1} to
\begin{equation*}
\| |D|^{-\rho+\frac{s-2}{4}} u \|_{L^p L^q} \lesssim (1+ \|(\varepsilon,\mu)\|_{C^s}) \| u \|_{L^2} + \| |D|^{\frac{s-2}{2}} P u \|_{L^2} + \| |D|^{-\frac{1}{2}+\frac{s-2}{4}} \rho_{em} \|_{L^2}.
\end{equation*}
\end{remark}
We prove Theorem \ref{thm:LocalStrichartzEstimatesFullyAnisotropic} by phase space analysis close to the endpoint.
To this end, we apply the FBI transform (cf. \cite{Tataru2000,Tataru2001,Tataru2002}) to change into phase space after
\begin{itemize}
\item reducing to high frequencies,
\item spatial  and dyadic frequency localisation of the solution,
\item frequency truncation of $\varepsilon(x)$ and $\mu(x)$.
\end{itemize}
After these reductions,  as in \cite{SchippaSchnaubelt2022,Schippa2021Maxwell3d} we arrive at the following proposition.
\begin{proposition}
\label{prop:ReductionFrequencyLocalizedEstimate}
Let $\lambda \gg \lambda_0$ with $\lambda_0$ depending on the ellipticity of $\varepsilon$ and $\mu$ and $\| (\varepsilon,\mu) \|_{C^s_{x}}$
for $0<s\leq 1$. Let $P_\lambda$ denote the operator $P$ with coefficients frequency truncated at frequencies $< \lambda/8$ and let $u$
be essentially supported in the unit cube. Under asumptions of Theorem \ref{thm:LocalStrichartzEstimatesFullyAnisotropic}, estimate
\eqref{eq:StrichartzEstimateFullyAnisotropicC1} follows from
\begin{equation}
\label{eq:FrequencyLocalizedEstimate}
\begin{split}
\lambda^{-\rho + \frac{s-2}{4}} \| S_\lambda u \|_{L^p L^q} &\lesssim (1 + \| (\varepsilon,\mu) \|_{C^s}) \| S_\lambda u \|_{L^2} + \lambda^{\frac{s-2}{2}} \| P_\lambda S_\lambda u \|_{L^2} \\
&\quad + \lambda^{- \frac{1}{2} + \frac{s-2}{4} } \| S_\lambda \rho_{em} \|_{L^2}.
\end{split}
\end{equation}
\end{proposition}

The estimate corresponds to Strichartz estimates from \cite{Tataru2000} for wave equations with rough coefficients
\emph{with integrability exponents as in the two-dimensional case}. The reason that we cannot hope for the Strichartz range of the
wave equation in three dimensions is that the characteristic surface is no longer regular, but has conical singularities.
In \cite{Tataru2001,Tataru2002} Tataru used the regularity of the Hamilton flow to recover the Euclidean Strichartz estimates locally-in-time for $C^2$-coefficients and coefficients with $\partial^2 g \in L_t^1 L_{x'}^\infty$. In the present context it is not clear how to take advantage of the Hamilton flow due to the strong system character. For this reason, we do not recover the derivative loss over scaling as for scalar wave equations with H\"older-continuous coefficients in two space dimensions proved in \cite{Tataru2001}. (Also note that the results in \cite{Tataru2001} were shown to be essentially optimal in \cite{SmithTataru2002}).

\medskip

In a first step we show the above reduction result.

\subsection{Proof of Proposition \ref{prop:ReductionFrequencyLocalizedEstimate}}
Choosing $\lambda_0$ large enough, the uniform ellipticity is preserved because
\begin{equation*}
\| \kappa_{\gtrsim \lambda} \|_{L^\infty_x} \lesssim \lambda^{-s} \| \kappa_{\gtrsim \lambda} \|_{C^s_{x}}, \quad \kappa \in \{ \varepsilon; \mu \}.
\end{equation*}
Using Sobolev's embedding,  we first  estimate  low frequencies by
\[\| |D|^{-\rho + \frac{s-2}{4}} S_{0} u \|_{L^p L^q} \lesssim \| |D|^{\frac12 + \frac{s-2}{4}} S_{0} u \|_{L^2}
  \lesssim \|S_{0} u \|_{L^2}.\]
It remains to show the estimate
\begin{align*}
\| |D|^{-\rho + \frac{s-2}{4}} S_{\geq 1} u \|_{(L^p L^q)_2} &\lesssim (1+\| (\varepsilon,\mu) \|_{C^s}) \| u \|_{L^2}
    + \| |D|^{\frac{s-2}{2}} P u \|_{L^2} \\
   & \quad + \| |D|^{-\frac{1}{2}+\frac{s-2}{4}}  \rho_{em} \|_{L^2}
\end{align*}
for high frequencies.
To deduce this inequality from \eqref{eq:FrequencyLocalizedEstimate} by means of the definition of $(L^p L^q)_{2}$,
we have to bound $\|P_\lambda S_\lambda u\|_{L^2}$  by   $\|S_\lambda Pu\|_{L^2}$ plus summable error terms. We write
$\tilde{S}_\lambda = \sum_{\mu = \lambda/8}^{8 \lambda} S_\mu$ and $P_{\sim\lambda}$ for the operator $P$ with coefficients
frequency localized near $\lambda$. We then obtain
\begin{align*}
\| \tilde{S}_\lambda P_\lambda S_\lambda u \|_{L^2}
  &\le  \| \tilde{S}_\lambda P S_\lambda u \|_{L^2}  + \| \tilde{S}_\lambda P_{\sim\lambda}S_\lambda u \|_{L^2}\\
&\leq \| S_\lambda P u \|_{L^2} + \| \tilde{S}_\lambda [P,S_\lambda] u \|_{L^2}
    +\|\tilde{S}_\lambda P_{\sim\lambda} S_\lambda u\|_{L^2}.
\end{align*}
Since $P$ is in divergence form, we can factor out the derivatives and estimate
\begin{align*}
\| \tilde{S}_\lambda [P,S_\lambda] u \|_{L^2}
  &\lesssim \lambda \| [(\varepsilon,\mu),S_\lambda] u \|_{L^2} \lesssim \lambda^{1-s} \| (\varepsilon,\mu) \|_{C^s} \| u \|_{L^2},\\
\|\tilde{S}_\lambda P_{\sim\lambda} S_\lambda u\|_{L^2}
  &\lesssim \lambda^{1-s}\| \lambda^s (\tilde{S}_\lambda(\varepsilon,\mu))  S_\lambda u\|_{L^2}
   \lesssim \lambda^{1-s} \| (\varepsilon,\mu) \|_{C^s} \| u \|_{L^2}.
\end{align*}
In the last step, we used a well-known kernel estimate.
This finishes the proof of Proposition \ref{prop:ReductionFrequencyLocalizedEstimate}. $\hfill \Box$

\medskip

In the following we tacitly suppose $\varepsilon$ and $\mu$ are truncated to frequencies $< \lambda / 8$ to lighten the notation.

\subsection{Applying the FBI transform}
In this section we prove \eqref{eq:FrequencyLocalizedEstimate}. In the following we use the FBI transform to reduce to an estimate in phase space. We refer to \cite[Section~2]{SchippaSchnaubelt2022} for further explanations.

For $\lambda \in 2^{\mathbb{Z}}$, the FBI transform of an integrable function $f:\R^m \to \C$ is defined by
\begin{equation*}
\begin{split}
T_\lambda f(z) &= C_m\lambda^{\frac{3m}{4}}\int_{\R^m} e^{-\frac{\lambda}{2}(z-y)^2} f(y) dy,\quad z = x-i\xi \in T^*\R^m \simeq \R^{2m},\\
C_m &= 2^{-\frac{m}{2}} \pi^{-\frac{3m}{4}}.
\end{split}
\end{equation*}
The FBI transform is an isometric mapping $T_\lambda : L^2(\R^m) \to L^2_\Phi(T^* \R^{m})$, where $\Phi(z) = e^{-\lambda \xi^2}$.
But $T_\lambda$ is not surjective because $T_\lambda f$ is holomorphic (which motivates to write $z = x-i\xi$).
The adjoint in $L^2_\Phi$ provides the inversion formula $T^*_\lambda T_\lambda f=f$ with
\begin{equation*}
T_\lambda^* F(y) = C_m \lambda^{\frac{3m}{4}} \int_{\R^{2m}} e^{-\frac{\lambda}{2}(\bar{z}-y)^2} \Phi(z) F(z) dx d\xi.
\end{equation*} 
Originally, the FBI transform was used in the context of microlocal analysis to obtain an expansion of analytic symbols. Tataru reversed
the logic in \cite{Tataru2000,Tataru2001} to find expansions of rough symbols $a \in C^s_x C^\infty_c$, which satisfy suitable error
bounds. Let $a \in C^s_x C^\infty_c$ with $a(x,\xi) = 0$ for $\xi \notin B(0,2) \backslash B(0,1/2)$ and
\begin{equation*}
\tilde{a}^s_\lambda = \sum_{|\alpha| + |\beta| < s} (\partial_\xi - \lambda \xi)^\alpha \frac{\partial_x^\alpha \partial_\xi^\beta a(x,\xi)}{|\alpha|! |\beta|! (-i \lambda)^{|\alpha|} \lambda^{|\beta|}} ( \frac{1}{i} \partial_x - \lambda \xi)^\beta.
\end{equation*}
For $s \leq 1$, we have $\tilde{a}^s_\lambda = a$. Tataru proved the following approximation theorem in
\cite[Theorem~5]{Tataru2001} and \cite[Theorem~2]{Tataru2000}.
\begin{theorem} \label{thm:ApproximationTheoremFBITransform}
Suppose that $a \in C^s_x C^\infty_c$. Set $a_\lambda(x,\xi) = a(x,\xi/\lambda)$ and $R^s_\lambda = T_\lambda a_\lambda(x,D) - \tilde{a}^s_\lambda(x,\xi) T_\lambda$. Then,
\begin{equation}
\label{eq:FBITransformErrorEstimateI}
\| R^s_\lambda \|_{L^2 \to L^2_\Phi} \lesssim \lambda^{-\frac{s}{2}}.
\end{equation}
Moreover, if $\partial a \in L^2_{x_0} L^\infty_{x'}$, then
\begin{equation}
\label{eq:FBITransformErrorEstimateII}
 \| R^s_\lambda \|_{L^\infty_{x_0} L^2_{x'} \to L^2_\Phi} \lesssim \lambda^{-\frac{1}{2}}.
\end{equation}
\end{theorem}

We return to the proof of \eqref{eq:FrequencyLocalizedEstimate}. Write $P_\lambda S_\lambda u = J_\lambda$ and $T_\lambda S_\lambda u = v_\lambda$, $T_\lambda (J_\lambda / \lambda)= f_\lambda$. An application of Theorem \ref{thm:ApproximationTheoremFBITransform} yields
\begin{equation*}
p(x,\xi) v_\lambda = f_\lambda + g_\lambda \quad
       \text{with \ } \| g_\lambda \|_{L^2_\Phi} \lesssim_{ \| (\varepsilon,\mu)  \|_{C^s}} \lambda^{-\frac{s}{2}} \| S_\lambda u \|_{L^2}
\end{equation*}
and
\begin{equation*}
p(x,\xi) = 
\begin{pmatrix}
-i \xi_0 \varepsilon & i \mathcal{C}(\xi') \\
i \mathcal{C}(\xi') & i \xi_0 \mu
\end{pmatrix}
\in \C^{6 \times 6}
.
\end{equation*}
In the above display we denote $\mathcal{C}(\xi') = (- \varepsilon^{ijk} \xi_k)_{ij}$ with $(\varepsilon^{ijk})$ the Levi-Civita symbol such that $(\nabla \times f) \widehat (\xi) = i \mathcal{C}(\xi') \hat{f}(\xi)$. The estimate \eqref{eq:FrequencyLocalizedEstimate} becomes
\begin{equation*}
\lambda^{-\rho + \frac{s-2}{4}} \| T_\lambda^* v_\lambda \|_{L^p L^q} \lesssim (1+\| (\varepsilon,\mu) \|_{C^s}) \| v_\lambda \|_{L^2_\Phi} + \lambda^{\frac{s}{2}} \| f_\lambda \|_{L^2_\Phi} + \lambda^{\frac{s-2}{4} - \frac{1}{2} } \| S_\lambda \rho_{em} \|_{L^2}
\end{equation*}
by invoking \eqref{eq:FBITransformErrorEstimateI}. 
It thus suffices to prove
\begin{equation}
\label{eq:EstimateinPhaseSpace}
\| T_\lambda^* v_\lambda \|_{L^p L^q} \lesssim \lambda^\rho ( \lambda^{\frac{2-s}{4}} \| v_\lambda \|_{L^2_\Phi} + \lambda^{\frac{2+s}{4}} \| p(x,\xi) v_\lambda \|_{L^2_\Phi} + \lambda^{-\frac{1}{2}} \| S_\lambda \rho_{em} \|_{L^2} ).
\end{equation}

By non-stationary phase, we can suppose that $\text{supp} (v_\lambda) \subseteq B(0,1) \times \{ 1/4 \leq |\xi| \leq 4 \} =: U$ (up to a negligible error). In the following we reduce the vector-valued estimate to a scalar one. For this purpose we adapt arguments from the analysis of the constant-coefficient time-harmonic case given in \cite{MandelSchippa2022}. Let
\begin{equation}
\label{eq:Characteristic}
  \tilde{q}(x,\xi)= - \xi_0^2 q(x,\xi)
  := - \xi_0^2 (\xi_0^4 - \xi_0^2 q_0(x,\xi) + q_1(x,\xi)) ,
\end{equation}
where
\begin{align*}
q_0(x,\xi) &= \xi_1^2 \big( \frac{1}{\varepsilon_2 \mu_3} + \frac{1}{\mu_2 \varepsilon_3} \big) + \xi_2^2 \big( \frac{1}{\varepsilon_1 \mu_3} + \frac{1}{\mu_1 \varepsilon_3} \big) + \xi_3^2 \big( \frac{1}{\varepsilon_1 \mu_2} + \frac{1}{\varepsilon_2 \mu_1} \big), \\
q_1(x,\xi) &= \frac{1}{\varepsilon_1 \varepsilon_2 \varepsilon_3 \mu_1 \mu_2 \mu_3} (\varepsilon_1 \xi_1^2 + \varepsilon_2 \xi_2^2 + \varepsilon_3 \xi_3^2) (\mu_1 \xi_1^2 + \mu_2 \xi_2^2 + \mu_3 \xi_3^2).
\end{align*}
The dependence of $\varepsilon$ and $\mu$ on $x$ is suppressed in the above display for the sake of brevity.
\begin{proposition}
\label{prop:ScalarReduction}
For the proof of \eqref{eq:EstimateinPhaseSpace} under the assumptions of Proposition \ref{prop:ReductionFrequencyLocalizedEstimate}, it suffices to show the estimate
\begin{equation}
\label{eq:ScalarReducedEstimate}
\| T_\lambda^* w_\lambda \|_{L^p L^q} \lesssim \lambda^\rho (\lambda^{\frac{2-s}{4}} \| w_\lambda \|_{L^2_\Phi} + \lambda^{\frac{2+s}{4}} \| q(x,\xi) w_\lambda \|_{L^2_\Phi} )
\end{equation}
for $w_\lambda:\R^4 \times \R^4 \to \C$ with $\text{supp} (w_\lambda) \subseteq [-1,1]^4 \times \{(\xi_0,\xi') \in \R \times \R^3 : |\xi_0| \sim |\xi'| \sim 1 \}$.
\end{proposition}
\begin{proof}
In the first step, we separate the evolution of the $\mathcal{E}$- and $\mathcal{H}$-variables by passing to a second order system. We multiply $p$ with the symmetrizer $\sigma$
\begin{equation*}
\sigma(x,\xi) = 
\begin{pmatrix}
-i \xi_0 \varepsilon^{-1} & i \varepsilon^{-1} \mathcal{C}(\xi') \mu^{-1} \\
i \mu^{-1} \mathcal{C}(\xi') \varepsilon^{-1} & i \xi_0 \mu^{-1}
\end{pmatrix}
\in \C^{6 \times 6}
\end{equation*}
 (cf. \cite[Proposition~1.2]{MandelSchippa2022}), obtaining
\begin{equation}\label{eq:ME-MH}
\sigma(x,\xi) p(x,\xi) = 
\begin{pmatrix}
M_E(x,\xi) - \xi_0^2 & 0 \\
0 & M_H(x,\xi) - \xi_0^2
\end{pmatrix}
.
\end{equation}
Here we set
\begin{align*}
M_E(x,\xi)= - \varepsilon^{-1}(x) \mathcal{C}(\xi') \mu^{-1}(x) \mathcal{C}(\xi'), \quad
M_H(x,\xi) = -\mu^{-1}(x) \mathcal{C}(\xi') \varepsilon^{-1}(x) \mathcal{C}(\xi').
\end{align*}
Moreover, in \cite[p.~1835,~Eq.~(13)]{MandelSchippa2022} it was pointed out that
\begin{equation*}
 q(x,\xi) = \det(M_E(x,\xi') - \xi_0^2) = \det(M_H(x,\xi') - \xi_0^2).
\end{equation*}

Let $v_\lambda = (v_{\lambda,1} , v_{\lambda,2})$ denote the grouping into $3$-vectors. By boundedness of the entries of $\sigma(x,\xi)$ for $\xi \in U$, it suffices to prove
\begin{equation*}
\| T_\lambda^* v_\lambda \|_{L^p L^q} \lesssim \lambda^\rho ( \lambda^{\frac{2-s}{4}} \| v_\lambda \|_{L^2_\Phi} + \lambda^{\frac{2+s}{4}} \| \sigma(x,\xi) p(x,\xi) v_\lambda \|_{L^2_\Phi} + \lambda^{-\frac{1}{2}} \| S_\lambda \rho_{em} \|_{L^2})
\end{equation*}
which can be written  as
\begin{align}
\label{eq:PhaseSpaceEstimateA}
\| T_\lambda^* v_{\lambda,1} \|_{L^p L^q} &\lesssim \lambda^\rho (\lambda^{\frac{2-s}{4}} \| v_{\lambda,1} \|_{L^2_\Phi}
   + \lambda^{\frac{2+s}{4}} \| (M_E(x,\xi') - \xi_0^2) v_{\lambda,1} \|_{L^2_\Phi} \notag\\
   &\quad + \lambda^{-\frac{1}{2}} \| S_\lambda \rho_e \|_{L^2}), \\
\| T_\lambda^* v_{\lambda,2} \|_{L^p L^q} &\lesssim \lambda^\rho (\lambda^{\frac{2-s}{4}} \| v_{\lambda,2} \|_{L^2_\Phi}
  + \lambda^{\frac{2+s}{4}} \| (M_H(x,\xi') - \xi_0^2) v_{\lambda,2} \|_{L^2_\Phi} \notag\\
  &\quad+ \lambda^{-\frac{1}{2}} \| S_\lambda \rho_m \|_{L^2} ).\label{eq:PhaseSpaceEstimateB}
\end{align}
The estimates will be proved in the regions
\begin{itemize}
 \item[(1)] $\{ 1 \sim |\xi_0| \gg |\xi'| \}$,
 \item[(2)] $\{ 1 \sim |\xi'| \gg |\xi_0| \}$, \ for which we have to consider the charges,
 \item[(3)] $\{ 1 \sim |\xi_0| \sim |\xi'| \}$.
\end{itemize}

The region $\{ 1 \sim |\xi_0| \gg |\xi'| \}$ is easy to handle. Here we have
\begin{equation*}
\| (M_X(x,\xi') - \xi_0^2) v_\lambda \|_{L^2_\Phi} \gtrsim \| v_\lambda \|_{L^2_\Phi}, \quad X \in \{ E; H \},
\end{equation*}
and so the claim follows by Sobolev embedding via
\begin{equation*}
\| T_\lambda^* v_{\lambda,i} \|_{L^p L^q} \lesssim \lambda^{\rho + \frac{1}{2}} \| v_{\lambda,i} \|_{L^2_\Phi} \lesssim \lambda^{\rho + \frac{1}{2} + \frac{s}{4}} \| (M_X(x,\xi') - \xi_0^2) v_\lambda \|_{L^2_\Phi}.
\end{equation*}

\smallskip

For the second region $\{|\xi_0| \ll |\xi'| \sim 1 \}$ the charges come into play. We decompose $v_{\lambda,1} = v_{\lambda,1}^s + v_{\lambda,1}^p$ with
\begin{equation*}
v_{\lambda,1}^p = \frac{(\xi'.\varepsilon^\lambda v_{\lambda,1}) \xi'}{|\xi'|^2_\varepsilon}, \quad
  |\xi'|^2_{\varepsilon} = \xi'.\varepsilon^\lambda \xi',
\end{equation*}
 again indicating  the frequency cutoff of the coefficients temporarily.
The contribution of $v_{\lambda,1}^p$ is estimated by Sobolev embedding:
\begin{equation*}
\| T_\lambda^* v_{\lambda,1}^p \|_{L^p L^q} \lesssim \lambda^{\rho + \frac{1}{2}} \| v_{\lambda,1}^p \|_{L^2_\Phi} \lesssim \lambda^{\rho + \frac{1}{2}} \| (\frac{\xi'}{|\xi'|_\varepsilon}.\varepsilon^\lambda v_{\lambda,1}) \|_{L^2_\Phi}.
\end{equation*}
 Theorem \ref{thm:ApproximationTheoremFBITransform} and commutator estimates yield
\begin{equation*}
\| \frac{\xi'}{|\xi'|_\varepsilon}.(\varepsilon^\lambda v_{\lambda,1}) \|_{L^2_\Phi} \lesssim \lambda^{-\frac{s}{2}} \| u_\lambda \|_{L^2} + \lambda^{-1} \| S_\lambda \nabla'.(\varepsilon^\lambda \mathcal{E}) \|_{L^2}.
\end{equation*}
By expanding $\varepsilon^\lambda = \varepsilon + (\varepsilon - \varepsilon^\lambda)$, we find
\begin{equation*}
\begin{split}
\lambda^{-1} \| S_\lambda \nabla'.(\varepsilon^\lambda \mathcal{E}) \|_{L^2} &\lesssim \lambda^{-1} \| S_\lambda \nabla'. (\varepsilon \mathcal{E}) \|_{L^2} + \| (\varepsilon - \varepsilon^\lambda) \tilde{S}_\lambda \mathcal{E} \|_{L^2} \\
&\lesssim \lambda^{-1} \| S_\lambda \rho_e \|_{L^2} + \lambda^{-s} \| \tilde{S}_\lambda \mathcal{E} \|_{L^2}.
\end{split}
\end{equation*}
This is an acceptable estimate for $v_{\lambda,1}^p$ in terms of $\rho_e$ and $\mathcal{E}$.

We turn to $v_{\lambda,1}^s$ noting that  $\xi'. (\varepsilon^\lambda v_{1,\lambda}^s) = 0$.
In the following we show that $|(M_E(x,\xi)- \xi_0^2) v_1^\lambda| \gtrsim |v_1^\lambda|$ for $\xi'.(\varepsilon^\lambda v_1^\lambda) = 0$
and $\{|\xi_0| \ll |\xi'| \}$. We mit again the superscript $\lambda$ of the coefficients.
For $\xi_0 \neq 0$ we have $q(x,\xi) \neq 0$. Let $w = (M_E - \xi_0^2) v_1$. We have to show that
\begin{equation*}
|(M_E - \xi_0^2)^{-1} w | \lesssim |w|
\end{equation*}
noting that $\xi'.(\varepsilon w) = 0$. In \cite[Section~2.1]{MandelSchippa2022} the matrix $(M_E - \xi_0^2)$ was inverted by
\begin{equation*}
(M_E - \xi_0^2)^{-1} = q^{-1}(x,\xi) \text{adj}(M_E - \xi_0^2) = \frac{-1}{\xi_0^2( \xi_0^4 - \xi_0^2 q_0(x,\xi) + q_1(x,\xi))} Z_{\varepsilon,\mu} \frac{\varepsilon}{\varepsilon_1 \varepsilon_2 \varepsilon_3},
\end{equation*}
wher the components of $Z=Z_{\varepsilon,\mu}$ are given by
\begin{align*} 
  \begin{aligned}
  Z_{11}(\xi)
  &= \xi_1^2( \frac{\xi_1^2}{\mu_2\mu_3}+\frac{\xi_2^2}{\mu_1\mu_3}+\frac{\xi_3^2}{\mu_1\mu_2}) 
  - \xi_0^2(
  \frac{\eps_2}{\mu_2}\xi_1^2+\frac{\eps_3}{\mu_3}\xi_1^2
  +\frac{\eps_2}{\mu_1}\xi_2^2+\frac{\eps_3}{\mu_1}\xi_3^2) +\xi_0^4\eps_2\eps_3, \\
  Z_{12}(\xi)
  &= Z_{21}(\xi) 
  =  \xi_1\xi_2(\frac{\xi_1^2}{\mu_2\mu_3}+\frac{\xi_2^2}{\mu_1\mu_3}+\frac{\xi_3^2}{\mu_1\mu_2}-\xi_0^2\frac{\eps_3}{\mu_3}),
  \\
  Z_{13}(\xi)
  &= Z_{31}(\xi)
  =
  \xi_1\xi_3(\frac{\xi_1^2}{\mu_2\mu_3}+\frac{\xi_2^2}{\mu_1\mu_3}+\frac{\xi_3^2}{\mu_1\mu_2}-\xi_0^2\frac{\eps_2}{\mu_2}),
  \\
  Z_{22}(\xi)
  &= \xi_2^2( \frac{\xi_1^2}{\mu_2\mu_3}+\frac{\xi_2^2}{\mu_1\mu_3}+\frac{\xi_3^2}{\mu_1\mu_2}) 
  - \xi_0^2(
  \frac{\eps_1}{\mu_2}\xi_1^2+\frac{\eps_3}{\mu_3}\xi_2^2
  +\frac{\eps_1}{\mu_1}\xi_2^2+\frac{\eps_3}{\mu_2}\xi_3^2) +\xi_0^4\eps_1\eps_3, \\
  Z_{23}(\xi)
  &= Z_{32}(\xi)
   =\xi_2\xi_3(\frac{\xi_1^2}{\mu_2\mu_3}+\frac{\xi_2^2}{\mu_1\mu_3}+\frac{\xi_3^2}{\mu_1\mu_2}-\xi_0^2\frac{\eps_1}{\mu_1}),\\ 
  Z_{33}(\xi) 
  &= \xi_3^2( \frac{\xi_1^2}{\mu_2\mu_3}+\frac{\xi_2^2}{\mu_1\mu_3}+\frac{\xi_3^2}{\mu_1\mu_2}) 
  - \xi_0^2(
  \frac{\eps_1}{\mu_3}\xi_1^2+\frac{\eps_2}{\mu_3}\xi_2^2
  +\frac{\eps_1}{\mu_1}\xi_3^2+\frac{\eps_2}{\mu_2}\xi_3^2) +\xi_0^4\eps_1\eps_2.
\end{aligned}
\end{align*}
It is a crucial observation that $Z_{\varepsilon,\mu}$ will be applied to vectors $v$ with $\xi'. v = 0$. This reduces to
$Z_{\varepsilon,\mu}^{\text{eff}}$, having the coefficcients
\begin{align*} 
  \begin{aligned}
  Z^{\text{eff}}_{11}(\xi)
  &= - \xi_0^2(
  \frac{\eps_2}{\mu_2}\xi_1^2+\frac{\eps_3}{\mu_3}\xi_1^2
  +\frac{\eps_2}{\mu_1}\xi_2^2+\frac{\eps_3}{\mu_1}\xi_3^2) +\xi_0^4\eps_2\eps_3, \\
  Z^{\text{eff}}_{12}(\xi)
  &= Z^{\text{eff}}_{21}(\xi) 
  =  - \xi_1 \xi_2 \xi_0^2\frac{\eps_3}{\mu_3},
  \\
  Z^{\text{eff}}_{13}(\xi)
  &= Z^{\text{eff}}_{31}(\xi)
  =
  -\xi_0^2\frac{\eps_2}{\mu_2} \xi_1 \xi_3,
  \\
  Z^{\text{eff}}_{22}(\xi)
  &= - \xi_0^2(
  \frac{\eps_1}{\mu_2}\xi_1^2+\frac{\eps_3}{\mu_3}\xi_2^2
  +\frac{\eps_1}{\mu_1}\xi_2^2+\frac{\eps_3}{\mu_2}\xi_3^2) +\xi_0^4\eps_1\eps_3, \\
  Z^{\text{eff}}_{23}(\xi)
  &= Z^{\text{eff}}_{32}(\xi)
   = - \xi_2 \xi_3 \xi_0^2\frac{\eps_1}{\mu_1},\\ 
  Z^{\text{eff}}_{33}(\xi) 
  &= - \xi_0^2(
  \frac{\eps_1}{\mu_3}\xi_1^2+\frac{\eps_2}{\mu_3}\xi_2^2
  +\frac{\eps_1}{\mu_1}\xi_3^2+\frac{\eps_2}{\mu_2}\xi_3^2) +\xi_0^4\eps_1\eps_2.
\end{aligned}
\end{align*}
We can write $Z^{\text{eff}} = \xi_0^2 \tilde{Z}^{\text{eff}} $ with $\tilde{Z}^{\text{eff}}$ having bounded entries for $|\xi_0| \lesssim 1$. This gives
\begin{align*}
(M_E(x,\xi) - \xi_0^2)^{-1} w &= q^{-1}(x,\xi) \text{adj}(M_E - \xi_0^2) w \\
&= \frac{-1}{(\xi_0^4 - \xi_0^2 q_0(x,\xi) + q_1(x,\xi))} \tilde{Z}^{\text{eff}}_{\varepsilon,\mu} \frac{\varepsilon}{\varepsilon_1 \varepsilon_2 \varepsilon_3} w.
\end{align*}
For $|\xi_0| \ll |\xi'|$ or $|\xi_0| \gg |\xi'|$, we have
\begin{equation*}
| \xi_0^4 - \xi_0^2 q_0(\xi) + q_1(\xi) | \sim 1 \text{ provided that } \max(|\xi_0|,|\xi'|) \sim 1.
\end{equation*}
Consequently, for $|\xi_0| \ll |\xi'| \sim 1$, it follows
\begin{equation*}
\| (M_E(x,\xi) - \xi_0^2) v \|_{L^2_\Phi(U)} \gtrsim \| v \|_{L^2_{\Phi}}.
\end{equation*}
This shows estimate \eqref{eq:PhaseSpaceEstimateA} in the second region. \eqref{eq:PhaseSpaceEstimateB} follows by exchanging the roles of $\varepsilon$ and $\mu$.

\smallskip

We turn to the third region. Again we focus on \eqref{eq:PhaseSpaceEstimateA}, for which we prove
\begin{equation*}
 \| T_\lambda^* v_{\lambda,1} \|_{L^p L^q} \lesssim \lambda^\rho (\lambda^{\frac{2-s}{4}} \| v_{\lambda,1} \|_{L^2_\Phi} + \lambda^{\frac{2+s}{4}} \| (M_E - \xi_0^2) v_{\lambda,1} \|_{L^2_\Phi})
\end{equation*}
for $v_\lambda$ supported in $B(0,1) \times \{ |\xi_0| \sim |\xi'| \sim 1 \}$. This inequality is a consequence from
\eqref{eq:ScalarReducedEstimate}. Indeed, for any component $v_{\lambda,1,k}$, $k=1,2,3$, the estimate \eqref{eq:ScalarReducedEstimate} yields
\begin{equation*}
 \| T_\lambda^* v_{\lambda,1,k} \|_{L^p L^q} \lesssim \lambda^{\rho} ( \lambda^{\frac{2-s}{4}} \| v_{\lambda,1,k} \|_{L^2_\Phi} + \lambda^{\frac{2+s}{4}} \| (Z_{\varepsilon,\mu} \frac{\varepsilon}{\varepsilon_1 \varepsilon_2 \varepsilon_3} (M_E - \xi_0^2) v_{\lambda,1})_k \|_{L^2_\Phi} )
\end{equation*}
because $Z_{\varepsilon,\mu} \frac{\varepsilon}{\varepsilon_1 \varepsilon_2 \varepsilon_3} (M_E - \xi_0^2) = q(x,\xi)$. The assertion easily follows.
\end{proof}

It thus suffices to show for scalar $v_\lambda \in L^2_\Phi$ supported in $\{ \frac{1}{4} \leq |\xi| \leq 4 \}$ that the estimate
\begin{equation}
\label{eq:ScalarEstiamte}
\| T_\lambda^* v_\lambda \|_{L^p L^q} \lesssim \lambda^\rho ( \lambda^{\frac{2-s}{4}} \| v_\lambda \|_{L^2_\Phi} + \lambda^{\frac{2+s}{4}} \| q(x,\xi) v_\lambda \|_{L^2_\Phi} )
\end{equation}
holds true. We consider the operator
\begin{equation*}
W_\lambda = T_\lambda^* \frac{a(x,\xi)}{\lambda^{-\frac{s}{4}} + \lambda^{\frac{s}{4}} |q(x,\xi)|}
\end{equation*}
with $a \in C^\infty_c$ localizing to $B(0,1) \times \{ \frac{1}{4} \leq |\xi| \leq 4 \}$. Hence it is enough to establish
\begin{equation*}
\| W_\lambda \|_{L_{\Phi}^2 \to L^p L^q} \lesssim \lambda^{\rho + \frac{1}{2}}.
\end{equation*}
By the $TT^*$-argument, we can likewise prove the estimate
\begin{equation*} 
\| \tilde{V}_\lambda \|_{L^{p'} L^{q'} \to L^p L^q} \lesssim \lambda^{1 +2 \rho}, \quad \tilde{V}_\lambda = T_\lambda^* \frac{a^2(x,\xi) \Phi(\xi)}{(\lambda^{-\frac{s}{4}} + \lambda^{\frac{s}{4}} |q(x,\xi)|)^2} T_\lambda.
\end{equation*}

For this estimate it is crucial to understand the curvature properties of $\{ \xi \in \R^4: q(x,\xi) = 0 \}$.
Since $q(x,\xi)$ is $4$-homogeneous in $\xi$, we have
\begin{equation*}
q(x,\xi) = \xi_0^4 q(x,1,\xi'/\xi_0).
\end{equation*}
So it suffices to consider the Fresnel surface $S= \{ \xi' \in \R^3: q(x,1,\xi') = 0 \}$. The Fresnel surface is classical and
has been studied since the 19th century (cf. \cite{Darboux1993,Knoerrer1986,FladtBaur1975}). We refer to detailed computations
laid out in \cite{MandelSchippa2022}
and recall properties of the Fresnel surface $S$. We can consider $S$ as a union of three components $S = S_1 \cup S_2 \cup S_3$ with
\begin{itemize}
\item $S_1$ describes a regular bounded surface with Gaussian curvature bounded from below and above,
\item $S_2$ is a neighbourhood of the Hamilton circles, which is regular, but has one principal curvature vanishing along the Hamilton circles,
\item $S_3$ is a neighbourhood of four conical singularities.
\end{itemize}
We redenote $a^2(x,\xi)$ by $a(x,\xi)$ for the sake of brevity, and split $a = a_1 + a_2 + a_3$  and the operator
$V_\lambda = V_{\lambda,1} + V_{\lambda,2} + V_{\lambda,3}$ according to the above components.
We establish the three estimates separately, which is meant to highlight a different dispersive behavior. To see that the uniformity of the curvature bounds follows from the separability condition \eqref{eq:SeparationEigenvalues}, we revisit the analysis of the
Fresnel surface provided in \cite[Section~3]{MandelSchippa2022}.

First, the analysis of Fresnel's surface for $|\xi_0| \sim |\xi'| \sim 1$
\begin{equation*}
 S_{\xi_0} = \{ \xi' \in \R^3: p(\xi_0,\xi') = - \xi_0^2 (\xi_0^4 - \xi_0^2 q_0(\xi') + q_1(\xi')) = 0 \}
\end{equation*}
is reduced to the standard form $\mu_1 = \mu_2 = \mu_3 = \xi_0 = 1$ by means of the substitution
\begin{equation*}
 \lambda_i = \frac{\xi_i}{\xi_0 \sqrt{\mu_{i+1} \mu_{i+2}}} \qquad (i =1,2,3).
\end{equation*}
Note that the change of coordinates is as smooth as $\mu$. We use cyclic notation in the following, i.e., $\mu_4 := \mu_1$ and $\mu_5 := \mu_2$, likewise for $\varepsilon_i$. In the following let
$\varepsilon_i:= \frac{\varepsilon_i}{\mu_i}$ (the right-hand side refers to the original quantities).

We shall analyze $S^* = \{ \lambda \in \R^3: 1 - q_0^*(\lambda) + q_1^*(\lambda) = 0 \}$ with $q_0^*$ and $q_1^*$ denoting the reduced polynomials $q_0$ and $q_1$ for
$\mu_1 = \mu_2 = \mu_3 = 1$. We find that the coordinates of the four singular points are given by
\begin{equation*}
 \lambda_j^2 = \frac{\varepsilon_{j+2} (\varepsilon_j - \varepsilon_{j+1})}{\varepsilon_j - \varepsilon_{j+2}}, \quad \lambda_{j+1} = 0, \quad \lambda_{j+2}^2 = \frac{\varepsilon_j (\varepsilon_{j+2} - \varepsilon_{j+1})}{\varepsilon_{j+2} - \varepsilon_j }
\end{equation*}
for $\varepsilon_{j+1} \in \langle \varepsilon_j, \varepsilon_{j+2} \rangle$, which denotes $\varepsilon_{j+1} \in [\varepsilon_j, \varepsilon_{j+2}] \cup [\varepsilon_{j+2},\varepsilon_j]$.
By uniform ellipticity and the separation condition \eqref{eq:SeparationEigenvalues}, the singular points vary as smoothly in $x \in \R^4$
as $(\varepsilon,\mu)$.

To analyze  the regularity of the curvature in $(\varepsilon,\mu)$, we recall the parametrization of
the regular part of the Fresnel surface given by
\begin{equation*}
 s = \lambda_1^2 + \lambda_2^2 + \lambda_3^2, \quad \varepsilon_1 \lambda_1^2+ \varepsilon_2 \lambda_2^2 + \varepsilon_3 \lambda_3^2 = \varepsilon_1 \varepsilon_2 \varepsilon_3 t^{-1},
\end{equation*}
and
\begin{equation*}
 \alpha(s,t) = \frac{(t-s) \varepsilon_1 \varepsilon_2 \varepsilon_3}{s^2 t - (\varepsilon_1 + \varepsilon_2 + \varepsilon_3) st + (\varepsilon_1 \varepsilon_2 + \varepsilon_1 \varepsilon_3 + \varepsilon_2 \varepsilon_3) t - \varepsilon_1 \varepsilon_2 \varepsilon_3}.
\end{equation*}
In $(s,t)$ coordinates, we have for the Gaussian curvature
\begin{equation*}
 K(s,t) = 
 \frac{(\alpha(s,t) - \varepsilon_1) (\alpha(s,t) - \varepsilon_2) (\alpha(s,t) - \varepsilon_3) }{\alpha(s,t) (s- \varepsilon_1) (s- \varepsilon_2)(s - \varepsilon_3) }
\end{equation*}
and for the mean curvature
\begin{equation*}
\begin{split}
 K_m(s,t) &= - \frac{1}{2} \big( \frac{s}{\sqrt{\alpha}} K(s,t) - \frac{1}{\sqrt{\alpha}} \big( \frac{(\alpha- \varepsilon_1)(\alpha - \varepsilon_2)}{(s-\varepsilon_1)(s-\varepsilon_2)} + \frac{(\alpha- \varepsilon_2)(\alpha-\varepsilon_3)}{(s-\varepsilon_2)(s-\varepsilon_3)} \\
 &\qquad + \frac{(\alpha- \varepsilon_1)(\alpha-\varepsilon_3)}{(s-\varepsilon_1)(s-\varepsilon_3)} \big) \big).
 \end{split}
\end{equation*}
The limiting case $s=t$ corresponds to singular points. For separated $s$ and $t$ we have $\alpha \geq d > 0$ and the uniformity of the curvature bounds follows. For the analysis of the
curvature close to the singular points, we rescale a dyadic decomposition away from the singular points to unit distance (see below), for which the previous arguments apply.

\smallskip

We now  conclude the proof of the  Strichartz estimates in Theorem \ref{thm:LocalStrichartzEstimatesFullyAnisotropic} for sharp Strichartz exponents close to the endpoint.
Note that the following proposition implies Strichartz estimates for the contribution of $V_{\lambda,1}$ like in \eqref{eq:ScalarEstiamte} for the wider range $\frac{1}{p} + \frac{1}{q} \leq \frac{1}{2}$ and $2<p < q \leq \infty$ by Sobolev embedding. We recall
the defintion of $V_{\lambda,i}$ in \eqref{eq:TT*Estimate} below.
\begin{proposition}
\label{prop:OperatorEstimatesTT^*}
Let $\rho = 3 \big( \frac{1}{2} - \frac{1}{q} \big) - \frac{1}{p}$ and $2 < p < q \leq \infty$. The estimate
\begin{equation}
\label{eq:LocalizedTT^*Estimate}
\| V_{\lambda,i} \|_{L^{p'} L^{q'} \to L^p L^q} \lesssim \lambda^{1+2\rho} 
\end{equation}
holds true, if
\begin{itemize}
\item $i=1$ and $\frac{1}{p} + \frac{1}{q} = \frac{1}{2}$,
\item $i \in \{2,3\}$ and $\frac{2}{p} + \frac{1}{q} = \frac{1}{2}$.
\end{itemize}
\end{proposition}
For $i=1,2$, the assumption $2<p < q \leq \infty$ is not necessary because we can interpolate with \eqref{eq:LocalizedTT^*Estimate}
for $p=\infty$ and $q=2$. We show this case. In the proof we will see that if fails for $i=3$ at the singular points
because these satisfy $\partial_\xi q(x,\xi_x) = 0$. 

After factoring out the FBI transform in the $x'$ variables by virtue of its $L^2$-boundedness, we it remains to prove the boundedness of
\begin{equation}
\label{eq:EnergyEstimateTT^*}
\lambda^{-1} T_{\lambda,0}^* a_i(x,\xi) \Phi(\xi) \delta_{q=0} T_{\lambda,0}: L_t^1 L_{x',\xi'}^2 \to L_t^\infty L_{x',\xi'}^2
\end{equation}
for the restricted FBI transform $T_{\lambda,0}: L^2(\R^n) \to L^2_{\Phi}(\R^2 \times \R^{n-1})$ given by
\begin{equation*}
T_{\lambda,0} f(x_0,\xi_0,y') = 2^{-\frac{1}{2}} \pi^{-\frac{3}{4}} \lambda^{\frac{3}{4}} \int e^{-\frac{\lambda}{2}(z-y_0)^2} f(y_0,y') dy_0
\end{equation*} 
with $\Phi(x_0,\xi_0,x') = e^{-\lambda |\xi_0|^2}$. (Compare the proof of Proposition \ref{prop:OperatorEstimatesTT^*} below.)
Hence, to establish \eqref{eq:EnergyEstimateTT^*}, we have  to show the boundedness of the kernel
\begin{equation*}
K(y_0,\bar{y}_0) =  \lambda^{\frac{1}{2}} \int_{\R} dx_0 e^{-\frac{\lambda}{2}(x_0-y_0)^2} e^{-\frac{\lambda}{2}(x_0-\bar{y}_0)^2} \int_{\R} d\xi_0 e^{i \lambda \xi_0.(y_0-\bar{y}_0)} a_0(x_0,\xi_0) \delta_{q=0}.
\end{equation*}
It is enough to check the hyperbolicity condition 
\begin{equation}
\label{eq:Hyperbolicity}
\partial_{\xi_0} q(x,\xi) \neq 0 \quad \text{for \ } a_i(x,\xi) \neq 0, \;\; i=1,2.
\end{equation}
Indeed, from \eqref{eq:Hyperbolicity} follows
\begin{equation*}
\big| \int_{\R} d\xi_0 e^{i\lambda \xi_0.(y_0-\bar{y}_0)} a_0(x_0,\xi_0) \delta_{q=0} \big| \lesssim 1
\end{equation*}
and
\begin{equation*}
\int_{\R} dx_0 e^{-\frac{\lambda}{2}(x_0-y_0)^2} e^{-\frac{\lambda}{2}(x_0-\bar{y}_0)^2} \lesssim \lambda^{- \frac{1}{2}},
\end{equation*}
yielding \eqref{eq:LocalizedTT^*Estimate} for $(p,q)=(\infty,2)$.
For the remaining proof of \eqref{eq:Hyperbolicity}, we consider
\begin{equation*}
\partial_{\xi_0} q = 4 \xi_0^3 - 2 \xi_0 q_0 = 0.
\end{equation*}
This implies $q_0 = 2 \xi_0^2$. Since $q(\xi_0,\xi')= 0$, we find $q_1 = \xi_0^4$ and so $q_1= ( \frac{q_0}{2})^2$.
By homogeneity of $q$, it is enough to consider $\xi_0 = 1$. Then we use again the parametrisations of \cite{MandelSchippa2022}:
\begin{align*}
s=\xi_1^2 + \xi_2^2 + \xi_3^2, &\qquad \varepsilon_1 \xi_1^2 + \varepsilon_2 \xi_2^2 + \varepsilon_3 \xi_3^2 = \varepsilon_1 \varepsilon_2 \varepsilon_3 t^{-1}, \\
q_1^*(\xi) = st^{-1}, &\qquad q_0^*(\xi) = 1 + st^{-1}.
\end{align*}

Let $b = st^{-1}$. From $q_1= ( \frac{q_0}{2} )^2$ follows $b = \big( \frac{1+b}{2} \big)^2$ and hence $(b-1)^2 = 0$.
In the regular part of $S$ we have $s<t$ or $t<s$ corresponding to the two sheets of $S$, where $t=s$ formally corresponds to the singular
points. This ensures hyperbolicity within $\text{supp}(a_1) \cup \text{supp}(a_2)$.
However, within $\text{supp} (a_3)$ the argument fails. We can thus prove \eqref{eq:LocalizedTT^*Estimate} only close to the endpoint,
which yields the corresponding Strichartz estimates.

\begin{proof}[Proof of Proposition \ref{prop:OperatorEstimatesTT^*}]
For $i=1,2$, the result is shown by a straight-forward adaptation of Tataru's arguments \cite{Tataru2000}. We shall be brief, but repeat
the argument because in the most difficult case $i=3$ we will elaborate on it. To prove a uniform bound for the family of operators
\begin{equation}
\label{eq:TT*Estimate}
\lambda^{-1-2\rho} V_{\lambda,i}= \lambda^{-1-2\rho}T_\lambda^* \frac{a_i(x,\xi) \Phi(\xi)}{(\lambda^{-\frac{s}{4}} + \lambda^{\frac{s}{4}} |q(x,\xi)|)^2} T_\lambda: L^{p'} L^{q'} \to L^p L^q,
\end{equation}
we use the integrability of the weight
\begin{equation*}
\int \frac{1}{(\lambda^{-\frac{s}{4}} + \lambda^{\frac{s}{4}} |q|)^2} dq \lesssim 1
\end{equation*}
and can reduce to level sets $\delta_{q(x,\xi) = c}$ by foliation. It is enough to consider $c=0$ as $\{ \xi \in \R^4: q(x,\xi) = 0 \}$
is a regular surface within the supports of $a_i$, $i=1,2$, and the supports can be properly foliated.

For the proof of the non-endpoint estimate \eqref{eq:LocalizedTT^*Estimate} in case $2 < p < q \leq \infty$, we consider the analytic family of operators
\begin{equation*}
V_{\lambda,i}^\theta = \lambda^{-\theta(2 \rho + 1)} \frac{e^{(\theta-1)^2}}{\Gamma(1-\theta)} T_\lambda^* a_i(x,\xi) \Phi(\xi) q^{-\theta}(x,\xi) T_\lambda
\end{equation*}
such that $\lambda^{-1-2 \rho} V_{\lambda,i} = V_{\lambda,i}^1$.

The estimate follows from interpolation between $\lambda$-independent bouns of
\begin{align*}
&V_{\lambda,i}^\theta: L^2 \to L^2, \qquad \qquad \quad \ \, \Re \theta = 0, \\
&V_{\lambda,i}^\theta: L^{p_1'} L^1 \to L^{p_1} L^\infty, \qquad \Re \theta = \theta_1,
\end{align*}
where $p_1,\theta_1$ are chosen so that the following points are collinear:
\begin{equation}
\label{eq:InterpolationPoints}
\big( \frac{1}{2}, \frac{1}{2}, 0 \big), \qquad \big( \frac{1}{p}, \frac{1}{q}, 1 \big), \qquad \big( \frac{1}{p_1},0, \theta_1 \big).
\end{equation}

The $L^2$-estimate is immediate from the $L^2$-mapping properties of the FBI transform. For the $L^{p_1'} L^1 \to L^{p_1} L^\infty$ bound, we compute the kernel $K^\theta_{\lambda,i}$ of $ V^\theta_{\lambda,i}$ as
\begin{equation*}
\begin{split}
K^\theta_{\lambda,i}(y,\bar{y}) &= \lambda^{6- \theta(2\rho + 1)} \int_{\R^4} dx e^{-\frac{\lambda}{2}(y-x)^2} e^{-\frac{\lambda}{2}(\bar{y}-x)^2} \\
&\quad \times \frac{e^{(\theta-1)^2}}{\Gamma(1-\theta)} \int_{\R^4} d\xi e^{i\lambda(y-\bar{y}).\xi} q^{-\theta}(x,\xi) a_i(x,\xi).
\end{split}
\end{equation*}

We have the well-known oscillatory integral estimate (cf.\ \cite[Chapter~IX,~Section~1.2.3]{Stein1993})
\begin{equation}
\label{eq:OscillatoryIntegralDecayAnalyticFamily}
\big| \frac{e^{(\theta-1)^2}}{\Gamma(1-\theta)} \int_{\R^d} e^{iy.\xi} a_i(x,\xi) q^{-\theta}(x,\xi) d\xi \big| \lesssim (1+|y|)^{(\Re \theta-1) - \frac{k_i}{2}},
\end{equation}
where $k_1 = 2$ and $k_2 = 1$ are the numbers of principal curvatures of the regular surface $\{q(x,\xi) = 0\}$\footnote{The surface is
regular in $\xi$ within the support of $a_i$, $i=1,2$.} bounded from below. Estimate \eqref{eq:LocalizedTT^*Estimate} now follows in the
same spirit as in \cite{Tataru2000}. We give the details only for $i=2$ because the case $i=1$ is covered in \cite[pp.~363-364]{Tataru2000}
verbatim.  For $i=2$, we find the kernel bound
\begin{equation*}
|K^{\theta}_{\lambda,i}(y,\bar{y})| \lesssim \lambda^{4- \Re \theta (2\rho +1)} (1+\lambda |y-\bar{y}|)^{\Re \theta - \frac{3}{2}}.
\end{equation*}
The relations
\begin{equation*}
\theta - \frac{3}{2} = -\frac{2}{p} - \frac{1}{q}, \quad 4- \theta(2\rho + 1) = \frac{2}{p} + \frac{6}{q},
\end{equation*}
hold for the first two points in \eqref{eq:InterpolationPoints}. So they hold for the third, which gives
\begin{equation}
\label{eq:RelationsB}
\theta_1 - \frac{3}{2} = - \frac{2}{p_1}, \quad 4 - \theta_1(2 \rho +1) = \frac{2}{p_1}.
\end{equation}
Consequently, we have proved the kernel estimate
\begin{equation*}
|K^{\theta}_{ \lambda, i}(y,\bar{y})| \lesssim |y-\bar{y}|^{-\frac{2}{p_1}}, \qquad \Re \theta = \theta_1.
\end{equation*}
The Hardy--Littlewood--Sobolev inequality then yields \eqref{eq:LocalizedTT^*Estimate} for $i=2$.

\smallskip
 
We turn to the most difficult case $i=3$. Since there are four isolated singular points  $\xi_x'$, it suffices to estimate the contribution
of each one separately. We make the  decomposition $a_3(x,\xi) = \sum_{k \geq 0} a_{3,k}(x,\xi)$ where
\begin{equation*}
a_{3,k}(x,1,\xi') = a_{3,0}(x,1,(\xi'-\xi'_x) 2^k)
\end{equation*}
 localizes smoothly to a $2^{-k}$-annulus around $\xi_x'$. This cutoff is extended $1$-homoge\-neously in $\xi_0$, i.e.,
 $a_{3,k}(x,\xi_0,\xi') = \xi_0 a_{3,k}(x,1,\xi'/\xi_0)$. This is possible because $|\xi_0| \in [1/4,4]$. The kernel $K_k$ of
\begin{equation*}
A_{\lambda, k} = T_\lambda^* \frac{a_{3,k}(x,\xi) \Phi(\xi)}{(\lambda^{-\frac{s}{4}} + \lambda^{\frac{s}{4}} |q(x,\xi)|)^2} T_\lambda
\end{equation*}
is given by
\begin{equation*}
K_k(y,\bar{y}) = \lambda^{6-(2\rho +1)} \int_{\R^4} dx e^{-\frac{\lambda(y-x)^2}{2}} e^{-\frac{\lambda(\bar{y}-x)^2}{2}} \int_{\R^4} d\xi \frac{e^{i\lambda \xi.(y-\bar{y})} a_{3,k}(x,\xi_0,\xi')}{(\lambda^{-\frac{s}{4}} + \lambda^{\frac{s}{4}} |q(x,\xi)|)^2}.
\end{equation*}
By homogeneity of $q(x,\xi_0,\xi')$ and $a_3(x,\xi_0,\xi')$ in $\xi_0$, we find
\begin{equation*}
\frac{a_{3,k}(x,\xi_0,\xi') a_0(\xi_0)}{ (\lambda^{-\frac{s}{4}} + \lambda^{\frac{s}{4}} |q(x,\xi_0,\xi')|)^2} = \frac{\xi_0 a_{3,k}(x,1,\tilde{\xi}') a_0(\xi_0)}{\xi_0^8 ( \lambda^{-\frac{s}{4}} \xi_0^{-4} + \lambda^{\frac{s}{4}} |q(x,1,\tilde{\xi}')|)^2}
\end{equation*}
with $\tilde{\xi}' = \xi'/\xi_0$ and a suitable smooth cutoff $a_0$.

Furthermore, we carry out a Taylor expansion of $q(x,1,\cdot)$ around $\xi_x'$ to find
\begin{equation*}
q(x,1,\xi'+\xi_x') = \frac{\langle \partial^2_{\xi' \xi'} q(x,1,\xi_x') \xi', \xi' \rangle }{2} + O(|\xi'|^3),
\end{equation*}
recalling that
\begin{equation*}
q(x,1,\xi_x')= 0, \quad \nabla_{\xi'} q(x,1,\xi_x') = 0.
\end{equation*}
We write
\begin{equation*}
\begin{split}
q(x,1,\tilde{\xi}') =  q(x,1,2^{-k} \xi' + \xi_x') &= 2^{-2k} \big( \frac{\langle \partial^2_{\xi' \xi'} q(x,1,\xi_x') \xi', \xi' \rangle}{2} + O(2^{-k} |\xi'|^3) \big) \\
&=: 2^{-2k} q_k(x,1,\xi').
\end{split}
\end{equation*}
For the analysis of the kernel we perform the change of variables $\tilde{\xi}' = \tilde{\tilde{\xi}}' + \xi_x'$ and $\xi' = 2^k \tilde{\tilde{\xi}}'$ to find
\begin{equation*}
\begin{split}
K_k(y,\bar{y}) &= \lambda^{6-(2\rho+1)}\!\! \int_{\R^4}\!\! dx e^{-\frac{\lambda}{2}(y-x)^2} e^{-\frac{\lambda}{2}(\bar{y}-x)^2} \!\!
  \int_{\R} \!d\xi_0 \frac{e^{i \lambda \xi_0.(y_0-\bar{y}_0)} e^{i \lambda \xi_0 \xi_x'.(y'-\bar{y}')}  a_0(\xi_0)}{\xi_0^4} \\
&\quad \times \int_{\R^3} d\xi' \frac{2^{-3k} a_{3,0}(x,1,\xi') e^{i \lambda 2^{-k} \xi_0 \xi'.(y'-\bar{y}')}}{(\lambda^{-\frac{s}{4}} \xi_0^{-4} + \lambda^{\frac{s}{4}} 2^{-2k} |q_k(x,1,\xi')|)^2} \\
&= \lambda^{6-(2\rho+1)}\!\! \int_{\R^4} \!\!dx e^{-\frac{\lambda}{2}(y-x)^2} e^{-\frac{\lambda}{2}(\bar{y}-x)^2} \!\!\int_{\R}\! d\xi_0
  \frac{e^{i \lambda \xi_0.(y_0-\bar{y}_0)} e^{i \lambda \xi_0 \xi_x'.(y'-\bar{y}')}  a_0(\xi_0)}{\xi_0^4} \\
&\quad \times \int_{\R^3} d\xi' \frac{2^{k} a_{3,0}(x,1,\xi') e^{i \lambda 2^{-k} \xi_0 \xi'.(y'-\bar{y}')}}{(\lambda^{-\frac{s}{4}} \xi_0^{-4} 2^{2k} + \lambda^{\frac{s}{4}} |q_k(x,1,\xi')|)^2}.
\end{split}
\end{equation*}

The weight $(2^{2k} \lambda^{-\frac{s}{4}} \xi_0^{-4} + \lambda^{\frac{s}{4}} |q_k(x,1,\xi')|)^{-2}$ is integrable with respect to the level sets, and we have
\begin{equation*}
\int_0^\infty \frac{dq}{(2^{2k} \lambda^{-\frac{s}{4}} + q \lambda^{\frac{s}{4}})^2} \lesssim 2^{-2k}.
\end{equation*}
After foliation over level sets $q(x,\xi)=c$, we are thus left with the operator having the kernel
\begin{align*}
K^1_k(y,\bar{y})&= \lambda^{6-(2\rho+1)} \!\!\int_{\R^4} dx e^{-\frac{\lambda}{2}(y-x)^2} e^{-\frac{\lambda}{2}(\bar{y}-x)^2}\!\!
\int_{\R} d\xi_0 \frac{e^{i \lambda \xi_0.(y-\bar{y})} e^{i \lambda \xi_0 \xi_x'.(y'-\bar{y}')}  a_0(\xi_0)}{\xi_0^4} \notag\\
&\quad \times \int_{\R^3} d\xi' 2^{-k} a_{3,0}(x,1,\xi') e^{i \lambda 2^{-k} \xi_0 \xi'.(y'-\bar{y}')} \delta_{\tilde{q}_k(x,1,\xi') = 0}.
\end{align*}
We estimate it  by interpolation using the analytic family of operators $V^\theta_{k,\lambda}$ given by the kernels
\begin{equation}
\label{eq:AnalyticFamily}
\begin{split}
K^\theta_{\lambda,k}&= \lambda^{6-\theta(2\rho+1)} \int_{\R^4} dx e^{-\frac{\lambda}{2}(y-x)^2} e^{-\frac{\lambda}{2}(\bar{y}-x)^2}
   \int_{\R} d\xi_0 e^{i \lambda \xi_0.(y_0-\bar{y}_0)} \frac{e^{i\lambda \xi_0 \xi_x'.(y'-\bar{y}')} a_0(\xi_0)}{\xi_0^4} \\
&\quad \times \frac{e^{(\theta-1)^2}}{\Gamma(1-\theta)} \int_{\R^3} d\xi' 2^{-k} a_{3,0}(x,1,\xi') e^{i \lambda 2^{-k} \xi_0 \xi'.(y'-\bar{y}')}  (q_k+i0)^{-\theta}.
\end{split}
\end{equation}
We show+ that the corresponding operators uniformly bounded on the spaces
\begin{equation*}
\begin{split}
&V^\theta_{\lambda,k}: L^2 \to L^2 \qquad \quad \; \quad \; \Re \theta = 0, \\
&V^\theta_{\lambda,k}: L^{p_1'} L^1 \to L^{p_1} L^\infty \quad \Re \theta = \theta_1,
\end{split}
\end{equation*}
where $p_1$, $\theta_1$ are chosen so that the points
\begin{equation*}
\big( \frac{1}{2}, \frac{1}{2}, 0 \big), \quad \big( \frac{1}{p}, \frac{1}{q},1 \big), \quad \big( \frac{1}{p_1}, 0 , \theta_1 \big)
\end{equation*}
are collinear.
\begin{itemize}
 \item Estimate of $V^\theta_{\lambda,k}: L^2 \to L^2, \quad \Re \theta = 0$:
 After inverting the changes of variables and using the $L^2$-boundedness of the FBI transform, we find
 \begin{equation}
 \label{eq:BadEnergyEstimate}
  \| V^\theta_{\lambda,k} \|_{L^2 \to L^2} \lesssim 2^{2k}.
 \end{equation}
\item Estimate of $V^\theta_{\lambda,k}: L^{p_1'} L^1 \to L^{p_1} L^\infty$, $\Re \theta = \theta_1$:
Inequality \eqref{eq:OscillatoryIntegralDecayAnalyticFamily} implies the kernel estimates
\begin{equation*}
 |K_{\lambda,k}^\theta | \lesssim \lambda^{4- \theta (2 \rho +1)} 2^{-k} (1 + \lambda 2^{-k} |y'-\bar{y}'|)^{\Re \theta - \frac{3}{2}}.
\end{equation*}
By the relations \eqref{eq:RelationsB} and provided that $|y_0-\bar{y}_0| \lesssim |y'-\bar{y}'|$, this gives
\begin{equation*}
|K^\theta_{\lambda,k}| \lesssim 2^{\frac{2k}{p_1}-k} |y-\bar{y}|^{-\frac{2}{p_1}}\le 2^{\frac{2k}{p_1}-k} |y_0-\bar{y}_0|^{-\frac{2}{p_1}}.
\end{equation*}
Since $p_1 \geq 4$, we can apply the Hardy--Littlewood--Sobolev inequality
On the other hand, if $|y_0 - \bar{y}_0| \gg |y' - \bar{y}'|$, we can integrate by parts in $\xi_0$ in \eqref{eq:AnalyticFamily},
which gains factors $(\lambda |y_0 - \bar{y}_0|)^{-1}$. So we obtain
\[\| V^\theta_{k,\lambda} \|_{L^{p_1'} L^1 \to L^{p_1} L^\infty}\lesssim  2^{\frac{2k}{p_1}-k}.\]
\end{itemize}
Interpolation yields
\begin{equation*}
\| V_{k,\lambda}^\theta \|_{L^{p'} L^{q'} \to L^p L^q} \lesssim \| V^\theta_{k,\lambda} \|_{L^2 \to L^2}^\alpha \| V^\theta_{k,\lambda} \|_{L^{p_1'} L^1 \to L^{p_1} L^\infty}^{1-\alpha}
\end{equation*}
with
\begin{equation*}
\big( \frac{1}{p}, \frac{1}{q},1 \big) = \alpha \big( \frac{1}{2},\frac{1}{2},0 \big) + (1-\alpha) \big( \frac{1}{p_1}, 0,\theta_1 \big)
\quad \Longrightarrow \quad \alpha = \frac{2}{q}.
\end{equation*}
These identities imply $\theta_1 = \frac{1/2}{1/2 - 1/q}$ and $(2k/p_1) ( 1 - 2/q) = k(2/p - 2/q)$, and thus  the operator norm
\begin{equation*}
\| V_{k,\lambda}^\theta \|_{L^{p'} L^{q'} \to L^p L^q} \lesssim 2^{k \big( \frac{6}{q} - 1 \big)} 2^{\frac{2k}{p_1} (1-\alpha)} \lesssim 2^{k \big( \frac{3}{q} - \frac{1}{2} \big)}.
\end{equation*}
This is summable for $q > 6$, but since $\frac{2}{p} + \frac{1}{q} = \frac{1}{2}$ and $p<q$ this is always the case under the assumptions.
\end{proof}
Note that \eqref{eq:BadEnergyEstimate} reflects the lack of hyperbolicity at the singular point. For this purpose, we only prove the bounds at the level of the scalar equation close to the endpoint. 
We remark that in two dimensions, it suffices to prove the (allowed) endpoint estimate $L_t^4 L_x^\infty$. We included the above arguments to show that the proof remains valid for a characteristic surface
$\tilde{S} = \{ q(\xi) = 0 \}$ in higher dimensions with the degeneracy $q(\xi_0) = 0$, $\nabla q(\xi_0) = 0$, and $\det (\partial^2 q(\xi_0)) \neq 0$.

\subsection{Short-time Strichartz estimates}
\label{subsection:ShorttimeStrichartz}
In the estimate \eqref{eq:StrichartzEstimateFullyAnisotropicC1} proved so far, there is no gain in the regularity of the homogeneous part
$\| u \|_{L^2}$ by the Strichartz estimate over Sobolev embedding. Nonetheless, by passing through short-time Strichartz estimates we show the claimed improvement (see also Bahouri--Chemin \cite{BahouriChemin1999}).
In the following we recast the smoothing in the inhomogeneous term $Pu$ as derivative gain by using frequency dependent time localization.
\begin{proposition}
\label{prop:ShorttimeStrichartz}
Let $(\varepsilon,\mu) \in C^s_x$ for $0<s \leq 1$ and $\partial(\varepsilon,\mu) \in L_t^2 L_{x'}^\infty$.
 Suppose that the estimate
 \begin{equation}
 \label{eq:SmoothingInhomogeneity}
  \| |D|^{-\rho+\frac{s-2}{4}} u \|_{(L_{t}^4 L^{\infty}_{x'})_2} \lesssim_{C} \| u \|_{L_x^2} + \| |D|^{\frac{s-2}{2}} Pu \|_{L^2_x} + \| | D |^{\frac{s-2}{4}-\frac{1}{2}} \rho_{em} \|_{L^2_x}
 \end{equation}
holds true with $C=C(\|(\varepsilon,\mu) \|_{C^s_x},\| \partial (\varepsilon, \mu) \|_{L_t^2 L_{x'}^\infty})$.
Let $\delta > 0$ and $T> 0$. We then obtain the inquality
\begin{equation}
\label{eq:ImprovementStrichartz}
\begin{split}
 \| \langle D' \rangle^{-\rho + \frac{s-2}{8} - \delta} u \|_{L^4(0,T;L^\infty_{x'})} &\leq C( \| u_0 \|_{L^2_{x'}} + \| Pu \|_{L^1_T L^2_{x'}} + \| \langle D' \rangle^{\frac{s-2}{8}-\frac{1}{2}} \rho_{em}(0) \|_{L^2_{x'}} \\
 &\quad + \| \langle D' \rangle^{\frac{s-2}{8}-\frac{1}{2}} \partial_t \rho_{em} \|_{L^1_T L^2_{x'}})
 \end{split}
\end{equation}
with $C = C( \| (\varepsilon,\mu) \|_{C_x^s}, \| \partial (\varepsilon,\mu) \|_{L_t^2 L_{x'}^\infty}, T, \delta)$.
\end{proposition}
\begin{proof}
 An application of \eqref{eq:SmoothingInhomogeneity} with $P_{\lambda}$ (recall this is the Maxwell operator with coefficients truncated at frequencies $\leq \lambda/8$) on $S_\lambda u$ yields 
\begin{equation}
 \label{eq:DyadicEstimate}
 \lambda^{-\rho+\frac{s-2}{4}} \| S_\lambda u \|_{L_{t}^4 L_{x'}^\infty} \lesssim \| S_\lambda u \|_{L_x^2} + \lambda^{\frac{s-2}{2}} \| P_{\lambda} S_\lambda u \|_{L^2_{t} L_{x'}^2} + \lambda^{\frac{s-2}{4}-\frac{1}{2}} \| S_\lambda \rho_{em} \|_{L^2}.
\end{equation}

For the proof of \eqref{eq:ImprovementStrichartz}, we use Minkowski's inequality to split
\begin{equation*}
 \begin{split}
  \| \langle D' &\rangle^{-\rho + \frac{s-2}{8} - \frac{\delta}{2}} u \|_{L_T^4 L_{x'}^\infty} \\
  &\leq \sum\nolimits_\lambda \| \langle D' \rangle^{-\rho + \frac{s-2}{8} - \frac{\delta}{2}} S_\lambda u \|_{L_T^4 L_{x'}^\infty} \\
  &\leq \sum\nolimits_\lambda \big[ \| \langle D' \rangle^{-\rho + \frac{s-2}{8} - \frac{\delta}{2}} S_\lambda S'_\lambda u \|_{L_T^4 L_{x'}^\infty} + \| \langle D' \rangle^{-\rho + \frac{s-2}{8} - \frac{\delta}{2}} S_\lambda^{\ll \tau} u \|_{L_T^4 L_{x'}^\infty} \big],
 \end{split}
\end{equation*}
where $S_\lambda^{\ll \tau}$ projects to the space-time frequencies $\{ \lambda \sim |\tau| \gg |\xi'| \}$. In this region the operator $P$
is elliptic, which gains one derivative for Lipschitz coefficients. The summand can then be bounded by Sobolev embedding. We refer to the
proof of \cite[Corollary~1.7]{SchippaSchnaubelt2022} for details on this estimate if $\partial (\varepsilon,\mu) \in L^2_{x_0} L^\infty_{x'}$.

For the first term
we want to apply \eqref{eq:DyadicEstimate}. Let $\chi \in C^\infty_c(-2,2)$ be radially decreasing with $\chi(t) = 1$ for $t \in [-1,1]$.
Note that passing to $\chi(\lambda^{\frac{2-s}{2}} t - n) u$ does not essentially change the Fourier support in time because the inverse
Fourier transform of $\chi_{\lambda^{\frac{2-s}{2}}}(t) = \chi(\lambda^{\frac{2-s}{2}} t - n)$ is essentially
supported in a $\lambda^{\frac{2-s}{2}}$-ball. We suppress dependence on $n$ in the following due to uniformity of the estimates.
We will apply \eqref{eq:DyadicEstimate} to $S'_\lambda \chi_{\lambda^{\frac{2-s}{2}}} u$ with
$I_\lambda = \text{supp}(\chi_{\lambda^{\frac{2-s}{2}}})$  and $I_\lambda^*=\{\chi_{\lambda^{\frac{2-s}{2}}}=1\}$.
At first, applying $\partial_t$ to $\chi_{\lambda^{\frac{2-s}{2}}}$ in $[P_\lambda,\chi_{\lambda^{\frac{2-s}{2}}}]$
we derive
\begin{equation*}
\lambda^{\frac{s-2}{2}} \| P_\lambda \chi_{\lambda^{\frac{2-s}{2}}} S_\lambda S'_\lambda u \|_{L_x^2}
  \lesssim \| S_\lambda S'_\lambda u \|_{L_t^2(I_\lambda; L^2_{x'})} + \lambda^{\frac{s-2}{2}} \| P_\lambda S_\lambda S'_\lambda u \|_{L_t^2(I_\lambda;L^2_{x'})} .
\end{equation*}
This gives
\begin{equation*}
\begin{split}
\lambda^{-\rho + \frac{s-2}{4} - \frac{\delta}{2}} \| S_\lambda S'_\lambda u \|_{L_t^4(I_\lambda^*;L_{x'}^\infty)}
 &\lesssim  \| S_\lambda' u \|_{L_t^2(I_\lambda; L^2_{x'})} + \lambda^{\frac{s-2}{2}} \| P_\lambda S_\lambda S'_\lambda u \|_{L_t^2(I_\lambda;L^2_{x'})} \\
&\quad + \lambda^{\frac{s-2}{4}- \frac{1}{2}} \| S'_\lambda \rho_{em} \|_{L_t^2(I_\lambda;L^2_{x'})}.
\end{split}
\end{equation*}
An application of H\"older's inequality on $\| S_\lambda' u \|_{L_t^2(I_\lambda;L_{x'}^2)}$ yields
\begin{equation*}
\begin{split}
 \lambda^{-\rho + \frac{s-2}{4} - \frac{\delta}{2}} \| S_\lambda S'_\lambda u \|_{L_t^4(I_\lambda^*;L^\infty_{x'})} &\lesssim \lambda^{\frac{s-2}{4}} \| S'_\lambda u \|_{L_t^\infty(I_\lambda; L^2_{x'})} + \lambda^{\frac{s-2}{2}} \| P_{\lambda} S_\lambda S'_\lambda u \|_{L_t^2(I_\lambda;L^2_{x'})} \\
 &\quad + \lambda^{\frac{s-2}{4}-\frac{1}{2}} \| S'_\lambda \rho_{em} \|_{L_{t}^2(I_\lambda;L^2_{x'})}.
 \end{split}
\end{equation*}
We sum this estimate over the $T \lambda^{\frac{2-s}{2}}$ disjoint intervals $I_\lambda^*$ partitioning $I = (0,T)$ in $\ell^4$ such that
\begin{equation*}
\begin{split}
 \lambda^{-\rho + \frac{s-2}{4} - \frac{\delta}{2}} \| S_\lambda S'_\lambda u \|_{L_T^4(I;L_{x'}^\infty)} &\lesssim_T \lambda^{\frac{2-s}{8} + \frac{s-2}{4}} \| S'_\lambda u \|_{L_t^{\infty} L^2_{x'}} + \lambda^{\frac{s-2}{2}} \|P_{\lambda} S_\lambda S'_\lambda u \|_{L_t^2 L^2_{x'}} \\
 &\quad + \lambda^{\frac{s-2}{4}-\frac{1}{2}} \| S'_\lambda \rho_{em} \|_{L_x^2}.
 \end{split}
\end{equation*}
Because of $\ell^2 \hookrightarrow \ell^4$ we do not lose powers of $\lambda$ when summing $P_\lambda S_\lambda S'_\lambda u$ and $S'_\lambda \rho_{em}$.
By commutator arguments using $\| \partial (\varepsilon,\mu) \|_{L^2_t L_{x'}^\infty} \lesssim 1$, cf.\ (4.4) in \cite{Tataru2002}, we find
\begin{equation*}
\| P_\lambda S_\lambda S'_\lambda u \|_{L_t^2 L_{x'}^2} \lesssim \| S_\lambda' P u \|_{L_t^2 L_{x'}^2} + \| S'_\lambda u \|_{L_t^\infty L_{x'}^2}
\end{equation*}
and hence,
\begin{align*}
\lambda^{-\rho + \frac{s-2}{8} - \frac{\delta}{2}} \| S_\lambda S'_\lambda u \|_{L_t^4(I;L^\infty_{x'})}
&\lesssim_T \| S'_\lambda u \|_{L^\infty_t L^2_{x'}} + \lambda^{\frac{3(s-2)}{8}} \| P u \|_{L^2_t L^2_{x'}} \\
   &\quad + \lambda^{\frac{s-2}{8}-\frac{1}{2} }\| S'_\lambda \rho_{em} \|_{L_x^2}.
\end{align*}
Summing over $\lambda$, we find the estimate
\begin{equation}
\label{eq:EnergyEstimateI}
\| \langle D' \rangle^{-\rho + \frac{s-2}{8} - \delta} u \|_{L_t^4(I;L^\infty_{x'})} \lesssim_{T,\delta,X} \| u \|_{L_t^\infty L^2_{x'}} + \| P u \|_{L^2_t L^2_{x'}} + \| \langle D' \rangle^{-\frac{1}{2}+\frac{s-2}{8}} \rho_{em} \|_{L^2_x}
\end{equation}
with $X = \| (\varepsilon,\mu) \|_{C^s_x} + \| \partial(\varepsilon,\mu)\|_{L_t^2 L_{x'}^\infty}$. (Note that we can treat frequencies
$\lambda \sim |\tau| \gg |\xi'|$ as above.)

The $L_t^2 L_{x'}^2$-norms are changed to $L_t^1 L_{x'}^2$-norms by the energy estimate and Duha\-mel's formula: An application of \eqref{eq:EnergyEstimateI} to homogeneous solutions gives
\begin{equation*}
\| \langle D' \rangle^{-\rho + \frac{s-2}{8} - \delta} u \|_{L_t^4(I;L^\infty_{x'})} \lesssim_{T,\delta,X} \| u_0 \|_{L^2_{x'}} + \| \langle D' \rangle^{-\frac{1}{2}+\frac{s-2}{8}} \rho_{em}(0) \|_{L^2_{x'}}.
\end{equation*}
Now we write the general function $u$ by Duhamel's formula $u(t) = U(t,0) u_0 + \int_0^t U(t,s) (Pu)(s) ds$ and note that $\nabla \cdot ((Pu)_1(s),(Pu)_2(s)) = (-\partial_t \rho_e(s),\partial_t \rho_m(s))$. We find after applying Minkowski's inequality
\begin{align}\label{eq:duhamel}
\| \langle D' \rangle^{-\rho + \frac{s-2}{8} - \delta} u \|_{L_t^4(I;L^\infty_{x'})} &\lesssim_{T,\delta,X} \| u_0 \|_{L^2_{x'}}
 + \| P u \|_{L^1_T L^2_{x'}} + \| \langle D' \rangle^{-\frac{1}{2}+\frac{s-2}{8}} \rho_{em}(0) \|_{L^2_{x'}} \notag\\
 &\quad + \| \langle D' \rangle^{-\frac{1}{2}+\frac{s-2}{8}} \partial_t \rho_{em} \|_{L^1_T L^2_{x'}}.
\end{align}
The proof is complete.
\end{proof}
Interpolation with the energy estimate from Proposition \ref{prop:EnergyEstimate}
\begin{equation*}
\| u \|_{L^\infty_T L^2_{x'}} \lesssim \| u(0) \|_{L^2_{x'}} + \| Pu \|_{L^1_T L^2_{x'}}
\end{equation*}
yields Theorem \ref{thm:StrichartzEstimatesFullyAnisotropic} on the sharp line $\frac{2}{p} + \frac{1}{q} = \frac{1}{2}$.
The general case follows from Sobolev embedding.
 In the next section we generalize the result in the Lipschitz case to diagonalizable coefficients.

\section{Reducing to the case of diagonal material laws}
\label{section:Reductions}
We now prove Theorem \ref{thm:StrichartzEstimatesFullyAnisotropicOffDiagonal} by transforming permittivity and permeability to diagonal matrices. In the following we consider $\varepsilon,\mu \in C^1(\R \times \R^3; \R^{3 \times 3}_{\text{sym}})$.
\subsection{Orthogonality transformations}
\label{subsection:OrthogonalityTransformations}
Suppose there is $\Phi \in C^1(\R^4; \R^{3 \times 3})$  with
\begin{equation}\label{eq:PermittivityPermeabilityTransformation}
\Phi^t(x) \Phi(x) = 1_{3 \times 3}, \quad \varepsilon^d = \Phi^t \varepsilon \Phi, \quad \mu^d = \Phi^t \mu \Phi.
\end{equation}
Write $\Phi = (\varphi_1 \; \varphi_2 \; \varphi_3)$. We use the transformations
\begin{equation*}
\tilde{\mathcal{E}} = \Phi^{-1} \mathcal{E}, \quad \tilde{\mathcal{H}} = \Phi^{-1} \mathcal{H}, \quad \tilde{\mathcal{J}}_k = \Phi^{-1} \mathcal{J}_k, \; k \in \{ e, m \}
\end{equation*}
to reduce the analysis to the case that $\varepsilon$ and $\mu$ are diagonal matrices. If $\mathcal{E}$, $\mathcal{H}$, $\mathcal{J}$ satisfy \eqref{eq:3dMaxwellEquations}, then $\tilde{\mathcal{E}}$, $\tilde{\mathcal{H}}$, $\mathcal{\tilde{J}}$ satisfy the system
\begin{equation}
\label{eq:TransformedMaxwell}
\left\{ \begin{array}{cl}
\partial_t (\varepsilon^d \tilde{\mathcal{E}}) &= \Phi^t \nabla \times (\Phi \tilde{\mathcal{H}}) - (\partial_t \Phi) \varepsilon^d \tilde{\mathcal{E}} - \tilde{\mathcal{J}}_e, \\
\partial_t (\mu^d \tilde{\mathcal{H}}) &= - \Phi^t \nabla \times (\Phi \tilde{\mathcal{E}}) - (\partial_t \Phi) \mu^d \tilde{\mathcal{H}} - \tilde{\mathcal{J}}_m, \\
\rho_e &= \nabla \cdot (\Phi \varepsilon^d \tilde{\mathcal{E}}), \quad \rho_m = \nabla \cdot (\Phi \mu^d \tilde{\mathcal{H}}).
\end{array} \right.
\end{equation}
We write $\tilde{u} = (\tilde{\mathcal{E}},\tilde{\mathcal{H}})$. The curl transforms as follows. We denote
\begin{equation*}
\nabla \times f = i \mathcal{C}(D) f, \quad \mathcal{C}(\xi') = (-\varepsilon^{ijk} \xi_k)_{ij}
, \quad \mathcal{C}(\xi') v = \xi' \times v.
\end{equation*}
The leading order term of $\Phi^t (\nabla \times (\Phi \cdot))$ can be written as
\begin{equation*}
\begin{split}
\begin{pmatrix}
\varphi_1^t \\ \varphi_2^t \\ \varphi_3^t
\end{pmatrix}
\begin{pmatrix}
\xi' \times \varphi_1 & \xi' \times \varphi_2 & \xi' \times \varphi_3
\end{pmatrix}
&= (\varphi_i \cdot (\xi' \times \varphi_j))_{ij} = (\xi' \cdot (\varphi_j \times \varphi_i))_{ij} \\
&= ( - \varepsilon^{ijk} \varphi_k(x) \cdot \xi')_{ij} =: (\mathcal{C}(\eta))_{ij}.
\end{split}
\end{equation*}
We let $\eta_k(x,\xi) = \varphi_k(x) \cdot \xi'$ and\footnote{It is important to work with operators in divergence form. We highlight this by the notation $a(D,x)$.} $\eta_k(D,x) = \sum_j \partial_j (\varphi_{kj}(x) \cdot)$. Consequently,
\begin{equation*}
\big( \Phi^t ( \nabla \times (\Phi \cdot)) \big)_{ij} = (- \varepsilon^{ijk} \varphi_k(x) \cdot \nabla_{x'} )_{ij} + b_{ij}(x)
\end{equation*}
with $\| b_{ij} \|_{L^\infty} \lesssim \| \Phi \|_{C^1}$. In the following we use that $\varepsilon^d$ and $\mu^d$ are diagonal as in Assumption \ref{AssumptionMaterialLaws}:
\begin{equation}
\label{eq:DiagonalMaterialLawsReduction}
\varepsilon^d(x)= \text{diag}(\varepsilon_1^d(x),\varepsilon^d_2(x),\varepsilon^d_3(x)), \quad \mu^d(x) = \text{diag}(\mu^d_1(x),\mu_2^d(x),\mu^d_3(x)).
\end{equation}

 For the electric charge we note that
\begin{equation*}
\rho_e = \nabla \cdot( \varepsilon \mathcal{E}) = \nabla \cdot (\Phi \varepsilon^d \tilde{u}^1) = \sum_{j,k = 1}^3 \partial_j (\varphi_{kj} \varepsilon_k^d \tilde{u}_k^1) =: \eta(D,x) \cdot \varepsilon^d \tilde{u}^1.
\end{equation*}
For the magnetic charge we find similarly
\begin{equation*}
 \rho_m = \eta(D,x) \cdot (\mu^d \tilde{u}^2).
\end{equation*}
We define
\begin{equation}
\label{eq:TransformedCharges}
 \tilde{\rho}_e = \eta(D,x) \cdot (\varepsilon^d \tilde{u}^1), \quad \tilde{\rho}_m = \eta(D,x) \cdot (\mu^d \tilde{u}^2).
\end{equation}
It follows
\begin{equation*}
 \begin{pmatrix}
  - \partial_t (\varepsilon^d \cdot) & \Phi^t (\nabla \times (\Phi(\cdot)) ) \\
  \Phi^t( \nabla \times (\Phi(\cdot)) ) & \partial_t (\mu^d \cdot) 
 \end{pmatrix}
= 
\begin{pmatrix}
 - \partial_t (\varepsilon^d \cdot) & \mathcal{C}(\eta(D,x)) \\
 \mathcal{C}(\eta(D,x)) & \partial_t (\mu^d \cdot)
\end{pmatrix}
+ R(x)
\end{equation*}
with $R(x) \in L^\infty$. Let
\begin{equation}
 \label{eq:TransformedMaxwellOperator}
 \tilde{P} =
 \begin{pmatrix}
  - \partial_t( \varepsilon^d \cdot) & \mathcal{C}(\eta(D,x)) \\
  \mathcal{C}(\eta(D,x)) & \partial_t (\mu^d \cdot)
 \end{pmatrix}
 .
\end{equation}
We have the following lemma:
\begin{lemma}
\label{lem:OrthogonalityTransformations}
Let $\Phi \in C^1(\R^4;\R^{3 \times 3})$ satisfy \eqref{eq:PermittivityPermeabilityTransformation} and $\varepsilon^d$, $\mu^d$ be like in \eqref{eq:DiagonalMaterialLawsReduction}.
Set $\tilde{u} = (\tilde{\mathcal{E}},\tilde{\mathcal{H}}) = (\Phi^{-1} \mathcal{E}, \Phi^{-1} \mathcal{H})$,
$\tilde{\mathcal{J}}_k = \Phi^{-1} \mathcal{J}_k$, $k \in \{e,m\}$, $\varepsilon^d = \Phi^t \varepsilon \Phi$, $\mu^d = \Phi^t \mu \Phi$, and
\begin{equation*}
\tilde{P} =
\begin{pmatrix}
- \partial_t (\varepsilon^d \cdot) & \mathcal{C}(\eta(D,x)) \\
\mathcal{C}(\eta(D,x)) & \partial_t (\mu^d \cdot)
\end{pmatrix}; \quad
\tilde{\rho}_{em} = (\eta(D,x) (\varepsilon^d \tilde{u}^1), \eta(D,x) (\mu^d \tilde{u}^2)).
\end{equation*}
Suppose that
\begin{equation}
\label{eq:HypothesisStrichartzTransformed}
\| |D|^{-\rho-\frac{1}{4}} \tilde{u} \|_{(L^4 L^{\infty})_2} \lesssim \| \tilde{u } \|_{L^2} + \| |D|^{-\frac{1}{2}} \tilde{P} \tilde{u} \|_{L^2_x} + \| |D|^{-\frac{3}{4}} \tilde{\rho}_{em} \|_{L^2}.
\end{equation}
Then, for $0<\delta < \frac{1}{4}$, the following estimate holds:
\begin{equation}
\label{eq:StrichartzEstimatesOriginal}
\| |D|^{-\rho-\frac{1}{4}-\delta} u \|_{(L^4 L^\infty)_2} \lesssim \| u \|_{L^2} + \| |D|^{-\frac{1}{2}} P u \|_{L^2} + \| |D|^{-\frac{3}{4}} \rho_{em} \|_{L^2}.
\end{equation}
The implicit constants depend on $\| (\Phi,\Phi^t) \|_{C^1}$, $\| (\varepsilon,\mu) \|_{C^1}$.
\end{lemma}
\begin{proof}
The low frequencies can be treated like in the proof of Proposition \ref{prop:ReductionFrequencyLocalizedEstimate} by Sobolev embedding.
For $0<\delta<\frac{1}{4}$, we thus obtain
\begin{equation*}
\| S_{0} |D|^{-\rho- \frac{1}{4}-\delta} u \|_{L^4 L^\infty} \lesssim \| u \|_{L^2}.
\end{equation*}
For high frequencies $M \gtrsim 1$, we write
\begin{equation*}
\begin{split}
\| S_M |D|^{-\rho - \frac{1}{4}-\delta} u \|_{(L^4 L^\infty)_2} &\sim M^{-\rho - \frac{1}{4}-\delta} \| S_M (\Phi^t \tilde{u}) \|_{L^4 L^\infty} \\
&\lesssim M^{-\rho-\frac{1}{4}-\delta}  \sum_{1 \leq K \leq M} \| S_M S'_K (\Phi^t \tilde{u}) \|_{L^4 L^\infty}  \\
&\quad + M^{-\rho - \frac{1}{4}-\delta} \| S_M S'_{0} (\Phi^t \tilde{u}) \|_{L^4 L^\infty}.
\end{split}
\end{equation*}
The second term is estimated
\begin{equation*}
M^{-\frac{3}{2}-\delta} \| S_M S'_{0} (\Phi^t u) \|_{L^4 L^\infty} \lesssim M^{-1} \| \Phi^t u \|_{L^2}
\end{equation*}
 by Bernstein's inequality, with easy summation in $M$. Hence, it suffices to estimate the spatial frequencies, which are greater than $1$. To this end,
 we also use paraproduct decompositions  both in  space and space-time frequencies. First, we write (with $X = L_{x_0}^4 L^\infty_{x'}$)
\begin{align}
\label{eq:ParaproductDecompositionI}
\| S_M S'_K (\Phi^t \tilde{u}) \|_X &\leq \| S_M((S'_{\sim K} \Phi^t) (S'_{\ll K} \tilde{u})) \|_X
       + \| S_M (S'_K ((S'_{\gtrsim K} \Phi^t) (S'_{\gtrsim K} \tilde{u})) \|_X \notag\\
&\quad + \| S_M ((S'_{\ll K} \Phi^t) (S'_{\sim K} \tilde{u})) \|_X.
\end{align}
Bernstein's inequality implies for  the first term in \eqref{eq:ParaproductDecompositionI}
\begin{equation*}
\begin{split}
\sum_{M \geq 1} &M^{-2\rho - \frac{1}{2}} \big( \sum_{1 \leq K \leq M} \| S_M (S'_K \Phi^t S'_{\ll K} \tilde{u} ) \|_{L_t^4 L_{x'}^\infty} \big)^2 \\
&\lesssim \sum_{M \geq 1} M^{-2\rho - \frac{1}{2}} M^{\frac{1}{2}} \big( \sum_{K \leq M} K^\frac32 \| S'_K \Phi^t S'_{\ll K} \tilde{u} \|_{L^2_{t,x'}} \big)^2 \\
&\lesssim \sum_{M \geq 1} M^{-\frac{5}{2}} \big( \sum_{K \leq M} K^\frac12 \| \nabla_{x'} S'_K \Phi^t \|_{L^\infty} \big)^2 \| \tilde{u} \|^2_{L^2} \lesssim \| \Phi^t \|_{C^1}^2 \| \tilde{u} \|_{L^2}^2.
\end{split}
\end{equation*}
The second term in \eqref{eq:ParaproductDecompositionI} can be estimated likewise:
\begin{equation*}
\begin{split}
\sum_{M \geq 1}& M^{-2\rho - \frac{1}{2}} \big( \sum_{1 \leq K \leq M} \| S_M S'_K (S'_{\gtrsim K} \Phi^t S'_{\gtrsim K} \tilde{u} \big)\|_{L_t^4 L_{x'}^\infty} \big)^2 \\
&\lesssim \sum_{M \geq 1} M^{-2\rho} \big( \sum_{1 \leq K \leq M} K^{\frac{3}{2}} \| S'_{\gtrsim K} \Phi^t \|_{L^\infty_{x'}} \| S'_{\gtrsim K} \tilde{u} \|_{L^2_x} \big)^2 \lesssim \| \Phi^t \|_{C^1}^2 \| \tilde{u} \|_{L^2}^2.
\end{split}
\end{equation*}

For the last term in \eqref{eq:ParaproductDecompositionI}, we additionally make the paraproduct decomposition in space-time frequencies
\begin{align}
\label{eq:ParaproductDecompositionII}
\| S_M( S'_{\ll K} \Phi^t S'_K \tilde{u}) \|_X & \le\| (S_M S'_{\ll K} \Phi^t) (S_{\ll M} S'_{\sim K} \tilde{u}) \|_X \notag\\
&\quad  + \| S_M (S_{\gtrsim M} S'_{\ll K} \Phi^t S_{\gtrsim M} S'_{\sim K} \tilde{u}) \|_X \\
&\quad + \| (S_{\ll M} S'_{\ll K} \Phi^t) (S_{\sim M} S'_{\sim K} \tilde{u}) \|_X.\notag
\end{align}
The first and second term in \eqref{eq:ParaproductDecompositionII} can be estimated using Bernstein's inequality. For the first term in \eqref{eq:ParaproductDecompositionII} we find
\begin{equation*}
\begin{split}
 \sum_{M \geq 1} &M^{-2\rho - \frac{1}{2}} \big( \sum_{1 \leq K \leq M} \| (S_M S'_{\ll K} \Phi^t) (S_{\ll M} S'_{\sim K} \tilde{u}) \|_{L^4 L^\infty} \big)^2 \\
&\lesssim \sum_{M \geq 1} M^{-2\rho - \frac{1}{2}} \big( \sum_{1 \leq K \leq M} K^{\frac{3}{2}} M^{\frac{1}{4}} \| S_M \Phi^t \|_{L^\infty} \| \tilde{u} \|_{L^2} \big)^2 \\
&\lesssim \sum_{M \geq 1} M^{-3} M^{3} M^{\frac{1}{2}} \| S_M \Phi^t \|_{L^\infty}^2 \| \tilde{u} \|_{L^2}^2
\lesssim \| \Phi^t \|_{C^1}^2 \| \tilde{u} \|_{L^2}^2.
\end{split}
\end{equation*}
The estimate of the second term in \eqref{eq:ParaproductDecompositionII} is given by
\begin{equation*}
\begin{split}
\sum_{M \geq 1}& M^{-2\rho - \frac{1}{2}} \big( \sum_{1 \leq K \leq M} \| S_M (S_{\gtrsim M} S'_{\ll K} \Phi^t S_{\gtrsim M} S'_{\sim K} \tilde{u} \big) \|_{L^4 L^\infty} \big)^2 \\
&\lesssim \sum_{M \geq 1} M^{-3} M^{\frac{3}{2}+\epsilon} \| \Phi \|_{C^1}^2 \| \tilde{u} \|_{L^2}^2
\lesssim \| \Phi \|_{C^1}^2 \| \tilde{u} \|_{L^2}^2.
\end{split}
\end{equation*}
 For the third term in \eqref{eq:ParaproductDecompositionII}, we find
\begin{equation*}
\begin{split}
 \sum_{M \geq 1} &M^{-2\rho-\frac{1}{2}-2\delta} \big( \sum_{1 \leq K \leq M} \| (S_{\ll M} S'_{\ll K} \Phi^t) (S_{\sim M} S'_{\sim K} \tilde{u}) \|_{L^4 L^\infty} \big)^2 \\
&\lesssim \| \Phi^t \|_{L^\infty}^2 \sum_{M \geq 1} M^{-2\delta} \big( \sum_{1 \leq K \leq M} \| \langle D \rangle^{- \rho - \frac{1}{4}} S_{\sim M} S'_{\sim K} \tilde{u} \|_{L^4 L^\infty} \big)^2 \\
&\lesssim \| \Phi^t \|_{L^\infty}^2 (\| \tilde{u} \|_{L^2} + \| |D|^{-\frac{1}{2}} \tilde{P} \tilde{u} \|_{L^2} + \| |D|^{-\frac{3}{4}} \tilde{\rho}_{em} \|_{L^2} )^2.
\end{split}
\end{equation*}
In the last step we used the hypothesis. It remains to show
\begin{equation*}
\| \tilde{u} \|_{L^2} + \| |D|^{-\frac{1}{2}} \tilde{P} \tilde{u} \|_{L^2} + \| |D|^{-\frac{3}{4}} \tilde{\rho}_{em} \|_{L^2} \lesssim \| u \|_{L^2} + \| |D|^{-\frac{1}{2}} P u \|_{L^2} + \| |D|^{-\frac{3}{4}} \rho_{em} \|_{L^2}.
\end{equation*}

We first note that  the definitions imply
\begin{equation*}
\| \tilde{u} \|_{L^2} + \| |D|^{-\frac{3}{4}} \tilde{\rho}_{em} \|_{L^2} \lesssim \| u \|_{L^2} + \| |D|^{-\frac{3}{4}} \rho_{em} \|_{L^2}.
\end{equation*}
 We turn to the estimate of the second term. By the above we can write $\tilde{P} \tilde{u} = P' \tilde{u} + R u$ with
\begin{equation*}
P'  =
\begin{pmatrix}
-\partial_t (\varepsilon^d \cdot) & \Phi^t (\nabla \times (\Phi \cdot)) \\
\Phi^t \nabla \times (\Phi \cdot) & \partial_t (\mu^d \cdot)
\end{pmatrix}, \qquad R \in L^\infty.
\end{equation*}
Moreover, $P' \tilde{u}= \tilde{R} \tilde{u} +\Phi^t\mathcal{J}$, where we write $\Phi^t$ instead of  $\text{diag}(\Phi^t,\Phi^t).$
Since $\tilde{P}$ is in divergence form, the low frequencies satisfy with implicit constant depending on the coefficients:
\begin{equation*}
\| |D|^{-\frac{1}{2}} S_{0} (\tilde{P} \tilde{u}) \|_{L^2} \lesssim \| \tilde{u} \|_{L^2}.
\end{equation*}
For the high frequencies, we plug in the above relations to find
\begin{equation*}
\begin{split}
\| |D|^{-\frac{1}{2}} S_{\geq 1} (\tilde{P} \tilde{u}) \|_{L^2_x} 
&\lesssim \| |D|^{-\frac{1}{2}} S_{\geq 1} (\tilde{R} \tilde{u} + R u + \Phi^t Pu )\|_{L^2} \\
&\lesssim ( \| \tilde{R} \|_{L^\infty} + \| R \|_{L^\infty}) \| u \|_{L^2}+ \| |D|^{-\frac{1}{2}} S_{\geq 1} (\Phi^t Pu ) \|_{L^2_x}.
\end{split}
\end{equation*}
For the last term we again use a paraproduct decomposition, writing
\begin{equation}
\label{eq:ParaproductDecompositionIII}
\begin{split}
\| S_M (\Phi^t Pu ) \|_{L_x^2} &\leq \| (S_M \Phi^t) S_{\ll M} (Pu) \|_{L^2_x} + \| (S_{\ll M} \Phi^t) S_M (Pu ) \|_{L^2_x} \\
&\quad + \| \sum_{K \gtrsim M} S_K \Phi^t S_K (Pu) \|_{L^2_x}
\end{split}
\end{equation}
for $M\ge1$. The first term in \eqref{eq:ParaproductDecompositionIII} is estimated by
\begin{equation*}
\begin{split}
\sum_{M \geq 1} M^{-1} \| S_M \Phi^t S_{\ll M} (Pu) \|_{L^2}^2 &\lesssim \sum_{M \geq 1} M^{-1} \| S_M \Phi^t \|_{L^\infty}^2 \sum_{K \ll M} \| S_K (Pu) \|_{L^2}^2 \\
&\lesssim \sum_{M \geq 1} M^{-1} \| S_M \Phi^t \|^2_{L^\infty}\! \sum_{1 \leq K \ll M} \!K \| |D|^{-\frac{1}{2}} S_K (Pu) \|_{L^2}^2 \\
&\lesssim \sum_{M \geq 1} \|  S_M\Phi^t \|_{L^\infty}^2 \| |D|^{-\frac{1}{2}} Pu \|_{L^2}^2 \\
&\lesssim \| \Phi^t \|_{C^1}^2 \| |D|^{-\frac{1}{2}} Pu \|_{L^2}^2.
\end{split}
\end{equation*}
Similarly, for the second term in \eqref{eq:ParaproductDecompositionIII} we obtain
\begin{equation*}
\sum_{M \geq 1} M^{-1} \| S_{\ll M} \Phi^t S_M (Pu ) \|_{L^2}^2 \lesssim \| \Phi^t \|_{C^1}^2 \| |D|^{-\frac{1}{2}} Pu \|_{L^2}^2.
\end{equation*}
Finally, we estimate the third term in \eqref{eq:ParaproductDecompositionIII} by
\begin{align*}
\sum_{M \geq 1}& M^{-1} \sum_{K \geq M} \| S_K \Phi^t S_K (Pu) \|_{L^2}^2 \\
&\lesssim \sum_{M \geq 1} M^{-1} \sum_{K \gtrsim M} K \| S_K \Phi^t \|_{L^\infty}^2 K^{-1} \| S_K (Pu) \|_{L^2}^2 \\
&\lesssim \sum_{M \geq 1} M^{-1} \| \Phi^t \|_{C^1}^2 \| |D|^{-\frac{1}{2}} P u \|_{L^2}^2
\lesssim \| \Phi^t \|_{C^1}^2 \| |D|^{-\frac{1}{2}} Pu \|_{L^2}^2.
\qedhere
\end{align*}
\end{proof}

\subsection{Proof of Strichartz estimates for the transformed equation}
The proof of Theorem \ref{thm:StrichartzEstimatesFullyAnisotropicOffDiagonal} in the general case revolves around the proof of analogs of \eqref{eq:StrichartzEstimateFullyAnisotropicC1}, which is 
\begin{equation}
\label{eq:TransformedLocalStrichartzEstimatesI}
\| |D|^{-\rho - \frac{1}{4}} \tilde{u} \|_{(L^p L^q)_2} \lesssim (1+\| (\varepsilon^d,\mu^d) \|_{C^1}) \| \tilde{u} \|_{L^2} + \| |D|^{-\frac{1}{2}} \tilde{P} \tilde{u} \|_{L^2} + \| |D|^{-\frac{3}{4}} \tilde{\rho}_{em} \|_{L^2}.
\end{equation}
Using Lemma \ref{lem:OrthogonalityTransformations}, this inequality allows us to complete the proof of Theorem \ref{thm:StrichartzEstimatesFullyAnisotropicOffDiagonal}.
\begin{proposition}
\label{prop:TransformedMaxwellReductionI}
 Let $\varepsilon^d$, $\mu^d \in C^1(\R \times \R^3;\R^{3 \times 3}_{sym})$ be diagonal matrices, which are uniformly elliptic. Let $\tilde{P}$, $\tilde{\rho}_e$, and $\tilde{\rho}_m$ be and in
 \eqref{eq:TransformedCharges} and \eqref{eq:TransformedMaxwellOperator}. 
Then we find \eqref{eq:TransformedLocalStrichartzEstimatesI} to hold for $2 \leq p < q \leq \infty$,
$\rho = 3 \big( \frac{1}{2} - \frac{1}{q} \big) - \frac{1}{p}$, and $\frac{2}{p} + \frac{1}{q} = \frac{1}{2}$.
\end{proposition}
In the following we apply the analysis of Section \ref{section:DiagonalProof} with the role of the partial derivatives $\partial_k$
presently played by the differential operator $\eta_k = \varphi_k^{\leq \lambda} \cdot \nabla_{x'}$. By dyadic frequency localization
and the usual commutator estimates, \eqref{eq:TransformedLocalStrichartzEstimatesI} follows from
\begin{equation*}
\lambda^{-\rho - \frac{1}{4}} \| S_\lambda \tilde u \|_{L^p L^q} \lesssim \| S_\lambda u \|_{L^2} + \lambda^{-\frac{1}{2}} \|\tilde P_\lambda S_\lambda u \|_{L^2} + \lambda^{-\frac{3}{4}} \| S_\lambda \tilde\rho_{em} \|_{L^2}.
\end{equation*}
Above $\tilde{P}_\lambda$ denotes the frequency truncated version at frequencies $< \lambda / 8$ of $\tilde{P}$ given by
\begin{equation*}
\tilde{P}_\lambda = 
\begin{pmatrix}
- \partial_t (\varepsilon^d_\lambda \cdot) & \mathcal{C}(\eta_{\leq \lambda}(D,x)) \\
\mathcal{C}(\eta_\lambda(D,x)) & \partial_t (\mu^d_\lambda \cdot)
\end{pmatrix}
.
\end{equation*}
In $\eta_{\leq \lambda}$ the coefficients of $\varphi$ are also frequency truncated at frequencies $\lambda / 8$.
The frequency truncation is suppressed in the following to lighten the notation.

 We apply the FBI transform and let $v_\lambda = T_\lambda \tilde{u}_\lambda$ and
$T_\lambda(\tilde{\mathcal{J}} /\lambda) = f_\lambda$. For $C^1$-coefficients an application of Theorem \ref{thm:ApproximationTheoremFBITransform} yields
\begin{equation*}
 \tilde{p}(x,\xi) v_\lambda = f_\lambda + g_\lambda \quad
    \text{with \ } \| g_\lambda \|_{L^2_\Phi} \lesssim_{\| (\varepsilon^d,\mu^d) \|_{C^1}} \lambda^{-1/2} \| S_\lambda u \|_{L^2}
\end{equation*}
and
\begin{equation*}
 \tilde{p}(x,\xi) = 
 \begin{pmatrix}
  -i \xi_0 \varepsilon^d & i \mathcal{C}(\eta(x,\xi)) \\
  i \mathcal{C}(\eta(x,\xi)) & i \xi_0 \mu^d
 \end{pmatrix}
\in \mathbb{C}^{6 \times 6}.
\end{equation*}
The estimate \eqref{eq:TransformedLocalStrichartzEstimatesI} becomes
\begin{equation*}
 \lambda^{-\rho-\frac14} \| T_\lambda^* \tilde{v}_\lambda \|_{L^p L^q} \lesssim (1+ \| (\varepsilon^d, \mu^d) \|_{C^1}) \| \tilde{v}_\lambda \|_{L^2_\Phi} + \lambda^{\frac12} \| \tilde{f}_\lambda \|_{L^2_\Phi} + \lambda^{-\frac34} \| S_\lambda \tilde{\rho}_{em} \|_{L^2}.
\end{equation*}
We recast this as
\begin{equation}
\label{eq:ApplicationFBITransformGeneralCoefficientsII}
 \| T_\lambda^* \tilde{v}_\lambda \|_{L^p L^q} \lesssim \lambda^{\rho} ( \lambda^{\frac14} \| \tilde{v}_\lambda \|_{L^2_\Phi} + \lambda^{\frac34} \| \tilde{p}(x,\xi) \tilde{v}_\lambda \|_{L^2_\Phi} + \lambda^{-\frac12} \| S_\lambda \tilde{\rho}_{em} \|_{L^2} ).
\end{equation}

Next, we reduce \eqref{eq:ApplicationFBITransformGeneralCoefficientsII} to a scalar estimate. For this purpose, corresponding to \eqref{eq:Characteristic}  we let
\begin{equation*}
 \tilde{q}(x,\xi) = - \xi_0^2 (\xi_0^4 - \xi_0^2 q_0(x,\eta'(\xi)) + q_1(x,\eta'(\xi))),
\end{equation*}
with $q_0$ and $q_1$ defined like in Section \ref{section:DiagonalProof} and define the symmetrizer
\begin{equation*}
 \tilde{\sigma}(x,\xi)= 
 \begin{pmatrix}
   -i \xi_0 \varepsilon_d^{-1} & i \varepsilon_d^{-1} \mathcal{C}(\eta(x,\xi)) \mu_d^{-1} \\
   i \mu_d^{-1} \mathcal{C}(\eta(x,\xi)) \varepsilon_d^{-1} & i \xi_0 \mu_d^{-1}
 \end{pmatrix}
\in \mathbb{C}^{6 \times 6},
\end{equation*}
obtaining
\begin{equation*}
 \tilde{\sigma}(x,\xi) \tilde{p}(x,\xi) = 
 \begin{pmatrix}
  M_E(x,\eta) - \xi_0^2 & 0 \\
  0 & M_H(x,\eta) - \xi_0^2
 \end{pmatrix}
 ,
\end{equation*}
compare \eqref{eq:ME-MH}. We state the announced reduction.
\begin{proposition}
 For the proof of \eqref{eq:ApplicationFBITransformGeneralCoefficientsII} under the assumptions of Proposition \ref{prop:TransformedMaxwellReductionI}, it suffices to show the estimate
 \begin{equation}
 \label{eq:ScalarEstimate}
  \| T_\lambda^* w_\lambda \|_{L^p L^q} \lesssim \lambda^{\rho} (\lambda^{1/4} \| w_\lambda \|_{L^2_\Phi} + \lambda^{3/4} \| \tilde{q}(x,\xi) w_\lambda \|_{L^2_\Phi} ).
 \end{equation}
\end{proposition}
\begin{proof}
 Clearly,  \eqref{eq:ApplicationFBITransformGeneralCoefficientsII} follows from
 \begin{equation*}
  \| T_\lambda^* \tilde{v}_\lambda \|_{L^p L^q} \lesssim \lambda^{\rho} ( \lambda^{1/4} \| \tilde{v}_\lambda \|_{L^2_\Phi} + \lambda^{3/4} \| \tilde{\sigma}(x,\xi) \tilde{p}(x,\xi) \tilde{v}_\lambda \|_{L^2_\Phi} + \lambda^{-1/2} \| S_\lambda \tilde{\rho}_{em} \|_{L^2} ).
 \end{equation*}
For this purpose, we show that
 \begin{equation}
  \label{eq:ReductionGeneralCoefficientsI}
  \| T^*_\lambda \tilde{v}_{\lambda,1} \|_{L^p L^q} \lesssim \lambda^{\rho} ( \lambda^{1/4} \| \tilde{v}_{\lambda,1} \|_{L^2_\Phi} + \lambda^{3/4} \| (M_E(x,\eta) - \xi_0^2) \tilde{v}_{\lambda,1} \|_{L^2_\Phi} + \lambda^{-1/2} \| S_\lambda \tilde{\rho}_e \|_{L^2}),
  \end{equation}
  \begin{equation}
  \label{eq:ReductionGeneralCoefficientsII}
  \| T_\lambda^* \tilde{v}_{\lambda,2} \|_{L^p L^q} \lesssim \lambda^{\rho} ( \lambda^{1/4} \| \tilde{v}_{\lambda,2} \|_{L^2_\Phi} + \lambda^{3/4} \| (M_H(x,\eta) - \Xi_0^2) \tilde{v}_{\lambda,2} \|_{L^2_\Phi} + \lambda^{-1/2} \| S_\lambda \tilde{\rho}_m \|_{L^2}).
 \end{equation}
In the following it becomes relevant that $|\eta'| \sim |\xi'|$. We prove the  estimates in the  regions
\begin{itemize}
 \item[(1)] $\{ |\xi_0| \gg |\eta'| \}$, 
 \item[(2)] $\{ |\xi_0| \ll |\eta'| \}$, for which we take the generalized charges into account,
 \item[(3)] $\{ |\xi_0| \sim |\eta'| \}$.
\end{itemize}
The first region is handled like in the proof of Proposition \ref{prop:ScalarReduction} using $|\eta'| \sim |\xi'|$.
For the estimate in the second region, we focus on \eqref{eq:ReductionGeneralCoefficientsI} because \eqref{eq:ReductionGeneralCoefficientsII} can be proved in a similar way. We decompose 
$ \tilde{v}_{\lambda,1} = \tilde{v}^s_{\lambda,1} + \tilde{v}_{\lambda,1}^p $
with
\begin{equation*}
 \tilde{v}^p_{\lambda,1} = \frac{ (\eta'.\tilde{\varepsilon}^\lambda \tilde{v}_{\lambda,1}) \eta'}{|\eta'|^2_{\varepsilon^d}}.
\end{equation*}
The contribution of $\tilde{v}^p_{\lambda,1}$ is treated by Sobolev embedding
\begin{equation*}
 \| T_\lambda^* \tilde{v}^p_{\lambda,1} \|_{L^p L^q} \lesssim \lambda^{\rho + \frac{1}{2}} \| \tilde{v}^p_{\lambda,1} \|_{L^2_\Phi} \lesssim \lambda^{\rho + \frac{1}{2}} \| (\eta'.\varepsilon^d_\lambda \tilde{v}_{\lambda,1}) \|_{L^2_\Phi}.
\end{equation*}
By Theorem \ref{thm:ApproximationTheoremFBITransform} and a commutator estimate, we find
\begin{equation*}
 \| \eta'.(\varepsilon^d_\lambda v_{\lambda,1}) \|_{L^2_\Phi} \lesssim \lambda^{-\frac{1}{2}} \| v_{\lambda,1} \|_{L^2} + \| S_\lambda \frac{\eta'(x,D)}{\lambda} (\varepsilon^d \tilde{\mathcal{E}}) \|_{L^2}.
\end{equation*}
Secondly, it follows like in the proof of Proposition \ref{prop:ScalarReduction} by the same algebraic relations, replacing $\xi'$ with $\eta'$, that $|(M_E(x,\eta') - \xi_0^2) v_1^\lambda | \gtrsim |v_1^\lambda|$ for $\eta'.(\varepsilon^d_\lambda \tilde{v}^\lambda_1) = 0$
and $\{ |\xi_0| \ll |\xi'| \}$. The details are omitted.

Similarly, by replacing $\eta'$ with $\xi'$, we argue that the estimate
\begin{equation*}
 \| T_\lambda^* v_{\lambda,1} \|_{L^p L^q} \lesssim \lambda^\rho ( \lambda^{1/4} \| v_{\lambda,1} \|_{L^2_\Phi} + \lambda^{3/4} \| (M_E - \xi_0^2) v_{\lambda,1} \|_{L^2_\Phi})
\end{equation*}
holds true provided that
\begin{equation*}
 \| T_\lambda^* v_{\lambda,1,k} \|_{L^p L^q} \lesssim \lambda^{\rho} ( \lambda^{1/4} \| v_{\lambda,1,k} \|_{L^2_\Phi} + \lambda^{3/4} \| (Z_{\varepsilon,\mu} \frac{\varepsilon^d}{\varepsilon^d_1 \varepsilon^d_2 \varepsilon^d_3} (M_E - \xi_0^2) v_{\lambda,1})_k \|_{L^2_\Phi} ),
\end{equation*}
but this is \eqref{eq:ScalarEstimate}.
\end{proof}

To prove \eqref{eq:ScalarEstimate}, we modify the arguments of Section \ref{section:DiagonalProof}. We consider the operator
\begin{equation*}
 \tilde{W}_\lambda = T_\lambda^* \frac{a(x,\xi)}{\lambda^{\frac{1}{4}} + \lambda^{\frac{3}{4}} |\tilde{q}(x,\xi)|},
\end{equation*}
for which we shall prove
\begin{equation*}
 \| \tilde{W}_\lambda \|_{L^2 \to L^p L^q} \lesssim \lambda^{\rho + \frac{1}{2}}.
\end{equation*}
By the $TT^*$-argument, we can likewise prove the estimate
\begin{equation*}
 \tilde{V}_\lambda = T_\lambda^* \frac{a^2(x,\eta) \Phi(\xi)}{(\lambda^{-\frac{1}{4}} + \lambda^{\frac{1}{4}} |\tilde{q}(x,\xi)|)^2} T_\lambda.
\end{equation*}
Again we have to understand the curvature properties of $\{ \xi \in \R^4: \, \tilde{q}(x,\xi) = 0 \}$.
Note that $\tilde{q}(x,\xi)$ is $4$-homogeneous in $\xi$, i.e.,
\begin{equation*}
 \tilde{q}(x,\xi_0,\xi') = \xi_0^4 \tilde{q}(x,1,\xi'/\xi_0).
\end{equation*}
This reduces again to the analysis of the surface $\tilde{S} = \{ \xi' \in \R^3: \tilde{q}(x,1,\xi') = 0\}$. By the change of variables $\eta_k' = \varphi_k^{\leq \lambda}(x) \cdot \xi'$ we can 
reduce to the Fresnel surface $S = \{ \xi' \in \R^3: q(x,1,\xi') = 0 \}$. Like in Section \ref{section:DiagonalProof} we split
replace $a^2(x,\xi)$ by $a(x,\xi)$ and split it into $a=a_1 + a_2 + a_3$ according to the
regions of the Fresnel surface $S$. Correspondingly, we decompose $\tilde{V}_\lambda = \tilde{V}_{\lambda,1} + \tilde{V}_{\lambda,2} + \tilde{V}_{\lambda,3}$. We arrive at the following:
\begin{proposition}
 Let $\rho = 3 \big( \frac{1}{2} - \frac{1}{q} \big) - \frac{1}{p}$ and $2 < p \leq q \leq \infty$. The estimate
 \begin{equation*}
  \| \tilde{V}_{\lambda,i} \|_{L^{p'} L^{q'} \to L^p L^q} \lesssim \lambda^{1 + 2 \rho}
 \end{equation*}
holds true, if
\begin{itemize}
 \item $i=1$ and $\frac{1}{p} + \frac{1}{q} = \frac{1}{2}$,
 \item $i \in \{2,3\}$ and $\frac{2}{p} + \frac{1}{q} = \frac{1}{2}$.
\end{itemize}
\end{proposition}
\begin{proof}
 For fixed $x \in [-1,1]^4$ we observe that the change of variables
 \begin{equation}
  \label{eq:ChangeVariables}
  \eta_k' = \varphi_k^{\leq \lambda}(x) \cdot \xi'
 \end{equation}
is non-degenerate for $\lambda \gg 1$. Indeed, $|J \eta'| = |\det (\varphi_1^{\leq \lambda}, \varphi_2^{\leq \lambda}, \varphi_3^{\leq \lambda} )| \sim 1$.

Hence, for $i=1,2$, a uniform bound for the family of operators
\begin{equation*}
 \lambda^{-1-2\rho} T_\lambda^* \frac{a_i(x,\eta) \Phi(\xi)}{(\lambda^{-\frac{1}{4}} + \lambda^{\frac{1}{4}} |\tilde{q}(x,\xi)|)^2} T_\lambda : L^{p'} L^{q'} \to L^p L^q
\end{equation*}
follows from the integrability of the weight
\begin{equation*}
 \int \frac{1}{(\lambda^{-\frac{1}{4}} + \lambda^{\frac{1}{4}} |\tilde{q}|)^2} d\tilde{q} \lesssim 1,
\end{equation*}
by which we can reduce to level sets $\delta_{\tilde{q}(x,\xi) = c}$ by foliation. Note that within
$\text{supp}(a_1) \cup \text{supp}(a_2)$ the surface $\{ \xi \in \R^4: \tilde{q}(x,\xi) = 0 \}$ is a regular surface.
By the regular change of variables \eqref{eq:ChangeVariables} the curvature properties are inherited.

For the more involved case of neighbourhoods of the conical singularities, the additional dyadic decomposition is carried out
in $\eta$. After rescaling to unit distance of the singularity, one can then proceed as in Proposition \ref{prop:OperatorEstimatesTT^*}.
\end{proof}

\subsection{Conclusion of the proof of Theorem \ref{thm:StrichartzEstimatesFullyAnisotropicOffDiagonal}}

With the estimates for $\tilde{V}_{\lambda,i}$ at disposal, the propositions of the previuous subsection
yield the Strichartz estimates
\begin{equation*}
\| |D|^{-\rho - \frac{1}{4}} \tilde{u} \|_{(L^4 L^\infty)_2} \lesssim (1+ \| \partial(\varepsilon^d,\mu^d) \|_{L^\infty}) \| \tilde{u} \|_{L^2} + \| |D|^{-\frac{1}{2}} \tilde{P} \tilde{u} \|_{L^2} + \| |D|^{-\frac{3}{4}} \tilde{\rho}_{em} \|_{L^2}.
\end{equation*}
These estimates transpire to the original quantities by the considerations in Section \ref{subsection:OrthogonalityTransformations},
leading to
\begin{equation*}
\| |D|^{-\rho-\frac{1}{4}} u \|_{L_t^4 L^\infty_{x'}} \lesssim (1+ \| \partial (\varepsilon,\mu) \|_{L^\infty}) \| u \|_{L^2} +\| |D|^{-\frac{1}{2}} P u \|_{L^2} + \| |D|^{-\frac{3}{4}} \rho_{em} \|_{L^2}.
\end{equation*}

We can now conclude like in Section \ref{subsection:ShorttimeStrichartz}. Proposition \ref{prop:ShorttimeStrichartz} implies
\begin{equation*}
\begin{split}
\| \langle D' \rangle^{-\rho -\frac{1}{8} - \delta} u \|_{L_t^4(0,T;L^\infty_{x'})} &\leq C ( \| u_0 \|_{L^2} + \| P u \|_{L_T^1 L^2_{x'}}\\
&\qquad + \| \langle D' \rangle^{-\frac{5}{8}} \rho_{em}(0) \|_{L^2_{x'}} + \| \langle D' \rangle^{-\frac{5}{8}} \partial_t \rho_{em} \|_{L^1_t L^2_{x'}} )
\end{split}
\end{equation*}
with $C = C(X,T,\delta)$, $X = \| \partial(\varepsilon, \mu) \|_{L^\infty}$.
Interpolation with  Proposition \ref{prop:EnergyEstimate}
yields Theorem \ref{thm:StrichartzEstimatesFullyAnisotropic} for $\frac{2}{p} + \frac{1}{q} = \frac{1}{2}$. The general case follows from Sobolev embedding. $\hfill \Box$

\section{Application to quasilinear Maxwell equations}
\label{section:QuasilinearEquations}
With the Strichartz estimates at disposal, we can improve local well-posedness results for quasilinear Maxwell equations in the fully anisotropic case. This section is devoted to the proof of Theorem \ref{thm:QuasilinearAnisotropicMaxwellEquation}. By local well-posedness we mean existence, uniqueness, and continuous dependence of the data-to-solution mapping. We focus on proving estimates in rough norms for smooth solutions, whose existence we take for granted. The data-to-solution mapping can then be extended by continuity to rough initial data via standard arguments. We also refer to our previous work \cite{SchippaSchnaubelt2022}.

We consider
\begin{equation}
 \label{eq:QuasilinearMaxwell}
 \left\{ \begin{array}{rlrlrl}
  \partial_t \mathcal{D}\!\!\!\! &= \nabla \times \mathcal{H},& \quad \nabla \cdot \mathcal{D} \!\!\!\! &= 0,&
    \quad \mathcal{D}(0) \!\!\!\! &= \mathcal{D}_0 \in H^s(\R^3;\R^3), \\
  \partial_t \mathcal{B}\!\!\!\! &= - \nabla \times \mathcal{E},& \quad \nabla \cdot \mathcal{B} \!\!\!\! &= 0, &
     \quad \mathcal{B}(0)\!\!\!\! & = \mathcal{B}_0 \in H^s(\R^3;\R^3),
 \end{array} \right.
\end{equation}
with fields 
\begin{equation}
\label{eq:SmallFields}
\| \mathcal{E} \|_{L^\infty} + \| \mathcal{B} \|_{L^\infty} \leq \delta,
\end{equation}
$\delta$ to be specified later, and constitutive relations 
\begin{equation*}
 \varepsilon = \varepsilon(\mathcal{E}) \in \R^{3 \times 3}_{\text{sym}} , \quad \mu = 1_{3 \times 3}
\end{equation*}
such that $\varepsilon(\mathcal{E})$ is uniformly elliptic and has uniformly separated eigenvalues for $|\mathcal{E}| \leq \delta$.
We use this smallness condition to guarantee these two crucial properties of the permittivities, it is not used in the wellposedness
analysis given below.

We let $u = (\mathcal{D},\mathcal{B})$ in the following because these are the variables of which the
time-derivative is given. The proof of the theorem follows along the general scheme of \cite{Schippa2021Maxwell3d,SchippaSchnaubelt2022}:
Set $A = \sup_{0 \leq t' \leq t} \| u(t') \|_{L^\infty_{x'}}$ and $B(t) = \| \nabla_{x'} u(t) \|_{L^\infty_{x'}}$.

\begin{itemize}
 \item We show energy estimates to hold for $s \geq 0$:
 \begin{equation}
 \label{eq:AprioriEstimate}
  E^s(u(t)) \lesssim E^s(u(0)) e^{C \int_0^t B(s) ds}
 \end{equation}
with $E^s(u) \approx_{A} \| u \|_{H^s}$.
\item For differences of solutions we prove $L^2$-bounds depending on the $H^s$-norm $v= u^1 - u^2$ of the initial data:
\begin{equation*}
  \| v(t) \|_{L^2} \lesssim \| v_0 \|_{L^2} \text{ for } 0 \leq t \leq T = T(\| u^i(0) \|_{H^s}).
\end{equation*}
\item We use the frequency envelope approach due to Tao \cite{Tao2001} (see also \cite{IfrimTataru2020}) to show continuous dependence (but no locally uniform continuous dependence).
\end{itemize}

To improve on the threshold $s>\frac52$ dictated by the energy method, we want to estimate $\| \partial_x u \|_{L_T^1 L_{x'}^\infty}$ by Strichartz estimates. This is carried out via a bootstrap argument. For $u$ a solution in $H_{x'}^{\frac{3}{2}+\alpha}$, the coefficients of the quasilinear Maxwell equation are $\alpha$ H\"older-continuous and in the scope of Theorem \ref{thm:StrichartzEstimatesFullyAnisotropic}. If the provided Strichartz estimates  can control $\| \partial_x u \|_{L_T^1 L_{x'}^\infty}$ in terms of the $H_{x'}^{\frac{3}{2}+\alpha}$-norm, we can close the iteration. To find $\alpha$, we equate the derivative loss $\rho + 1$ of the Strichartz estimates with $\frac{3}{2}+\alpha$, obtaining
\begin{equation*}
3 \big( \frac{1}{2} - \frac{1}{q} \big) - \frac{1}{p} + \frac{2-\alpha}{2p} + 1 = \frac{3}{2} + \alpha, \quad q = \infty, \; p =4.
\end{equation*}
This gives $\alpha = \frac{8}{9}$ and suggests that the argument closes for initial data with regularity $s > \frac{5}{2} - \frac{1}{9}$.

\smallskip

One new ingredient compared to previous works is to find the energy norm for \eqref{eq:QuasilinearMaxwell}, which is carried out in detail below.
It turns out that in the fully anisotropic case introducing a suitable energy norm requires more care than in the isotropic case
considered in \cite{SchippaSchnaubelt2022,Schippa2021Maxwell3d}. We recall the following fact from \cite{Schippa2021Maxwell3d}.

\begin{proposition}[Existence~of~energy~estimates]
Suppose that there is symmetric $C(u)$ such that
\begin{equation}
 \label{eq:SymmetrizingCondition}
 \mathcal{A}^j(u)^t C(u) = C(u) \mathcal{A}^j(u),
\end{equation}
and $C(u)$ is uniformly elliptic. Then, we have
\begin{equation*}
 E^s(u) := \langle \langle D' \rangle^s u, C(u) \langle D' \rangle^s u \rangle \approx_A \| u \|^2_{H^s}
\end{equation*}
and  \eqref{eq:AprioriEstimate} holds, for $s \geq 0$.
\end{proposition}

To this end, we rewrite Maxwell equations in conservative form $\partial_t u = \mathcal{A}^j(u) \partial_j u$.
However, to state \eqref{eq:QuasilinearMaxwell} in conservative form, we have to recast $\mathcal{E} = \psi(\mathcal{D}) \mathcal{D}$, for which we can formulate a necessary and sufficient condition
on the existence of symmetrizers. Note that $\psi(\mathcal{D})$ is symmetric for symmetric $\varepsilon(\mathcal{E})$.
\begin{proposition}
 \label{prop:ExistenceSymmetrizer}
 There exists
 \begin{equation}
  \label{eq:AnsatzSymmetrizer}
  C(u) = 
  \begin{pmatrix}
   C^1(u) & 0 \\
   0 & 1_{3 \times 3}
  \end{pmatrix}
\in \R^{6 \times 6}_{sym}
 \end{equation}
 with $C^1(u) \in \R^{3 \times 3}$ that satisfies \eqref{eq:SymmetrizingCondition} if and only if
 \begin{equation}
  \label{eq:ConditionsCoefficientsForSymmetrizer}
  \varepsilon^{ijk} \partial_j \psi(\mathcal{D})_{k \ell} \mathcal{D}_{\ell} = 0
 \end{equation}
for $i \in \{1,2,3\}$ with summation over $k,j \in \{1,2,3\}$.
\end{proposition}
\begin{proof}
We write for $j=1,2,3$
\begin{equation*}
\mathcal{A}^j(u) = 
\begin{pmatrix}
0 & A_1^j(u) \\
A_2^j(u) & 0
\end{pmatrix};
\quad A_1^j(x), A_2^j(x) \in \R^{3 \times 3}
\end{equation*}
with $(A_1^j)_{mn} = - \varepsilon_{jmn}$. For the special form \eqref{eq:AnsatzSymmetrizer} the equation \eqref{eq:SymmetrizingCondition} translates to 
 \begin{equation}
 \label{eq:ReducedSymmetrizingCondition}
  (A_1^j)^t C^1 = A^j_2.
 \end{equation}
We compute $A^j_2$ from
\begin{equation*}
 - (\nabla \times (\psi(\mathcal{D}) \mathcal{D}))_i= (A_2^j(\mathcal{D}) \partial_j \mathcal{D})_i,
\end{equation*}
where we have
\begin{equation*}
\begin{split}
 (\nabla \times (\psi(\mathcal{D}) \mathcal{D}))_i &= \varepsilon^{ijk} \partial_j(\psi(\mathcal{D}) \mathcal{D})_k \\
 &= \varepsilon^{ijk} \partial_j( \psi(\mathcal{D})_{km} \mathcal{D}_m) \\
 &= \varepsilon^{ijk} [\psi(\mathcal{D})_{km} \partial_j \mathcal{D}_m] + \varepsilon^{ijk} \partial_{\ell} \psi(\mathcal{D})_{km} (\partial_j \mathcal{D}_\ell) \mathcal{D}_m.
\end{split}
\end{equation*}
We define the matrices
\begin{equation*}
 \mathcal{B}^{j}_{im} = - \varepsilon^{ijk} \psi(\mathcal{D})_{km}, \quad \mathcal{C}^j_{im} = - \varepsilon^{ijk} \partial_m \psi(\mathcal{D})_{kl} \mathcal{D}_l
\end{equation*}
and let $A_2^j= \mathcal{B}^j + \mathcal{C}^j$. We obtain
\begin{align*}
  \mathcal{B}^1 &=
 \begin{pmatrix}
  0 & 0 & 0 \\
 \psi(\mathcal{D})_{31} & \psi(\mathcal{D})_{32} & \psi(\mathcal{D})_{33} \\
 - \psi(\mathcal{D})_{21} & - \psi(\mathcal{D})_{22} & - \psi(\mathcal{D})_{23}
 \end{pmatrix}
 ,\\
 \mathcal{B}^2 &=
 \begin{pmatrix}
  - \psi(\mathcal{D})_{31} & - \psi(\mathcal{D})_{32} & - \psi(\mathcal{D})_{33} \\
  0 & 0 & 0 \\
  \psi(\mathcal{D})_{11} & \psi(\mathcal{D})_{12} & \psi(\mathcal{D})_{13}
 \end{pmatrix}
, \\
\mathcal{B}^3 &= 
\begin{pmatrix}
\psi(\mathcal{D})_{21} & \psi(\mathcal{D})_{22} & \psi(\mathcal{D})_{23} \\
- \psi(\mathcal{D})_{11} & - \psi(\mathcal{D})_{12} & - \psi(\mathcal{D})_{13} \\
0 & 0 & 0
\end{pmatrix}
.
\end{align*}
Moreover,
\begin{equation*}
 \begin{split}
  \mathcal{C}^1 &= 
  \begin{pmatrix}
   0 & 0 & 0 \\
   \partial_1 \psi(\mathcal{D})_{3 \ell} \mathcal{D}_\ell & \partial_2 \psi(\mathcal{D})_{3 \ell} \mathcal{D}_\ell & \partial_3 \psi(\mathcal{D})_{3 \ell} \mathcal{D}_\ell \\
   - \partial_1 \psi(\mathcal{D})_{2 \ell} \mathcal{D}_\ell & - \partial_2 \psi(\mathcal{D})_{2 \ell} \mathcal{D}_\ell & - \partial_3 \psi(\mathcal{D})_{2 \ell} \mathcal{D}_\ell
  \end{pmatrix}
, \\ \mathcal{C}^2 &= 
\begin{pmatrix}
 - \partial_1 \psi(\mathcal{D})_{3 \ell} \mathcal{D}_\ell & - \partial_2 \psi(\mathcal{D})_{3 \ell} \mathcal{D}_\ell & - \partial_3 \psi(\mathcal{D})_{3 \ell} \mathcal{D}_\ell \\
 0 & 0 & 0 \\
 \partial_1 \psi(\mathcal{D})_{1 \ell} \mathcal{D}_\ell & \partial_2 \psi(\mathcal{D})_{1 \ell} \mathcal{D}_{\ell} & \partial_3 \psi(\mathcal{D})_{1 \ell} \mathcal{D}_\ell
\end{pmatrix}
, \\
\mathcal{C}^3 &= 
\begin{pmatrix}
 \partial_1 \psi(\mathcal{D})_{2 \ell} \mathcal{D}_\ell & \partial_2 \psi(\mathcal{D})_{2 \ell} \mathcal{D}_\ell & \partial_3 \psi(\mathcal{D})_{2 \ell} \mathcal{D}_\ell \\
 - \partial_1 \psi(\mathcal{D}_{1 \ell} \mathcal{D}_\ell & -\partial_2 \psi(\mathcal{D})_{1 \ell} \mathcal{D}_\ell & - \partial_3 \psi(\mathcal{D})_{1 \ell} \mathcal{D}_\ell \\
 0 & 0 & 0
\end{pmatrix}
.
 \end{split}
\end{equation*}
Note that
\begin{equation*}
A_1^1 =
\begin{pmatrix}
0 & 0 & 0 \\
0 & 0 & -1 \\
0 & 1 & 0
\end{pmatrix}, \;
A_1^2 =
\begin{pmatrix}
0 & 0 & 1 \\
0 & 0 & 0 \\
-1 & 0 & 0
\end{pmatrix}, 
\;
A_1^3 =
\begin{pmatrix}
0 & - 1 & 0 \\
1 & 0 & 0 \\
0 & 0 & 0
\end{pmatrix}
.
\end{equation*}
This yields
\begin{equation*}
\begin{split}
(A_1^1)^t C^1 &= 
\begin{pmatrix}
0 & 0& 0 \\
C^1_{31} & C^1_{32} & C^1_{33} \\
-C^1_{21} & - C^1_{22} & - C^1_{23}
\end{pmatrix}, \;
(A_1^2)^t C^1 =
\begin{pmatrix}
-C^1_{31} & - C^1_{32} & -C^1_{33} \\
0 & 0 & 0 \\
C^1_{11} & C^1_{12} & C^1_{13}
\end{pmatrix}, \\
(A_1^3)^t C^1 &=
\begin{pmatrix}
C^1_{21} & C^1_{22} & C^1_{23} \\
-C^1_{11}  & - C^1_{12} & - C^1_{13} \\
0 & 0 & 0
\end{pmatrix}
.
\end{split}
\end{equation*}

 A comparison of the coefficients in \eqref{eq:ReducedSymmetrizingCondition} yields \eqref{eq:ConditionsCoefficientsForSymmetrizer} by symmetry of $\psi(\mathcal{D})$. For instance, we find
\begin{equation*}
\begin{split}
((A^1_1)^t C^1)_{21} &= C^1_{31} = \psi(\mathcal{D})_{31} + \partial_1 \psi(\mathcal{D})_{3 \ell} \mathcal{D}_{\ell}, \\ ((A^2_1)^t C^1)_{31} &= C^1_{13} = \psi(\mathcal{D})_{13} + \partial_3 \psi(\mathcal{D})_{1 \ell} \mathcal{D}_{\ell}.
\end{split}
\end{equation*} 
This gives $\varepsilon^{1 jk } \partial_j \psi(\mathcal{D})_{k \ell} \mathcal{D}_{\ell} = 0$. The other conditions are found likewise. 
\end{proof}

We remark that this assumption can be difficult to verify starting with $\varepsilon = \varepsilon(\mathcal{E})$. It turns out that the seemingly natural ansatz 
\begin{equation*}
 \varepsilon(\mathcal{E}) = \text{diag}(\varepsilon_0^1,\varepsilon_0^2,\varepsilon_0^3) + \text{diag}(\alpha_1,\alpha_2,\alpha_3) |\mathcal{E}|^2
\end{equation*}
for $\varepsilon_0^1 \neq \varepsilon_0^2 \neq \varepsilon_0^3 \neq \varepsilon_0^1$ fails if $\alpha_i \neq \alpha_j$ for some $i \neq j$.
An admissible choice of $\varepsilon$ is
\begin{equation}
\label{eq:ExampleQuasilinearMaterialLaw}
 \varepsilon (\mathcal{E})= \text{diag}(\varepsilon_0^1,\varepsilon_0^2,\varepsilon_0^3) + \text{diag}(\alpha_1 |\mathcal{E}_1|^2, \alpha_2 |\mathcal{E}_2|^2, \alpha_3 |\mathcal{E}_3|^2)
\end{equation}
with $\varepsilon_0^i>0$.
In this case, we have $\mathcal{D}_i = (\varepsilon_0^i + \alpha_i |\mathcal{E}_i|^2) \mathcal{E}_i$, which admits inversion as
\begin{equation*}
 \mathcal{E}_i = \psi_i(\mathcal{D}_i) \mathcal{D}_i \text{ and } \psi_i(0) = (\varepsilon_0^i)^{-1}.
\end{equation*}
This allows for verification of \eqref{eq:ConditionsCoefficientsForSymmetrizer}. Instead of starting with $\varepsilon = \varepsilon(\mathcal{E})$, one can make the ansatz
\begin{equation}
 \label{eq:AnsatzPsi}
 \psi(\mathcal{D})_{ij} = \varepsilon^0_{ij} + \alpha_{ij} \mathcal{D}_i \mathcal{D}_j,
\end{equation}
for which it is easy to see that \eqref{eq:ConditionsCoefficientsForSymmetrizer} requires symmetry of $(\varepsilon^0_{ij})_{ij}$ and $(\alpha_{ij})_{ij}$. Then, by the implicit function theorem,
for small fields 
\begin{equation}
\label{eq:SmallFieldsDelta}
\| \mathcal{E} \|_{L^\infty} + \| \mathcal{D} \|_{L^\infty} \leq \delta
\end{equation}
 we can rewrite $\mathcal{D} = \varepsilon(\mathcal{E}) \mathcal{E}$ with $\varepsilon(\mathcal{E})=\psi(\mathcal{D}(\mathcal{E}))^{-1}$.
Clearly, $\varepsilon(\mathcal{E})$ is symmetric as an inverse of a symmetric matrix, and the uniform anisotropy
and ellipticity are true for small $\| \mathcal{E} \|_{L^\infty}$. Taking $\delta$ smaller than in \eqref{eq:SmallFieldsDelta}, if necessary, specifies $\delta$ in \eqref{eq:SmallFields}.

By the assumptions of Theorem \ref{thm:QuasilinearAnisotropicMaxwellEquation}, the energy estimate
\begin{equation}
\label{eq:EnergyEstimateAux}
\| u(t) \|_{H^s} \lesssim_\delta \| u(0) \|_{H^s} e^{C \int_0^t B(s) ds}
\end{equation}
is true for $\| u \|_{L_{t,x'}^\infty} \leq \delta$. In the next step we use Strichartz estimates to show a priori estimates for
$s > 2 + \frac{7}{18}$. Note that by Sobolev embedding \eqref{eq:EnergyEstimateAux} yields a priori estimates for $s > \frac{5}{2}$.

\begin{proposition}
\label{prop:APrioriQuasilinear}
Under the assumptions of Theorem \ref{thm:QuasilinearAnisotropicMaxwellEquation}, the a priori estimate
\begin{equation*}
\sup_{t \in [0,T]} \| u(t) \|_{H^s} \lesssim \| u(0) \|_{H^s}
\end{equation*}
holds for $s > 2 + \frac{7}{18}$,  $T = T(\| u_0 \|_{H^s})$ and $\|u_0 \|_{H^s} \leq \delta$ for some $\delta >0 $.
\end{proposition}

\begin{proof}
Suppose that the smooth solution $u= (\mathcal{D},\mathcal{H})$ exists for $[0,T_0]$. Let $0<T\le T_0$.
The proof follows from bootstrapping the energy estimate \eqref{eq:EnergyEstimateAux}
\begin{equation*}
E^s(u(t)) \lesssim_A e^{C \int_0^T B(s) ds} E^s(u(0))
\end{equation*}
and the Strichartz estimate
\begin{equation}
\label{eq:StrichartzQuasilinearAux}
\| \partial_{x'} u \|_{L_T^4 L_{x'}^\infty} \lesssim \| \langle D' \rangle^s u \|_{L_T^\infty L_{x}^2}
\end{equation}
for $s\ge s_0 > 2 + \frac{7}{18}$. The constant in the Strichartz estimate has to be uniform provided that $T$, $\| \partial_x u \|_{L_T^4 L_{x'}^\infty}$, $\| u \|_{C^{\frac{8}{9}+\sigma}_x}$ are bounded and $s \ge s_0$.

 We have to establish \eqref{eq:StrichartzQuasilinearAux}, for which we make use of commutator arguments. These only apply
 after changing Maxwell equations to non-divergence form.  Let
\begin{equation*}
\varepsilon(\mathcal{E}) = \text{diag}(\varphi_1(\mathcal{E}_1),\varphi_2(\mathcal{E}_2),\varphi_3(\mathcal{E}_3))
\end{equation*}
 and denote 
\begin{equation*}
\tilde{\varepsilon}(\mathcal{E}) = \text{diag}(\tilde{\varepsilon}_1(\mathcal{E}_1),\tilde{\varepsilon}_2(\mathcal{E}_2),\tilde{\varepsilon}_3(\mathcal{E}_3)) \quad \text{with \ } \tilde{\varepsilon}_i(\mathcal{E}) = \varphi_i(\mathcal{E}_i) + \varphi_i'(\mathcal{E}_i) \mathcal{E}_i.
\end{equation*} 
Note that $\partial_x (\varepsilon(\mathcal{E}) \mathcal{E}) = \tilde{\varepsilon}(\mathcal{E}) \partial_x \mathcal{E}$.


For the proof of \eqref{eq:StrichartzQuasilinearAux} we have to show
\begin{equation*}
\| \partial_{x'} (\mathcal{E},\mathcal{H}) \|_{L_T^4 L_{x'}^\infty} \lesssim \| (\mathcal{E},\mathcal{H}) \|_{L_T^\infty H^s_{x'}}
\quad \text{for \ } s \ge s_0> \frac{3}{2} + \frac{8}{9}.
\end{equation*}
Via a continuity argument, we shall control
\begin{equation*}
\| \langle D' \rangle^{1+\kappa} (\mathcal{E},\mathcal{H}) \|_{L_T^{4+ \beta} L_{x'}^q} \lesssim \| (\mathcal{E},\mathcal{H}) \|_{L_T^\infty H^s_{x'}} \quad  \text{for \ } s\ge s_0 > \frac{3}{2} + \frac{8}{9},
\end{equation*}
$q$ large enough, and $(4+\beta,q)$ being a sharp Strichartz pair.  The slightly larger exponent in $L_t^p$ allows us to bring powers
of $T$ into play. The positive number $\kappa$ is chosen such that $\| f \|_{L^\infty(\R^3)} \lesssim \| \langle D' \rangle^{\kappa} f \|_{L^q(\R^3)}$ by Sobolev embedding.

We consider the Maxwell system
\begin{equation*}
\left\{ \begin{array}{rlrl}
\partial_t (\tilde{\varepsilon} \mathcal{E}) \!\!\!\!&= \nabla \times \mathcal{H},& \quad
          \nabla \cdot (\tilde{\varepsilon} \mathcal{E}) \!\!\!\!&= \rho_e, \\
\partial_t \mathcal{H} \!\!\!\! &= - \nabla \times \mathcal{E},& \quad \nabla \cdot \mathcal{H}\!\!\!\!& = 0,
\end{array} \right.
\end{equation*}
and denote
\begin{equation*}
\tilde{P} = 
\begin{pmatrix}
- \partial_t (\tilde{\varepsilon} \cdot) & \nabla \times \\
\nabla \times & \partial_t
\end{pmatrix}
, \quad v = (v_1,v_2): \R \times \R^3 \to \R^3 \times \R^3.
\end{equation*}
We can choose $\kappa, \beta,\sigma > 0$, $\gamma\in(0,\frac19)$, and $q<\infty$ such that $W^{\kappa,q}_{x'}\hookrightarrow L^\infty_{x'}$,
$\frac32+\frac89+\gamma+\sigma\le s_0$ and such that the proof of Theorem \ref{thm:StrichartzEstimatesFullyAnisotropic},
see \eqref{eq:EnergyEstimateI}, yields
\begin{equation}
\label{eq:StrichartzAuxiliaryQuasilinear}
\| \langle D' \rangle^{-\hat{s}} v \|_{L_T^{4+\beta} L_{x'}^q}  \lesssim_{T_0,\beta,C}
   \| v \|_{L_T^\infty L_{x'}^2} + \| \tilde{P} v \|_{L_T^2 L_{x'}^2} + \| \nabla \cdot (\tilde{\varepsilon} v_1) \|_{L_T^2 L_{x'}^2}
\end{equation}
for $\hat{s}+\kappa\ge  s_0-1$, $\alpha = \frac{8}{9} + \gamma$
and $C= \| \partial_x \tilde{\varepsilon} \|_{L_T^2 L_{x'}^\infty} + \| \tilde{\varepsilon} \|_{C_x^{\frac{8}{9}+\gamma}}$.
By Moser estimates, we have $\| \tilde{\varepsilon} \|_{C^\alpha_x} \lesssim_{\| \mathcal{E} \|_{L^\infty}} \| \mathcal{E} \|_{C^\alpha_x}$. To estimate the space-time H\"older regularity, write
\begin{equation*}
\| (\mathcal{E},\mathcal{H}) \|_{C_x^\alpha} \lesssim \|(\mathcal{E},\mathcal{H}) \|_{C^\alpha_tL^\infty_{x'}}
                                           + \| (\mathcal{E},\mathcal{H}) \|_{L_t^\infty C_{x'}^\alpha}.
\end{equation*}
 We have
\begin{equation*}
\| (\mathcal{E},\mathcal{H}) \|_{L_t^\infty C_{x'}^\alpha}
   \lesssim \| (\mathcal{E},\mathcal{H}) \|_{L_t^\infty H_{x'}^{\frac{3}{2}+\alpha+\sigma}}, \quad
\|(\mathcal{E},\mathcal{H}) \|_{C^\alpha_tL^\infty_{x'}}
\lesssim \| (\mathcal{E},\mathcal{H}) \|_{L_t^\infty H_{x'}^{\frac{3}{2}+\alpha+\sigma}}.
\end{equation*}
The first estimate is immediate from Sobolev embedding. The latter estimate follows for $\alpha \in \{ 0, 1 \}$ from Sobolev embedding and
the equation. Then we interpolate to obtain bounds in $C^\alpha L^\infty$  with $\alpha \in (0,1)$.

\smallskip

We pass to non-divergence form in the Maxwell  system setting
\begin{equation*}
\eta(x,D) = \begin{pmatrix} \tilde{\varepsilon}_1 \partial_1 & \tilde{\varepsilon}_2 \partial_2 & \tilde{\varepsilon}_3 \partial_3 \end{pmatrix} \quad\text{ and } \quad
\tilde{P}' =
\begin{pmatrix}
- \tilde{\varepsilon} \partial_t & \nabla \times \\
\nabla \times & \partial_t
\end{pmatrix}
.
\end{equation*}
By H\"older's inequality and distributing derivatives, \eqref{eq:StrichartzAuxiliaryQuasilinear} gives
\begin{equation*}
\| \langle D' \rangle^{-\tilde{s}} v \|_{L_T^4 L_{x'}^\infty}
  \lesssim  \| \langle D' \rangle^{-\hat{s}} v \|_{L_T^{4+\beta} L_{x'}^q}
  \lesssim \| v \|_{L_T^\infty L_{x'}^2} + \| \tilde{P}' v \|_{L_T^2 L_{x'}^2} + \| \eta(x,D) v_1 \|_{L_{T}^2 L_{x'}^2}.
\end{equation*}
with $\tilde{s}=\hat{s}+\kappa\ge s_0-1$.
We apply these inequalities to $v = \langle D' \rangle^{\tilde{s}+1} (\mathcal{E},\mathcal{H})$, obtaining
\begin{align}
\| \langle D' \rangle (\mathcal{E},\mathcal{H}) \|_{L_T^4 L_{x'}^\infty}
  &\lesssim \| \langle D' \rangle^{\tilde{s}+1} (\mathcal{E},\mathcal{H}) \|_{L_T^\infty L_{x'}^2} + \| \tilde{P}'(\langle D' \rangle^{\tilde{s}+1} (\mathcal{E},\mathcal{H})) \|_{L_T^2 L_{x'}^2}\notag\\
&\quad + \| \eta(x,D) (\langle D' \rangle^{\tilde{s}+1} \mathcal{E}) \|_{L_T^2 L_{x'}^2},\label{eq:strich-quasi}\\
\|\langle D'\rangle^{1+\kappa} (\mathcal{E},\mathcal{H}) \|_{L_T^{4+\beta} L_{x'}^q}
  &\lesssim \|\langle D' \rangle^{\tilde{s}+1} (\mathcal{E},\mathcal{H}) \|_{L_T^\infty L_{x'}^2} + \| \tilde{P}'(\langle D' \rangle^{\tilde{s}+1} (\mathcal{E},\mathcal{H})) \|_{L_T^2 L_{x'}^2}\notag\\
&\quad + \| \eta(x,D) (\langle D' \rangle^{\tilde{s}+1} \mathcal{E}) \|_{L_T^2 L_{x'}^2}.\label{eq:strich-quasi1}
\end{align}
To conclude the argument, we prove commutator estimates for the last two terms. Here we need to have operators in non-divergence form.
For solutions to \eqref{eq:QuasilinearMaxwell}, we note that $\tilde{P}'(\mathcal{E},\mathcal{H})$ and $\eta(x,D) \mathcal{E}$
vanish since $\partial_t (\varepsilon \mathcal{E}) = \tilde{\varepsilon} \partial_t \mathcal{E}$ and
$\nabla \cdot ( \varepsilon \mathcal{E}) = \eta(x,D) \mathcal{E}$.

\smallskip

We turn to the proof of the first commutator estimate and compute
\begin{equation*}
\| \tilde{P}'(\langle D' \rangle^{\tilde{s}+1} (\mathcal{E},\mathcal{H})) \|_{L_T^2 L_{x'}^2} \lesssim \| \langle D' \rangle^{\tilde{s}+1} \tilde{P}'(\mathcal{E},\mathcal{H}) \|_{L_T^2 L_{x'}^2} + \| [\tilde{\varepsilon},\langle D' \rangle^{\tilde{s}+1}] \partial_t \mathcal{E} \|_{L_T^2 L_{x'}^2}.
\end{equation*}
Here and below the implicit constants depend on $\| (\mathcal{E},\mathcal{H}) \|_{L^\infty_x}$.
The Kato-Ponce commutator estimate at fixed times and H\"older's inequality imply
\begin{align*}
\| [\tilde{\varepsilon}, \langle D' \rangle^{\tilde{s}+1} ] \partial_t \mathcal{E} \|_{L_T^2 L_{x'}^2} &\lesssim \| \langle D' \rangle^{\tilde{s}+1} \tilde{\varepsilon} \|_{L_T^\infty L_{x'}^2} \| \partial_t \mathcal{E} \|_{L_T^2 L_{x'}^\infty} \notag \\
&\quad+ \| \partial_x \tilde{\varepsilon} \|_{L_T^4 L_{x'}^\infty} \| \langle D' \rangle^{\tilde{s}} \partial_t \mathcal{E} \|_{L_T^4 L_{x'}^2}.
\end{align*}
By Moser estimates, we have
\begin{equation*}
\| \langle D' \rangle^{\tilde{s}+1} \tilde{\varepsilon} \|_{L_T^\infty L_{x'}^2} \lesssim \| \langle D' \rangle^{\tilde{s}+1} \mathcal{E} \|_{L_T^\infty L_{x'}^2}.
\end{equation*}
For the second term we rewrite the Maxwell equation as $\partial_t \mathcal{E} = \tilde{\varepsilon}^{-1} \nabla \times \mathcal{H}$
and estimate by the fractional Leibniz rule
\begin{align*}
 \| \langle D' \rangle^{\tilde{s}}& (\tilde{\varepsilon}^{-1} \nabla \times \mathcal{H}) \|_{L_T^4 L_{x'}^2} \\
 &\lesssim \| \langle D' \rangle^{\tilde{s}} (\tilde{\varepsilon}^{-1}) \|_{L_T^\infty L_{x'}^2} \| \nabla \times \mathcal{H} \|_{L_T^4 L_{x'}^\infty} + \| \tilde{\varepsilon}^{-1} \|_{L_x^\infty} \| \langle D' \rangle^{\tilde{s}} \nabla \times \mathcal{H} \|_{L_T^4 L_{x'}^2}.
\end{align*}
Using again the fractional Leibniz rule, Moser estimates and the Maxwell system, we derive (with $w=(\mathcal{E},\mathcal{H})$)
\begin{align} \label{eq:AuxAPrioriII}
 \| \tilde{P}'(\langle D' \rangle^{\tilde{s}+1}w) \|_{L_T^2 L_{x'}^2}
& \lesssim \| \langle D' \rangle^{\tilde{s}+1} w \|_{L_T^\infty L_{x'}^2} \| \langle D' \rangle w  \|_{L_T^2 L_{x'}^\infty}
      +  \| \langle D' \rangle w \|_{L^4_T L_{x'}^\infty} \notag \\
&\qquad  \cdot \big(\| \langle D' \rangle^{\tilde{s}} w\|_{L_T^\infty L_{x'}^2} \| \langle D' \rangle w  \|_{L^4_T L_{x'}^\infty}
        +\| \langle D' \rangle^{\tilde{s}+1} w  \|_{L_T^4 L_{x'}^2}\big).
\end{align}
 We turn to the last term in  \eqref{eq:strich-quasi}, which leads to
 \begin{equation*}
 \| \eta(x,D) (\langle D'\rangle^{\tilde{s}+1} \mathcal{E}) \|_{L_T^2 L_{x'}^2} \leq \| \langle D' \rangle^{\tilde{s}+1} (\eta(x,D) \mathcal{E}) \|_{L_T^2 L_{x'}^2} + \| [\tilde{\varepsilon}^i, \langle D' \rangle^{\tilde{s}+1}] \partial_i \mathcal{E} \|_{L_T^2 L_{x'}^2}.
 \end{equation*}
By the Kato-Ponce commutator estimate at fixed times and H\"older's inequality, we obtain
\begin{equation}
\label{eq:AuxAPrioriIII}
\begin{split}
\| [\tilde{\varepsilon}^i, \langle D' \rangle^{\tilde{s}+1}] \partial_i \mathcal{E} \|_{L_T^2 L_{x'}^2} &\lesssim \| \langle D' \rangle^{\tilde{s}+1} \tilde{\varepsilon} \|_{L_T^\infty L_{x'}^2} \| \partial_i \mathcal{E} \|_{L_T^2 L_{x'}^\infty} \\
&\qquad + \| \partial \tilde{\varepsilon}^i \|_{L_T^2 L_{x'}^\infty} \| \langle D' \rangle^{\tilde{s}} \partial_i \mathcal{E} \|_{L_T^\infty L_{x'}^2}.
\end{split}
\end{equation}
This estimate can be handled as in \eqref{eq:AuxAPrioriII}.

We turn to the continuity argument. Let $F(T) = \| \langle D' \rangle^{1+\beta} (\mathcal{E},\mathcal{H}) \|_{L_T^{4+\beta} L_{x'}^q}$ and $E^s(T) = \sup_{t \in [0,T]} \| u(t) \|_{H^s}$ with $s = \tilde{s}+ 1>2 + \frac{7}{18}$. Inequalities
\eqref{eq:strich-quasi}--\eqref{eq:AuxAPrioriIII} and \eqref{eq:EnergyEstimateAux} imply
\begin{equation*}
\left\{ \begin{array}{cl}
F(T)\! \!\!\!&\lesssim E^s(T) + T^{\frac{1}{4}} F(T) E^s(T) + T^{\frac{\beta}{4(4+\beta)}} F(T)^2 E^s(T) + T^{\frac{1}{4}} E^s(T), \\
E^s(T) \! \!\!\!&\lesssim e^{T^{\frac{1}{4}} F(T)} \| u(0) \|_{H^s}.
\end{array} \right.
\end{equation*}
At this point a continuity argument allows us to choose $T^* = T(\| u_0 \|_{H^s})$ such that $F(T^*) \lesssim E^s(T^*) \lesssim \| u (0) \|_{H^s}$. The proof is complete.
\end{proof}

We give the statement on the $L^2$-bound for differences of solutions.
\begin{proposition}
Let $u^i = (\mathcal{E}^i,\mathcal{H}^i)$, $i=1,2$ be solutions to \eqref{eq:QuasilinearMaxwell} under the assumptions of Theorem \ref{thm:QuasilinearAnisotropicMaxwellEquation} with finite $A$ and $B$, and set $v = u^1-u^2$. Then the estimate
\begin{equation*}
\| v(t) \|^2_{L^2} \lesssim e^{c(A) \int_0^t B(s') ds'} \| v(0) \|^2_{L^2}
\end{equation*}
with 
\begin{equation*}
A= \| u^1 \|_{L_x^\infty} + \| u^2 \|_{L_{x}^\infty} \text{ and } B(s) = \| \partial_{x'} u^1(s) \|_{L^\infty_{x'}} + \| \partial_{x'} u^2(s) \|_{L^\infty_{x'}}
\end{equation*}
holds true. Moreover, if $s>\frac{3}{2}+\frac{8}{9}$, there is a time $T=T(\| u^i(0) \|_{H^s})$ such that $T$ is lower semicontinuous and
\begin{equation*}
\sup_{t \in [0,T]} \| v(t) \|_{L^2} \lesssim_{\| u^i(0) \|_{H^s}} \| v(0) \|_{L^2}.
\end{equation*}
\end{proposition}

\noindent
The proof is an obvious modification of the proof of \cite[Proposition~6.2]{SchippaSchnaubelt2022}. 

Lastly, continuous dependence is proved using frequency envelopes (cf. \cite{SchippaSchnaubelt2022,IfrimTataru2020}). We refer to \cite[pp.~358ff]{SchippaSchnaubelt2022} for the details. This finishes the proof of Theorem \ref{thm:QuasilinearAnisotropicMaxwellEquation}. 


\section{Existence and regularity of eigenvectors for self-adjoint matrices}
\label{section:RegularityEigenvectors}
In this section we prove regularity results for eigenvectors of differentiably varying self-adjoint matrices.
\subsection{Local existence and regularity}
We start with local existence and regularity of eigenvectors of differentiably varying self-adjoint matrices provided that the eigenvalue is simple. This is well-known in the literature, but included for the sake of completeness.

\begin{lemma}
\label{lem:LocalSolutionsSimpleEigenvalue}
Let $k,n,l \in \N$, and $U \subseteq \R^k$ be open. Suppose that $A\in C^l(U;\R^{n \times n})$ and  that
$\lambda_z$ is a simple eigenvalue of $A(z)$  for some $z \in U$. Then, there is $V \subseteq \R^k$ open with $z \in V$ and $v \in C^l(V; \R^n)$, $\lambda \in C^l(V;\R)$ such that $\lambda(z) = \lambda_z$, $\| v \|_2 = 1$ and $A(x) v(x) = \lambda(x) v(x)$ for any $x \in V$.
\end{lemma}
\begin{proof}
The claim follows by applying the implicit function theorem (IFT) to
\begin{equation*}
F(x,v,\lambda) = 
\begin{pmatrix}
(A(x) - \lambda) v \\ \| v \|^2_2 -1
\end{pmatrix}
.
\end{equation*}
Note that by assumption there are exactly two vectors $w \in \R^n$ with $F(z,w,\lambda_z) = 0$. We compute
\begin{equation}
\label{eq:Derivative}
\partial_{(v,\lambda)} F(x,v,\lambda) = 
\begin{pmatrix}
A(x) - \lambda & 2v \\ 
2v^t & 0
\end{pmatrix}
.
\end{equation}
To apply the IFT, we prove for $(u,\lambda) \in \ker(\partial_{(v,\lambda)} F(z,w,\lambda_z))$ that $(u,\lambda) = 0$. By the second line of \eqref{eq:Derivative}, $\langle u,w \rangle = 0$. By multiplying $(A(z) - \lambda_z) u - \lambda w = 0$ with $(A(z) - \lambda_z)$, we find $(A(z) - \lambda_z)^2 u = 0$. This implies $u = \alpha w$, which is easy to see after changing to Jordan normal form. Hence, $u = 0$ as $u \perp w$. Thus, the IFT applies and yields eigenpairs as smooth as the matrix.
\end{proof}

%
%
 We do not know how to generalize the argument to eigenspaces of higher dimensions. Possibly, one can consider maps
\begin{equation*}
F: \R^k \times Gr(2,n) \to Gr(n), \quad F(x,E) = (A(x)- \lambda) E
\end{equation*}
to at least show regular dependence of the eigenspaces. 

%

\subsection{Globalizing non-degenerate eigenpairs}
It is unclear how to construct a mapping $\tilde{F}: U \to (\R^n)^k \in C^l$ from $F:U \to Gr(k,n) \in C^l$ in the general case. Indeed, we have the following counterexample:
\begin{example}
\label{ex:DomainNot1Connected}
We remark that globalization of the local solutions is not always possible as this example 
\begin{equation*}
A(\xi) = 
\begin{pmatrix}
\xi_1 & \xi_2 \\
\xi_2 & - \xi_1
\end{pmatrix}
\end{equation*}
for $\xi \in \mathbb{S}^1$ shows. For $\| \xi \|_2 = 1$, $A(\xi)$ has the eigenvalues $\pm 1$. We identify $\R^2 \equiv \C$, parametrize $\xi \in \mathbb{S}^1$ as $\xi = e^{i \varphi}$, denote $z \in \C$ and $A(\xi)$ as mapping $z \mapsto e^{i \varphi} \bar{z}$ with $e^{i \varphi}$. Then, denoting the eigenvector to $1$ by $z(\varphi)$, we have
\begin{equation*}
e^{i \varphi} \bar{z} = z(\varphi) \Rightarrow e^{i \varphi} = z^2(\varphi).
\end{equation*}
But the square root cannot be continuously extended to the unit disk.

\end{example}

In case of non-degenerate eigenvalues on simply connected domains, we can glue together the local solutions to obtain a global solution. We use the results of Rheinboldt \cite{Rheinboldt1969}, whose terminology is repeated here for convenience.

Let $X$, $Y$ be topological spaces with the Hausdorff property. For simplicity, we consider only relations $\phi \subseteq X \times Y$ with $D(\phi) = X$.
\begin{definition}
 A relation $\phi \subseteq X \times Y$ is said to be a \emph{local mapping relation}, or to have the \emph{local mapping property}, if for each $(x_0,y_0) \in \phi$ there exist (relatively) open neighbourhoods $U(x_0)$ of $x_0$ and $V(y_0)$ of $y_0$ such that the restriction $\varphi = \phi \cap (U(x_0) \times V(y_0))$ is a continuous mapping from $U(x_0)$ into $V(y_0)$.
\end{definition} 
  For $Q \subseteq X$, $\mathcal{P}(Q)$ denotes the set of all possible continuous paths $p:[0,1] \to Q$. $p,q \in \mathcal{P}(Q)$ are called equal, $p \equiv q$, if $p(\tau(t)) = q(t)$, $t \in [0,1]$, where $\tau:[0,1] \to [0,1]$, $\tau(0) = 0$, $\tau(1) = 1$ is a continuous, strictly monotone mapping.

In our context, we consider the local mapping relation relating $x \in \R^m$ with the normalized eigenvectors $\| v(x) \|_2 = 1$, of which there are exactly two if the eigenspace is one-dimensional. The local continuity follows from Lemma \ref{lem:LocalSolutionsSimpleEigenvalue}.

 To prove existence of global solutions derived from a local mapping relation, Rheinboldt introduced the continuation and path-lifting property.
 \begin{definition}
 A relation $\phi \subseteq X \times Y$ is said to have the continuation property for the subset $\mathcal{P}_X \subseteq \mathcal{P}(X)$ if for any $p \in \mathcal{P}_X$ and any continuous function $q:[0,\hat{t}) \subseteq J \to Y$ with $p(t) \phi q(t)$ for any $t \in [0,\hat{t})$ there exists a sequence $(t_k) \subseteq [0,\hat{t})$ with $\lim_{k \to \infty} t_k = \hat{t}$ such that $\lim_{k \to \infty} q(t_k) = \hat{y}$ and $p(\hat{t}) \phi \hat{y}$.
 \end{definition}
Observe that for Example \ref{ex:DomainNot1Connected}, the continuation property fails. We turn to the second definition:
\begin{definition}
A relation $\phi \subseteq X \times Y$ is said to have the path-lifting property for a set $\mathcal{P}_Y \subseteq \mathcal{P}(X)$ if for any $p \in \mathcal{P}_X$ and any $y_0 \in \phi[p(0)]$, there exists a path $q \in \mathcal{P}(Y)$ such that $q(0) = y_0$ and $p(t) \phi q(\tau(t))$. We call $q$ a lifting of $p$ through $y_0$ and write $p \phi q$.
\end{definition}
One of Rheinboldt's main results is the equivalence of the path-lifting and continuation property:
\begin{theorem}
\label{thm:EquivalencePathLiftingContinuation}
Let $\phi \subseteq X \times Y$ be a local mapping relation. Then $\phi$ has the path-lifting property for $\mathcal{P}_X \subseteq \mathcal{P}(X)$ if and only if $\phi$ has the continuation property for $\mathcal{P}_X$.
\end{theorem}
Next, we suppose that $X$ is $\mathcal{P}_X$-simply connected if it is path-connected under $\mathcal{P}_X$ and if any two paths of $\mathcal{P}_X$ with the same endpoints are $\mathcal{P}_X$-homotopic. In this case, we have the following global result.
\begin{theorem}
\label{thm:GlobalExistence}
Let $\phi \subseteq X \times Y$ be a local mapping relation with the path-lifting property for $\mathcal{P}_X \subseteq \mathcal{P}(X)$ and suppose that $X$ is $\mathcal{P}_X$-simply connected and locally $\mathcal{P}_X$-path-connected. Then, there exists a family of continuous mappings $F_\mu: X \to Y$ such that $x \phi F_\mu (x)$ and $\phi[x] = \big( F_\mu(x) \big)_{\mu \in M}$.
\end{theorem}
By considering local mapping relations for the eigenvectors, we would like to apply the above theorem to find global parametrizations of eigenvectors. It turns out that the following is sufficient to yield the path-lifting property:
\begin{definition}
A relation $\phi \subseteq X \times Y$ with $D(\phi) = X$ is called a \emph{covering relation} if for each $x \in X$ there exists an open neighbourhood $U(x)$ such that $\phi[U(x)] = \bigcup_{\mu \in M} V_\mu$, where the $V_\mu$ are disjoint open sets in $R(\phi)$ and for each $\mu \in M$ the restriction $\varphi_\mu = \phi \cap (U(x) \times V_\mu)$ is a continuous mapping from $U(x)$ into $V_\mu$. $U(x)$ is called an admissible neighbourhood of $x$. 
\end{definition}
The above certainly applies to our eigenvector mappings as the normalized eigenvectors are antipodal. Hence, the following theorem can be applied:
\begin{theorem}
\label{thm:CoveringRelationsArePathLifting}
A covering relation $\phi \subseteq X \times Y$ is a local mapping relation with the path-lifting property for $\mathcal{P}(X)$.
\end{theorem}

This guarantees the path-lifting property and by Theorem \ref{thm:EquivalencePathLiftingContinuation} the continuation property. We can prove the following global result on eigenvector parametrizations.
\begin{theorem}
\label{thm:GlobalEigenvectors}
Let $U \subseteq \R^k$ be simply connected and $A \in C^l(U;\R^{n \times n})$ such that for any $x \in U$, $A(x)$ is symmetric. Suppose that there is $\lambda \in C(U;\R)$ such that  $\lambda(x)$ is a simple eigenvalue of $A(x)$ for any $x \in U$. Then, there is $F \in C^l (U; \R^n)$ such that for any $x \in U$, $\| F(x) \|_2 = 1$ and $A(x) F(x) = \lambda(x) F(x)$.
\end{theorem}
\begin{proof}
Lemma \ref{lem:LocalSolutionsSimpleEigenvalue} yields $\lambda \in C^l(U;\R)$. Moreover, Lemma \ref{lem:LocalSolutionsSimpleEigenvalue} shows that $\phi \subseteq U \times \R^n$ with $x \phi w(x)$ for $w(x)$ the two normalized eigenvectors of $A(x)$ is a covering relation. Hence, Theorem \ref{thm:CoveringRelationsArePathLifting} yields the path-lifting property and by Theorem \ref{thm:EquivalencePathLiftingContinuation} the continuation property. Applying Theorem \ref{thm:GlobalExistence} finishes the proof.
\end{proof}

We remark that even in the case of higher dimensional eigenspaces, we have global parametrizations $\lambda \in C^{1}(U;\R)$ of the eigenvalues, e.g.,
by Theorem~1.1~(E) in \cite{Rainer2013} for $l \geq 2$. Theorem \ref{thm:GlobalEigenvectors} also yields that, if $A \in C^l(U;\R^{n \times n})$, where $U$ is simply connected, such that for any $x \in U$, $A(x)$ is symmetric with simple eigenvalues, we find the eigenvalues to be given by functions $\lambda_1,\ldots,\lambda_n \in C^l(U;\R)$ and corresponding eigenvectors given by $v_1,\ldots,v_n \in C^l(U;\R^n)$.

\section*{Acknowledgements}

Funded by the Deutsche Forschungsgemeinschaft (DFG, German Research Foundation)  Project-ID 258734477 – SFB 1173.

\end{document}